\documentclass[10pt]{amsart}
\usepackage{amsthm,amsmath,amscd,enumerate,mathtools,mathrsfs,amsxtra}
\usepackage[all]{xy}
\usepackage{amssymb}
\usepackage{enumerate}
\usepackage{hyperref}

\newtheorem{theorem}{Theorem}[section] 
\newtheorem{lemma}[theorem]{Lemma}
\newtheorem{corollary}[theorem]{Corollary}
\newtheorem{proposition}[theorem]{Proposition}

\newtheorem{fact}[theorem]{Fact}

\newtheorem{question}{Question}

\theoremstyle{definition}
\newtheorem{definition}[theorem]{Definition}

\newtheorem{remark}[theorem]{Remark}

\numberwithin{equation}{section}

\newcommand{\op}{\operatorname}
\newcommand{\R}{\mathbb{R}}
\newcommand{\C}{\mathbb{C}}

\newcommand{\N}{\mathbb{N}}

\newcommand{\RE}{\mathrm{Re}\,}
\newcommand{\IM}{\mathrm{Im}\,}
\newcommand{\Hom}{\operatorname{Hom}}
\newcommand{\uR}{\prescript{u}{}{\mathcal{R}}}
\newcommand{\uL}{\prescript{u}{}{\mathcal{L}}}
\newcommand{\ad}{\operatorname{ad}}
\newcommand{\Ad}{\operatorname{Ad}}

\begin{document}  

\title[Asymptotic support of Plancherel measures]{On the asymptotic support of Plancherel measures for homogeneous spaces}

\author{Benjamin Harris}
\address[B.Harris]{The MITRE Corporation. Approved for Public Release; Distribution Unlimited. Public Release Case Number 21-3406. Affiliation with the MITRE Corporation is for identification purposes only and is not intended to convey or imply MITRE's concurrent with, or support for, the positions, opinions, or viewpoints expressed by the author.}
\email{bharris@mitre.org}

\author{Yoshiki Oshima}
\address[Y.Oshima]{Department of Pure and Applied Mathematics, Graduate School of Information Science and Technology, Osaka University, 1-5 Yamadaoka, Suita, Osaka 565-0871, Japan}
\email{oshima@ist.osaka-u.ac.jp}

\subjclass[2010]{22E46}

\keywords{Plancherel Measure, Homogeneous Space, The Orbit Method, Coadjoint Orbit, Harmonic Analysis, Reductive Group, Discrete Series}

\begin{abstract}
Let $G$ be a real linear reductive group
 and let $H$ be a unimodular, locally algebraic subgroup.
Let $\op{supp} L^2(G/H)$ 
 be the set of irreducible unitary representations of $G$
 contributing to the decomposition of $L^2(G/H)$,
 namely the support of the Plancherel measure. 
In this paper, we will relate
 $\op{supp} L^2(G/H)$ with the image of moment map
 from the cotangent bundle $T^*(G/H)\to \mathfrak{g}^*$.

For the homogeneous space $X=G/H$, we attach a complex Levi subgroup $L_X$
 of the complexification of $G$ and we show that in some sense ``most" of
 representations in 
 $\op{supp} L^2(G/H)$ are obtained as quantizations of coadjoint orbits
 $\mathcal{O}$ such that $\mathcal{O}\simeq G/L$ and that the complexification
 of $L$ is conjugate to $L_X$.
Moreover, the union of such coadjoint orbits $\mathcal{O}$
 coincides asymptotically with the moment map image.
As a corollary, we show that $L^2(G/H)$ has a discrete series if
 the moment map image contains a nonempty subset of elliptic elements.
\end{abstract}

\maketitle


\section{Introduction}\label{sec:introduction}

Let $G$ be a connected, complex reductive group, let $\sigma$ be an antiholomorphic involution of $G$, and let 
\[(G^{\sigma})_e\subset G_{\mathbb{R}}\subset G^{\sigma}\]
be a real form of $G$. Let $H\subset G$ be a (Zariski) closed, complex algebraic subgroup, and let $X= G/H$ be the corresponding algebraic homogeneous space for $G$. Assume $H$ is $\sigma$ stable with real points $H_{\mathbb{R}}:=H^{\sigma}\cap G_{\mathbb{R}}\subset H$. Let $\mathfrak{g}$ (resp. $\mathfrak{g}_{\mathbb{R}}$, $\mathfrak{h}$, $\mathfrak{h}_{\mathbb{R}}$) denote the Lie algebra of $G$ (resp. $G_{\mathbb{R}}$, $H$, $H_{\mathbb{R}}$). Let $H_0\subset G_{\mathbb{R}}$ be a closed (not necessarily algebraic) subgroup for which the Lie algebra $\mathfrak{h}_0$ of $H_0$ is equal to the Lie algebra $\mathfrak{h}_{\mathbb{R}}$ of $H_{\mathbb{R}}$. In this case, we say that the corresponding homogeneous space $X_0:=G_{\mathbb{R}}/H_0$ is \emph{locally algebraic}.

Next, we assume that $X_0$ admits a nonzero, $G_{\mathbb{R}}$-invariant density $\nu$. Recall $G_{\mathbb{R}}$ acts continuously on the Hilbert space
\[L^2(X_0):=\left\{f\colon X_0\rightarrow \mathbb{C}\ \text{measurable}\,\Bigl |\, \int_{X_0} |f(x)|^2d\nu<\infty\right\}\]
and it preserves the unitary structure on $L^2(X_0)$. 
The theory of direct integrals yields a decomposition of $L^2(X_0)$ into
 irreducible unitary representations of $G_{\mathbb{R}}$. 
To be more precise, let $\widehat{G}_{\mathbb{R}}$ be the unitary dual of $G_{\mathbb{R}}$,
 that is, the set of all isomorphism classes of irreducible unitary representations,
 equipped with the Fell topology
 and the corresponding Borel structure.
Then there exist a finite Borel measure $m$ on $\widehat{G}_{\mathbb{R}}$
 and a measurable function $n(\cdot):\widehat{G}_{\mathbb{R}}\to \mathbb{Z}_{>0}\cup\{\infty\}$
 such that
\[
L^2(X_0)\simeq \int_{\widehat{G}_{\mathbb{R}}}^{\oplus} \pi^{\oplus n(\pi)} dm.
\]
The measure $m$ is unique up to equivalence because $G_{\mathbb{R}}$ is of type I.
See \cite[Paragraphe~VIII]{Di}, \cite[\S\,7.4]{Fo16}, \cite{Ma76}
 and \cite[Chapter~14]{Wa92} for this theory.
The \emph{support} of $L^2(X_0)$, denoted $\op{supp}L^2(X_0)$,
 is defined to be the support of the measure $m$.
Therefore, $\op{supp}L^2(X_0)\subset \widehat{G}_{\mathbb{R}}$
 is the smallest closed subset satisfying
 $m\bigl(\widehat{G}_{\mathbb{R}}\setminus \op{supp}L^2(X_0)\bigr)=0$.

The explicit form of the above decomposition of $L^2(X_0)$ is called the Plancherel formula.
It has been studied for a long time in several settings after the pioneering work of Gelfand.
Among them we note that:
\begin{itemize}
\item Harish-Chandra obtained the Plancherel formula for Riemannian symmetric spaces
 $X_0=G_{\mathbb{R}}/K_{\mathbb{R}}$
 and the group case $X_0=(G'_{\mathbb{R}}\times G'_{\mathbb{R}})/\Delta(G'_{\mathbb{R}})$.
\item The Plancherel formula for symmetric spaces was established by works of T.Oshima,
 Delorme~\cite{De98}, and van den Ban-Schlichtkrull~\cite{Bs05b}.
\item Delorme-Knop-Kr\"{o}tz-Schlichtkrull~\cite{DKKS} is a recent study toward
 the Plancherel formula for real spherical spaces.
\item When $H_0$ is an arithmetic subgroup,
 the study of irreducible decomposition of $L^2(X_0)$ 
 is a vast subject in connection with automorphic representations.
\end{itemize}
Our setting that $H_0$ is unimodular and locally algebraic
 include these settings.
The aim of this paper is to study the asymptotic behavior of $\op{supp}L^2(X_0)$.
As far as the authors know, this is the first result about the spectrum of $L^2(X_0)$
 in this generality.

We would also like to note two general results
 on the space of functions on $X_0$ when $H_0$
 has finitely many connected components.
Kobayashi-Oshima~\cite{KO13} proved that the finiteness of multiplicities on
 the space of functions on $X_0$
 (or more generally, induced representations)
 is characterized by the real sphericity.
Recently, Benoist-Kobayashi \cite{BK15, BKII, BKIII, BKIV}
 obtained a simple criterion for $L^2(X_0)$ to be a tempered representation.
A relationship between Benoist-Kobayashi's result and
 our theorem will be discussed at the end of introduction.

Our study is motivated by the orbit method \cite{Kir04}, \cite{Ver83}.
Let us briefly explain.
For a Lie group $G$, we write $\widehat{G}$ for the unitary dual of $G$, that is, the set of equivalence classes of the irreducible unitary representations of $G$.
When $G$ is a connected, simply connected nilpotent Lie group,
 Kirillov~\cite{Kir62} establishes a bijective correspondence between $\widehat{G}$ and
 the coadjoint orbits of $G$.
Moreover, characters, inductions, and restrictions of representations
 can be simply described in terms of the corresponding coadjoint orbit geometry.
For example, when $H$ is a connected closed subgroup of $G$,
 the following equivalence holds for $\pi\in\widehat{G}$:
\begin{equation}\label{orbit_method_induction}
\pi\in \op{supp}L^2(G/H) \Longleftrightarrow \mathcal{O}\subset
 \op{Im}(\mu\colon T^*(G/H)\to \mathfrak{g}^*),
\end{equation}
where $\mathcal{O}$ denotes the coadjoint orbit for $G$ corresponding to $\pi$
 and $\mu$ denotes the moment map.
See \cite{Kir04} for the details.

For a reductive Lie group $G_{\mathbb{R}}$, most irreducible, unitary representations
 naturally arise from coadjoint orbits.
However, some do not.
For instance, complementary series of $\op{SL}(2,\mathbb{R})$ are not naturally
 associated to coadjoint orbits.
Nevertheless, the set of coadjoint orbits is a good approximation of $\widehat{G}_{\mathbb{R}}$.
In particular, we can define an irreducible, unitary representation from a semisimple orbital parameter (see Definition~\ref{def:semisimple_orbital_parameter}).
Our main result Theorem~\ref{thm:main2} shows that the equivalence
 \eqref{orbit_method_induction} is ``asymptotically true" in our setting.

To be more precise, we need some notation and terminology.
For $\xi\in \mathfrak{g}^*$,
 let $G(\xi)$  denote the stabilizer subgroup of $\xi$ for the coadjoint action of $G$ 
 and let $\mathfrak{g}(\xi)$ denote its Lie algebra, namely,
\[
G(\xi)=\{g\in G\mid \op{Ad}^*(g)(\xi)=\xi\},\quad
\mathfrak{g}(\xi)=\{Y\in\mathfrak{g}\mid \op{ad}^*(Y)(\xi)=0\}.
\]
Similarly, for $\xi\in \mathfrak{g}_{\mathbb{R}}^*$ or
 $\xi\in \sqrt{-1}\mathfrak{g}_{\mathbb{R}}^*$,
 define
\[
G_{\mathbb{R}}(\xi)=\{g\in G_{\mathbb{R}}\mid \op{Ad}^*(g)(\xi)=\xi\},\quad
\mathfrak{g}_{\mathbb{R}}(\xi)=\{Y\in\mathfrak{g}_{\mathbb{R}}\mid \op{ad}^*(Y)(\xi)=0\}.
\]
When $\xi$ is semisimple, i.e.\ the coadjoint orbit through $\xi$ is closed, 
 $\mathfrak{g}(\xi)$ (resp.\ $\mathfrak{g}_{\mathbb{R}}(\xi)$)
 is called a Levi subalgebra of $\mathfrak{g}$ (resp.\ $\mathfrak{g}_{\mathbb{R}}$).
In the following, we often abbreviate the coadjoint action $\op{Ad}^*(g)(\xi)$
 to $g\cdot \xi$.

Let $\mathfrak{l}\subset \mathfrak{g}$ be a Levi subalgebra.
Write $Z(\mathfrak{l})$ for the center of $\mathfrak{l}$  and define 
\begin{equation*}
Z(\mathfrak{l})^*_{\text{reg}}
:=\left\{\lambda\in Z(\mathfrak{l})^*
 \mid \mathfrak{g}(\lambda)=\mathfrak{l}\right\},
\end{equation*}
namely, $Z(\mathfrak{l})^*_{\text{reg}}$
 is the set of
 $\mathbb{C}$-linear functionals on the center of $\mathfrak{l}$
 with (minimal possible) stabilizer $\mathfrak{l}$. 
Fix a Cartan subalgebra $\mathfrak{j}\subset \mathfrak{l}$. 
Let $\Delta(\mathfrak{g},\mathfrak{j})$ (resp.\ $\Delta(\mathfrak{l},\mathfrak{j})$)
 be the roots of $\mathfrak{g}$ with respect to $\mathfrak{j}$
 (resp.\ $\mathfrak{l}$ with respect to $\mathfrak{j}$),
 and let $\Delta^+(\mathfrak{l},\mathfrak{j})\subset \Delta(\mathfrak{l},\mathfrak{j})$
 be a choice of positive roots. 
We say
 $\lambda\in Z(\mathfrak{l})^*_{\text{reg}}$
 is in the \emph{good range} if 
\[\alpha\in \Delta(\mathfrak{g},\mathfrak{j})\ \&\ \langle \lambda, \alpha^{\vee}\rangle\in \mathbb{R}_{>0}\Longrightarrow \langle \lambda+\rho_{\mathfrak{l}},\alpha^{\vee}\rangle\in \mathbb{R}_{>0}.\]
This definition is independent of the choices of
 $\mathfrak{j}\subset \mathfrak{l}$
 and $\Delta^+(\mathfrak{l},\mathfrak{j})$. 
Denote by $Z(\mathfrak{l})^*_{\text{gr}}$ the collection of
 $Z(\mathfrak{l})^*_{\text{reg}}$ that lie in the good range. 
Suppose moreover that $\mathfrak{l}$ is $\sigma$-stable and
 let $\mathfrak{l}_{\mathbb{R}}:=\mathfrak{l}^\sigma$.
Let $\sqrt{-1}Z(\mathfrak{l}_{\mathbb{R}})^*$
 denote the set of purely imaginary valued linear
 functionals on the center of $\mathfrak{l}_{\mathbb{R}}$.
Then $\sqrt{-1}Z(\mathfrak{l}_{\mathbb{R}})^*$ is naturally viewed as a real form
 of $Z(\mathfrak{l})^*$.
Let
\begin{equation*}
\sqrt{-1}Z(\mathfrak{l}_{\mathbb{R}})^*_{\text{reg}}
:=Z(\mathfrak{l})^*_{\text{reg}}\cap \sqrt{-1}Z(\mathfrak{l}_{\mathbb{R}})^*,
\quad
\sqrt{-1}Z(\mathfrak{l}_{\mathbb{R}})^*_{\text{gr}}
:=Z(\mathfrak{l})^*_{\text{gr}}\cap \sqrt{-1}Z(\mathfrak{l}_{\mathbb{R}})^*.
\end{equation*}
Then $\sqrt{-1}Z(\mathfrak{l}_{\mathbb{R}})^*_{\text{reg}}$
 is a complement of a finite union of coroot subspaces
 with codimension one or two in $\sqrt{-1}Z(\mathfrak{l}_{\mathbb{R}})^*$. 

If $\lambda\in \mathcal{O}\subset \sqrt{-1}\mathfrak{g}_{\mathbb{R}}^*$
 is a point within a coadjoint orbit, then we define the {\it Duflo double cover}
 of $G_{\mathbb{R}}(\lambda)$ by 
 $\widetilde{G}_{\mathbb{R}}(\lambda)=
 G_{\mathbb{R}}(\lambda)\times_{\op{Sp}(T_{\lambda}\mathcal{O})} \op{Mp}(T_{\lambda}\mathcal{O})$.
See \cite[\S 2.1]{HO20} for a more detailed explanation about this double cover.

\begin{definition}\label{def:semisimple_orbital_parameter}
A \emph{semisimple orbital parameter} for $G_{\mathbb{R}}$ is a pair $(\mathcal{O},\Gamma)$ where 
\begin{enumerate}[(a)]
\item $\mathcal{O}\subset \sqrt{-1}\mathfrak{g}_{\mathbb{R}}^*$ is a semisimple (i.e.\ closed) coadjoint orbit
\item for every $\lambda\in \mathcal{O}$, $\Gamma_{\lambda}$ is a genuine one-dimensional unitary representation of $\widetilde{G_{\mathbb{R}}}(\lambda)$.
\end{enumerate}
In addition, this pair must satisfy
\begin{enumerate}[(i)]
\item $g\cdot \Gamma_{\lambda}=\Gamma_{g\cdot \lambda}$ 
for every $g\in G_{\mathbb{R}}$, $\lambda\in \mathcal{O}$
\item $d\Gamma_{\lambda}=\lambda|_{\mathfrak{g}_{\mathbb{R}}(\lambda)}$
for every $\lambda\in \mathcal{O}$.
\end{enumerate}
\end{definition}

Let $(\mathcal{O},\Gamma)$ be a semisimple orbital parameter. 
Take $\lambda\in \mathcal{O}$ and put $\mathfrak{l}:=\mathfrak{g}(\lambda)$.
Then we can regard $\lambda$ as an element of 
 $\sqrt{-1} Z(\mathfrak{l}_{\mathbb{R}})^*_{\text{reg}}$
 by restriction.
Assume  $\lambda\in \sqrt{-1} Z(\mathfrak{l}_{\mathbb{R}})^*_{\text{reg}}$
 is in the good range.  This assumption only depends on $\mathcal{O}$ and
 not on the choice of $\lambda\in\mathcal{O}$;
 hence, in this case, we say $\mathcal{O}$ is in the good range.
Then we can construct an irreducible unitary representation
 $\pi(\mathcal{O},\Gamma)$ by using cohomological induction.
 See Section~\ref{sec:quantization_semisimple} for the definition.
If we take $\lambda\in \mathcal{O}$ and put
 $\mathfrak{l}_{\mathbb{R}}:=\mathfrak{g}_{\mathbb{R}}(\lambda)$, 
 then we also write $\pi(\mathfrak{l}_{\mathbb{R}},\Gamma_\lambda)$
 for $\pi(\mathcal{O},\Gamma)$.

Let $\mu\colon T^*X\rightarrow \mathfrak{g}^*$ denote the moment map defined by
\[(x,\xi)\mapsto \xi\in T_x^*X\simeq (\mathfrak{g}/\mathfrak{g}_x)^*\hookrightarrow \mathfrak{g}^*.\]
The following theorem is a consequence of \cite[3.3 Corollary]{Kno94}.

\begin{theorem}[cf.\ \cite{Kno94}]\label{thm:moment_image}
Let $X$ be an algebraic homogeneous space for a connected, complex reductive group $G$ admitting a nonzero $G$-invariant density. Then there exists a complex Levi subalgebra $\mathfrak{l}_X\subset \mathfrak{g}$ and a complex subspace $\mathfrak{a}_X^*\subset Z(\mathfrak{l}_X)^*$ satisfying $Z_{\mathfrak{g}}(\mathfrak{a}_X^*)=\mathfrak{l}_X$, both unique up to $G$-conjugacy, such that
$\overline{\mu(T^*X)}=\overline{G\cdot \mathfrak{a}_X^*}$.
\end{theorem}

To state our main result we introduce some notation.

Let $Z$ be a finite dimensional real vector space and let $S\subset Z$ be a subset.
We define the \emph{asymptotic cone} of $S$ in $Z$ to be
\begin{equation*}
\op{AC}(S) 
:=\biggl\{\xi\in Z \;\biggl|\;
\begin{aligned} 
 &\text{ $S\cap C$ is unbounded for}\\
 &\text{ any open conic neighborhood $C$ of $\xi$ } 
\end{aligned}
\biggr\}\cup\{0\}.
\end{equation*}

If $\mathfrak{l}_1,\mathfrak{l}_2\subset \mathfrak{g}$ are subalgebras, we write $\mathfrak{l}_1\sim \mathfrak{l}_2$ if there exists $g\in G$ such that $\operatorname{Ad}(g)\mathfrak{l}_1=\mathfrak{l}_2$.

\begin{theorem}\label{thm:main}
Let $\mathfrak{l}_{\mathbb{R}}$ be a Lie subalgebra of $\mathfrak{g}_{\mathbb{R}}$
 such that
 $\mathfrak{l}_{\mathbb{R}}\otimes \mathbb{C}\sim \mathfrak{l}_X$.
Then
\begin{equation}
\begin{split}
&\phantom{=} \op{AC}\Bigl(\Bigl\{
\lambda\in \sqrt{-1}Z(\mathfrak{l}_{\mathbb{R}})^*_{\mathrm{gr}}
\;\Bigl|\; 
\begin{aligned} 
 &\pi(\mathfrak{l}_{\mathbb{R}},\Gamma_{\lambda})\in \op{supp}L^2(X_0)\\
 &\text{and }(G\cdot\lambda) \cap \mathfrak{a}_X^*\neq\emptyset 
\end{aligned}\Bigr\}\Bigr)
\cap
 \sqrt{-1}Z(\mathfrak{l}_{\mathbb{R}})_{\rm reg}^*\\  \label{eq:main}
&=\op{AC}\bigl(\bigl\{
\lambda\in \sqrt{-1}Z(\mathfrak{l}_{\mathbb{R}})^*_{\mathrm{gr}}
\mid\pi(\mathfrak{l}_{\mathbb{R}},\Gamma_{\lambda})\in \op{supp}L^2(X_0)\bigr\}\bigr)
\cap
 \sqrt{-1}Z(\mathfrak{l}_{\mathbb{R}})_{\rm reg}^*\\  
&=\overline{\sqrt{-1}\mu(T^*X_0)}\cap \sqrt{-1}Z(\mathfrak{l}_{\mathbb{R}})_{\mathrm{reg}}^*.
\end{split}
\end{equation}
Further, we have either
\begin{equation}\label{eq:moment}
\dim\bigl(\overline{\sqrt{-1}\mu(T^*X_0)}\cap
 \sqrt{-1}Z(\mathfrak{l}_{\mathbb{R}})_{\mathrm{reg}}^*\bigr)
 =\dim_{\mathbb{C}} \mathfrak{a}_X^*
\end{equation}
or 
\begin{equation*}
\overline{\sqrt{-1}\mu(T^*X_0)}\cap \sqrt{-1}Z(\mathfrak{l}_{\mathbb{R}})_{\mathrm{reg}}^*
=\emptyset.
\end{equation*}
The intersection is always nonempty for some $\mathfrak{l}_{\mathbb{R}}$.
\end{theorem}

Note that
 $\overline{\sqrt{-1}\mu(T^*X_0)}\cap\sqrt{-1}Z(\mathfrak{l}_{\mathbb{R}})_{\mathrm{reg}}^*$
 is a semialgebraic set and its dimension is well-defined.
Theorem~\ref{thm:main} says there exist $\mathfrak{l}_{\mathbb{R}}$
 and infinitely many representations
 of the form $\pi(\mathfrak{l}_{\mathbb{R}},\Gamma_{\lambda})$ in $\operatorname{supp}L^2(X_0)$.

Following the spirit of orbit method, we can restate
 Theorem~\ref{thm:main} as follows.
\begin{theorem}\label{thm:main2}
In the above setting, we have
\begin{align*}
&\phantom{=}\op{AC}\Biggl(
\bigcup_{\substack{\pi(\mathcal{O},\Gamma)\in \op{supp} L^2(X_0) \\
 (G\cdot \mathcal{O})\cap \mathfrak{a}_X^* \neq \emptyset}}
 \mathcal{O}\Biggr)\cap (G\cdot Z(\mathfrak{l}_X)^*_{\mathrm{reg}})\\
&=\op{AC}\Biggl(
\bigcup_{\substack{\pi(\mathcal{O},\Gamma)\in \op{supp} L^2(X_0)}}
 \mathcal{O}\Biggr)\cap (G\cdot Z(\mathfrak{l}_X)^*_{\mathrm{reg}})\\
&=\overline{\sqrt{-1}\mu(T^*X_0)}\cap (G\cdot Z(\mathfrak{l}_X)^*_{\mathrm{reg}}).
\end{align*}
Here, we assume
 $\mathcal{O}\subset G\cdot Z(\mathfrak{l}_X)^*_{\mathrm{reg}}$
 and $\mathcal{O}$ is in the good range for the first two lines of above equations.
\end{theorem}

We remark that 
$\overline{\sqrt{-1}\mu(T^*X_0)}\cap (G\cdot Z(\mathfrak{l}_X)^*_{\mathrm{reg}})$
 is an open dense subset of $\overline{\sqrt{-1}\mu(T^*X_0)}$
 by Theorem~\ref{thm:moment_image}.

The significance of Theorem~\ref{thm:main} and Theorem~\ref{thm:main2} is that
 in some sense ``most'' of the representations in $\operatorname{supp}L^2(X_0)$
 are of the form $\pi(\mathfrak{l}_{\mathbb{R}},\Gamma_{\lambda})$
 where $\mathfrak{l}_{\mathbb{R}}$ is a real form of $\mathfrak{l}$
 and $\mathfrak{l}\sim \mathfrak{l}_X$.
Next, we give a precise statement along these lines.

If $\mathfrak{l}\subset \mathfrak{g}$ is a Levi subalgebra,
 denote by $\widehat{G}_{\mathbb{R}}^{\mathfrak{l}}$
 the collection of irreducible, unitary representations of $G_{\mathbb{R}}$
 of the form $\pi(\mathfrak{l}_{\mathbb{R}}',\Gamma_{\lambda})$
 such that the complexification of $\mathfrak{l}'_{\mathbb{R}}$ is
 $G$-conjugate to $\mathfrak{l}$ and $\lambda$ is in the good range.

Let $\mathfrak{j}$ be a Cartan subalgebra of $\mathfrak{g}$
 and let $W=W(\mathfrak{g},\mathfrak{j})$ be the Weyl group.
An irreducible unitary representation $\pi$ of $G_{\mathbb{R}}$
 has an infinitesimal character, which is regarded as a $W$-orbit
 in $\mathfrak{j}^*$ via the Harish-Chandra isomorphism.
We write $\chi_{\pi}\in \mathfrak{j}^*/W$ for this.
By taking a conjugation, we may assume $\mathfrak{j}\subset \mathfrak{l}_X$
 and then we have inclusions
 $\mathfrak{a}_X^*\subset Z(\mathfrak{l}_X)^* \subset \mathfrak{j}^*$.
Write $\rho_{\mathfrak{l}_X}\in\mathfrak{j}^*$ for the half sum of positive roots in
 $\mathfrak{l}_X$.

The following theorem is essentially same as Theorem~\ref{thm:reduction_semisimple}.

\begin{theorem} \label{thm:inf_char} \ 
\begin{enumerate}[{\rm (i)}]
\item \label{inf.char._cond} 
 If $\pi\in \op{supp}L^2(G_{\mathbb{R}}/H_0)$,
 then $\chi_{\pi}$ has a representative
 $\xi\in \mathfrak{a}_X^*+\rho_{\mathfrak{l}_X}$,
 namely, $\chi_{\pi}\subset W\cdot (\mathfrak{a}_X^*+\rho_{\mathfrak{l}_X})$.
\item \label{singular_inf.char.}
 There exists a constant $d>0$
 which only depends on $G$ such that the following  holds:
 if $\pi\in \op{supp}L^2(G_{\mathbb{R}}/H_0)
 \setminus \widehat{G}_{\mathbb{R}}^{\mathfrak{l}_X}$,
 then there exist a representative $\xi\in \chi_{\pi}(\subset \mathfrak{j}^*)$ and
 a root
 $\alpha\in \Delta(\mathfrak{g},\mathfrak{j})\setminus
 \Delta(\mathfrak{l}_X,\mathfrak{j})$
 such that $\xi\in \mathfrak{a}_X^*+\rho_{\mathfrak{l}_X}$ and 
 $|\langle \xi,\alpha^{\vee}\rangle |<d$.
\end{enumerate}
\end{theorem}

The conclusion of \eqref{singular_inf.char.} means
 that the distance between $\xi$ and
 $Z(\mathfrak{l}_X)^* \setminus Z(\mathfrak{l}_X)^*_{\rm reg}$ is bounded by a constant.

As a corollary to Theorems~\ref{thm:main} and \ref{thm:inf_char},
 we obtain the following.
The proof is given in Section~\ref{sec:reduction_ss}.

\begin{corollary}\label{cor:singular_representation} \ 
\begin{enumerate}[{\rm (i)}]
\item \label{AC.regular}
The asymptotic cone 
\[\op{AC}\Bigl(
\bigcup_{\pi\in \op{supp}L^2(G_{\mathbb{R}}/H_0) \cap \widehat{G}_{\mathbb{R}}^{\mathfrak{l}_X}}
 \chi_{\pi}\Bigr)\]
 in $\mathfrak{j}^*$ contains a real semialgebraic variety
 with real dimension $\dim_{\mathbb{C}}\mathfrak{a}_X^*$. 
\item \label{AC.singular}
The asymptotic cone 
\[\op{AC}\Bigl(
\bigcup_{\pi\in \op{supp}L^2(G_{\mathbb{R}}/H_0) \setminus
 \widehat{G}_{\mathbb{R}}^{\mathfrak{l}_X}}
 \chi_{\pi}\Bigr)\]
 in $\mathfrak{j}^*$ is contained in a real algebraic variety
 with real dimension less than $\dim_{\mathbb{C}}\mathfrak{a}_X^*$. 
\end{enumerate}
\end{corollary}

Finally, we can show that, under certain additional assumptions, some of the subfamilies of representations of the form $\pi(\mathfrak{l}_{\mathbb{R}},\Gamma_{\lambda})$ occurring in $\operatorname{supp}L^2(X_0)$ must be discrete.
An element $\xi\in \mathfrak{g}_{\mathbb{R}}^*$ is said to be \emph{elliptic}
 if there exists a Cartan involution $\theta$
 such that $\theta(\xi)=\xi$. 
A coadjoint orbit $\mathcal{O}\subset \mathfrak{g}_{\mathbb{R}}^*$ is said to be \emph{elliptic}
 if one of (or equivalently, every) element in $\mathcal{O}$ is elliptic.
Let $(\mathfrak{g}_{\mathbb{R}})^*_{\text{ell}}\subset \mathfrak{g}_{\mathbb{R}}^*$
 denote the subset of all elliptic elements.

\begin{theorem}\label{thm:ds}
Assume $\mathfrak{g}\neq \mathfrak{h}$.
If $\mu(T^*X_0)\cap (\mathfrak{g}_{\mathbb{R}}^*)_{\mathrm{ell}}$
contains a nonempty open subset of $\mu(T^*X_0)$, then there exist infinitely many distinct irreducible, unitary representations $(\pi,V)$ such that
\[\Hom_{G_{\mathbb{R}}}(V,L^2(X_0))\neq \{0\}.\]
In particular, $X_0$ has a discrete series.
\end{theorem} 

We have $\mu(T^*X_0)=G_{\mathbb{R}}\cdot \mathfrak{h}_{\mathbb{R}}^{\perp}$,
 where $\mathfrak{h}_{\mathbb{R}}^{\perp}
 := (\mathfrak{g}_{\mathbb{R}}/\mathfrak{h}_{\mathbb{R}})^*\subset \mathfrak{g}_{\mathbb{R}}^*$.
Hence the condition of Theorem~\ref{thm:ds} is equivalent to that 
 $\mathfrak{h}_{\mathbb{R}}^{\perp}\cap (\mathfrak{g}_{\mathbb{R}}^*)_{\mathrm{ell}}$
 contains a nonempty open subset of $\mathfrak{h}_{\mathbb{R}}^{\perp}$.

\begin{remark}\label{rem:ds}
Here are some remarks about Theorem~\ref{thm:ds}.
\begin{enumerate}
\item 
It follows from the proof of Theorem~\ref{thm:ds}
 that if the condition of Theorem~\ref{thm:ds} holds, then we can find
 a $\theta$-stable parabolic subalgebra $\mathfrak{q}\subset \mathfrak{g}$ such that 
 $A_{\mathfrak{q}}(\lambda)$ occurs as a discrete spectrum in $L^2(X_0)$
 for infinitely many parameters $\lambda$ in the good range.
 We will find such $\mathfrak{q}$ explicitly
 for an example in \S \ref{subsec:Sp}. 
\item \label{item:sympair}
For symmetric spaces, 
 the existence of discrete series is equivalent to the rank condition
 $\op{rank} G/H=\op{rank} K/(H\cap K)$ by \cite{FJ80}, \cite{MO84}.
 This rank condition is equivalent to the condition in Theorem~\ref{thm:ds} for symmetric spaces.
\item \label{item:spherical}
When $X_0$ is a spherical space, Theorem~\ref{thm:ds} is proved in \cite[\S 12]{DKKS}.
\end{enumerate}
\end{remark}

For some $X$, the Levi subalgebra $\mathfrak{l}_X$ becomes a Cartan subalgebra.
In that case, Theorem~\ref{thm:main2} was proved as \cite[Theorem 1.1]{HW17}.
For a Cartan subalgebra $\mathfrak{j}$, the set $\widehat{G}_{\mathbb{R}}^{\mathfrak{j}}$
 consists of all tempered representations with regular infinitesimal characters.
If we take a closure of $\widehat{G}_{\mathbb{R}}^{\mathfrak{j}}$
 with respect to the Fell topology of $\widehat{G}_{\mathbb{R}}$,
 then we get the set of all tempered representations.

We remark that it may happen that 
 $\widehat{G}_{\mathbb{R}}^{\mathfrak{l}}\cap \widehat{G}_{\mathbb{R}}^{\mathfrak{l}'}
 \neq \emptyset$ even if $\mathfrak{l}$ and $\mathfrak{l}'$ are not conjugate.
When $G_{\mathbb{R}}$ is compact for example,  we have
 $\widehat{G}_{\mathbb{R}}^{\mathfrak{l}}\subset \widehat{G}_{\mathbb{R}}^{\mathfrak{l}'}$
 if $\mathfrak{l}\supset \mathfrak{l}'$ and
 $\widehat{G}_{\mathbb{R}}^{\mathfrak{j}}=\widehat{G}_{\mathbb{R}}$
 for a Cartan subalgebra $\mathfrak{j}$.

Our proof can be divided into two parts:
 the first part (\S \ref{sec:annihilator}, \S \ref{sec:reduction_ss})
 is algebraic  and the second part (\S \ref{sec:WF1}--\S \ref{sec:proof})
 is analytic.

In the first part, we prove Theorem~\ref{thm:inf_char}.
Thanks to the local structure theorem for complex algebraic homogeneous spaces,
 we show that a certain ideal $J_{\mathfrak{a}_X}$ of the enveloping algebra
 $\mathcal{U}(\mathfrak{g})$
 annihilates all functions on $G_{\mathbb{R}}/H_0$.
Hence for $\pi\in \op{supp}L^2(G_{\mathbb{R}}/H_0)$, the annihilator
 of $\pi$ contains $J_{\mathfrak{a}_X}$.
This information together with the unitarity of $\pi$ is enough to get the conclusion
 of Theorem~\ref{thm:inf_char}.
In the course of proof, we utilize the Beilinson-Bernstein localization
 and realize representations as the global sections of twisted $\mathscr{D}$-modules
 on partial flag varieties.

In the second part, the wave front set of representations plays a central role.
Our argument is partly similar to \cite{HHO16, Har18, HW17},
 but requires some new ingredients.
It was proved in \cite[Theorem 2.1]{HW17} that the wave front set of $L^2(G_{\mathbb{R}}/H_0)$
 equals the image of moment map.
By the first part of our proof, we can show that the contribution
 from $\op{supp}L^2(G_{\mathbb{R}}/H_0) \setminus \widehat{G}_{\mathbb{R}}^{\mathfrak{l}_X}$
 to the wave front set is small.
Then we have a relationship between 
 $\op{supp}L^2(G_{\mathbb{R}}/H_0) \cap \widehat{G}_{\mathbb{R}}^{\mathfrak{l}_X}$
 and the image of moment map.
To obtain Theorem~\ref{thm:main}, we need a calculation of the wave front set
 of a direct integral of representations in $\widehat{G}_{\mathbb{R}}^{\mathfrak{l}_X}$
 (Theorem~\ref{thm:direct_integral}).
\S \ref{sec:WF1}--\S \ref{sec:KKSform} is devoted to the proof
 of Theorem~\ref{thm:direct_integral}. 
For this, we use a formula for the distribution character of
 $\pi\in \widehat{G}_{\mathbb{R}}^{\mathfrak{l}}$ in \cite{HO20}.
This formula is a consequence of Schmid-Vilonen's formula \cite{SV98} which gives
 characters of representations in terms of characteristic cycles of
 sheaves on the flag variety.

In the end of introduction we would like to pose some questions concerning theorems above,
 for which the authors do not know the answer.
The first one is about the converse of Theorem~\ref{thm:ds}.

\begin{question}\label{ques:ds_converse}
Assume $H_0$ has only finitely many connected components 
 and $X_0$ has a discrete series. 
Then does $\mu(T^*X_0)\cap (\mathfrak{g}_{\mathbb{R}}^*)_{\mathrm{ell}}$
contain a nonempty open subset of $\mu(T^*X_0)$? 
\end{question}

When $H_0$ is a cocompact discrete subgroup of $G_{\mathbb{R}}$
 and if $G_{\mathbb{R}}$ does not have a discrete series,
 then the statement of Question~\ref{ques:ds_converse}
 does not hold.
Thus, we require the assumption that $H_0$ has finitely many connected components.

When $X_0$ is a symmetric space, 
Question~\ref{ques:ds_converse} is known to be true
 as mentioned in Remark~\ref{rem:ds} \eqref{item:sympair}.

The existence of discrete series for non-symmetric spaces
 was considered in \cite{Ko94, Ko98a}.
The results there are compatible with the statement of Question~\ref{ques:ds_converse}.
For (generalized) Stiefel manifolds, discrete series were studied in
  \cite{Ko92, Li93}.
For spherical spaces, recent results are in \cite[\S 13]{DKKS} and \cite{KKOS}.

To state the second question, we will enlarge the set
 of representations $\widehat{G}_{\mathbb{R}}^{\mathfrak{l}_X}$.
If we drop the condition that $\mathcal{O}$ is in the good range,
 $\pi(\mathcal{O},\Gamma)$ is still unitary, but it may be reducible or zero
 (see Remark~\ref{rem:fair}).
We include all irreducible components of such $\pi(\mathcal{O},\Gamma)$
 and also include limits for these representations with respect to the Fell topology.
Write $\widehat{G}_{\mathbb{R}, \mathrm{e}}^{\mathfrak{l}_X}$ for this
 enlarged set. 

\begin{question}\label{ques:regular}
When $H_0$ has only finitely many connected components, 
 do we have $\op{supp} L^2(X_0)\subset
 \widehat{G}_{\mathbb{R}, \mathrm{e}}^{\mathfrak{l}_X}$?
\end{question}

Again, Question~\ref{ques:regular} does not hold when $H_0$ is an infinite discrete group
 in general.

For symmetric spaces, Question~\ref{ques:regular} is true by the Plancherel formula.
Question~\ref{ques:regular} is also true when $H_0$ is algebraic and 
 $\mathfrak{l}_X$ is a Cartan subalgebra 
 because  in that case $L^2(X_0)$ is tempered
 and $\widehat{G}_{\mathbb{R}, \mathrm{e}}^{\mathfrak{l}_X}$ is the set of
 all irreducible tempered representations.
This follows from Benoist-Kobayashi's results
 \cite[Corollary 5.6 (i)]{BKII} and \cite[Theorem 1.1]{BKIV}.

\bigskip

\noindent
{\bf Acknowledgments}

The authors thank Professor Bernhard Kr\"{o}tz for discussions about the relationship between this paper and Kr\"{o}tz's work on spherical spaces.
They are grateful to Professor Toshiyuki Kobayashi for constant encouragement and kind explanations about his studies which inspires us.
B. Harris was supported by an AMS-Simons Travel Grant during the early part of this work. 
Y. Oshima was partially supported by JSPS KAKENHI Grant Number JP20K14325. 


\section{Quantization of semisimple coadjoint orbits}\label{sec:quantization_semisimple}

In this section we recall from \cite{Du82a}, \cite{Vog00} and \cite[\S 2]{HO20}
 the definition of representations which correspond to semisimple coadjoint orbits,
 or more precisely semisimple orbital parameters $(\mathcal{O},\Gamma)$.
We follow notation and terminology of \cite[\S 2]{HO20}.

Let $(\mathcal{O},\Gamma)$ be a semisimple orbital parameter in the sense of Definition~\ref{def:semisimple_orbital_parameter}.
Fix $\lambda\in \mathcal{O}$ and
 let $L_{\mathbb{R}}:=G_{\mathbb{R}}(\lambda)$ and
 $\mathfrak{l}_{\mathbb{R}}:=\mathfrak{g}_{\mathbb{R}}(\lambda)$.
The Duflo double cover of $L_{\mathbb{R}}$ is defined as
 $\widetilde{L_{\mathbb{R}}}:=
 L_{\mathbb{R}}\times_{\op{Sp}(T_{\lambda}\mathcal{O})} \op{Mp}(T_{\lambda}\mathcal{O})$.
Then 
\[\Gamma_{\lambda}\colon \widetilde{L_{\mathbb{R}}}\rightarrow \mathbb{C}^{\times}\]
is a unitary one-dimensional representation
 satisfying $d\Gamma_{\lambda}=\lambda$. 
Let $\mathfrak{j}_{\mathbb{R}}$ be a Cartan subalgebra of
 $\mathfrak{l}_{\mathbb{R}}$.
We can regard $\lambda\in \sqrt{-1}\mathfrak{j}_{\mathbb{R}}^*$
 by extending $\lambda$ by zero on
 $\mathfrak{j}_{\mathbb{R}}\cap [\mathfrak{l}_{\mathbb{R}},\mathfrak{l}_{\mathbb{R}}]$.

In order to define the representation $\pi(\mathcal{O},\Gamma)$ of $G_{\mathbb{R}}$
 we need to choose a complex parabolic subalgebra $\mathfrak{q}\subset \mathfrak{g}$
 with Levi factor $\mathfrak{l}=\mathfrak{g}(\lambda)$,
 which we call a \emph{polarization} for $\lambda$.
We say a polarization $\mathfrak{q}$ with nilradical $\mathfrak{n}$ is
 \emph{admissible} if 
\[
\langle \lambda,\alpha^\vee \rangle \in \mathbb{R}_{>0}
 \ \Longrightarrow\ 
\alpha \in \Delta(\mathfrak{n},\mathfrak{j}).
\]
Moreover, we say an admissible polarization $\mathfrak{q}$ is \emph{maximally real}
 if $\dim (\mathfrak{q}\cap \sigma(\mathfrak{q}))$ is maximal among
 all admissible polarizations for $\lambda$.

Fix a maximally real, admissible polarization $\mathfrak{q}\subset \mathfrak{g}$
 with nilradical $\mathfrak{n}$. In addition, fix a maximal compact subgroup $K_{\mathbb{R}}\subset G_{\mathbb{R}}$ with Cartan involution $\theta$
 such that $K_{\mathbb{R}}\cap L_{\mathbb{R}}\subset L_{\mathbb{R}}$ is maximal compact. 
We decompose $\lambda=\lambda_c+\lambda_n$
 where $\lambda_c\in (\sqrt{-1}Z(\mathfrak{l}_{\mathbb{R}})^*)^{\theta}$
 and $\lambda_n\in (\sqrt{-1}Z(\mathfrak{l}_{\mathbb{R}})^*)^{-\theta}$. 
Define $\Delta(\mathfrak{n}_{\mathfrak{p}},\mathfrak{j})$ to be the collection of roots
 $\alpha\in \Delta(\mathfrak{n},\mathfrak{j})$
 with $\langle \lambda_n,\alpha^{\vee}\rangle\neq 0$.
As in \cite[\S 2.2]{HO20}, one checks that 
\[\mathfrak{p}=\mathfrak{g}(\lambda_n)+ \mathfrak{n}_{\mathfrak{p}}\ 
 \text{where}\ \mathfrak{n}_{\mathfrak{p}}
 =\sum_{\alpha\in \Delta(\mathfrak{n}_{\mathfrak{p}},\mathfrak{j})}\mathfrak{g}_{\alpha}\]
is a $\sigma$-stable parabolic subalgebra of $\mathfrak{g}$
 with real form $\mathfrak{p}_{\mathbb{R}}$. 
Define $P_{\mathbb{R}}:=N_{G_{\mathbb{R}}}(\mathfrak{p}_{\mathbb{R}})$
 to be the corresponding parabolic subgroup,
 and let $P_{\mathbb{R}}=M_{\mathbb{R}}A_{\mathbb{R}}(N_P)_{\mathbb{R}}$
 be the Langlands decomposition of $P_{\mathbb{R}}$
 with $G_{\mathbb{R}}(\lambda_n)=M_{\mathbb{R}}A_{\mathbb{R}}$.

Following \cite[\S 2.2]{HO20},
 we define an elliptic coadjoint orbit
 $\mathcal{O}^{M_{\mathbb{R}}}:=M_{\mathbb{R}}\cdot\lambda_c$. 
Further, we obtain a genuine, one-dimensional,
 unitary representation $\Gamma^{M_{\mathbb{R}}}_{\lambda_c}$
 of $\widetilde{M_{\mathbb{R}}}(\lambda)$ from $\Gamma_{\lambda}$
 by the formula \cite[(2.13)]{HO20}. 
The coadjoint orbit $\mathcal{O}^{M_{\mathbb{R}}}$ and
 the one-dimensional representation $\Gamma^{M_{\mathbb{R}}}_{\lambda_c}$
 give rise to an elliptic orbital parameter
 $(\mathcal{O}^{M_{\mathbb{R}}},\Gamma^{M_{\mathbb{R}}})$ for $M_{\mathbb{R}}$. 

In \cite[\S 2.3 and \S 2.4]{HO20},
 we give a unitary representation
 $\pi(\mathcal{O}^{M_{\mathbb{R}}},\Gamma^{M_{\mathbb{R}}})$
 of $M_{\mathbb{R}}$ associated to $(\mathcal{O}^{M_{\mathbb{R}}},\Gamma^{M_{\mathbb{R}}})$. 
Then a unitary representation $\pi(\mathcal{O},\Gamma)$
 is defined by the normalized parabolic induction
\[\pi(\mathcal{O},\Gamma)
 :=\op{Ind}^{G_{\mathbb{R}}}_{P_{\mathbb{R}}}
 (\pi(\mathcal{O}^{M_{\mathbb{R}}},\Gamma^{M_{\mathbb{R}}})).\]
We also denote the same representation
 by $\pi(\mathfrak{l}_{\mathbb{R}},\Gamma_{\lambda})$.
This representation does not depend on the choices of
 $\lambda$, $\mathfrak{q}$ or $K_{\mathbb{R}}$.

\begin{remark}\label{rem:fair}
The construction of $\pi(\mathcal{O},\Gamma)$ here can be extended to the
 case where $\mathcal{O}$ is not necessarily in the good range.
The admissibility of the polarization implies the elliptic orbital parameter above
 is in the fair range in the sense of \cite{KV95}. 
In general, we still obtain unitary representations
 but they can be reducible or zero.
In this paper, we only consider $\pi(\mathcal{O},\Gamma)$ for parameters
 in the good range as it is enough for our purpose and it makes our treatment easier.
\end{remark}

In the above construction, 
 $\pi(\mathcal{O}^{M_{\mathbb{R}}},\Gamma^{M_{\mathbb{R}}})$
 can be defined as the cohomological
 induction for a $\theta$-stable parabolic subalgebra
 treated in \cite[Chapter V]{KV95}.
On the $(\mathfrak{g},K)$-module level,
 the induction $\op{Ind}^{G_{\mathbb{R}}}_{P_{\mathbb{R}}}$
 can be also defined in terms of cohomological induction
 for a $\sigma$-stable parabolic subalgebra
  as in \cite[Proposition 11.47]{KV95}.
Following \cite[(11.71)]{KV95},
 we define functors 
$(\uR_{\mathfrak{q},L_\R\cap K_\R}^{\mathfrak{g},K_\R})^j(\cdot)$ and 
$(\uL_{\mathfrak{q},L_\R\cap K_\R}^{\mathfrak{g},K_\R})^j(\cdot)$
 from the category of $(\mathfrak{l},L_\R\cap K_\R)$-modules
 to that of $(\mathfrak{g},K_\R)$-modules
 as 
\begin{align*}
&(\uR_{\mathfrak{q},L_\R\cap K_\R}^{\mathfrak{g},K_\R})^j(Z)
=(\Gamma_{\mathfrak{g},L_\R\cap K_\R}^{\mathfrak{g},K_\R})^j
 \bigl(\Hom_{\mathfrak{q}}(\mathcal{U}(\mathfrak{g}),Z)_{L_\R\cap K_\R}\bigr),\\
&(\uL_{\mathfrak{q},L_\R\cap K_\R}^{\mathfrak{g},K_\R})_j(Z)
=(\Pi_{\mathfrak{g},L_\R\cap K_\R}^{\mathfrak{g},K_\R})_j
 \bigl(\mathcal{U}(\mathfrak{g})\otimes_{\mathcal{U}(\mathfrak{q})} Z \bigr)
\end{align*}
for $j\in \N$.
Here, an $(\mathfrak{l},L_\R\cap K_\R)$-module
 $Z$ is regarded as a $(\mathfrak{q},L_\R\cap K_\R)$-module by the trivial
 $\mathfrak{n}$-action, 
 $(\Gamma_{\mathfrak{g},L_\R\cap K_\R}^{\mathfrak{g},K_\R})^j$
 is the $j$-th derived Zuckerman functor,
 and $(\Pi_{\mathfrak{g},L_\R\cap K_\R}^{\mathfrak{g},K_\R})_j$
 is its dual version.
Then by induction in stages, we have an isomorphism on the $(\mathfrak{g},K)$-module level
\[
\pi(\mathfrak{l}_{\mathbb{R}},\Gamma_{\lambda})
=\pi(\mathcal{O},\Gamma)\simeq 
(\uR_{\mathfrak{q},L_{\mathbb{R}}\cap K_{\mathbb{R}}}^{\mathfrak{g},K_{\mathbb{R}}})^s
(\Gamma_{\lambda}\otimes e^{\rho(\mathfrak{n})}),
\]
where $s=\dim_{\mathbb{C}}(\mathfrak{n}\cap \mathfrak{k})$
 and $e^{\rho(\mathfrak{n})}$ denotes the genuine character of $\widetilde{L}_{\mathbb{R}}$
 associated with the Lagrangian subspace
 $\mathfrak{n}\subset T_{\lambda}{\mathcal{O}}$ (see \cite[Chapitre I]{Du82a} for the definition).
In fact, by \cite[Theorem 5.99 and Proposition 11.52]{KV95},
\begin{align*}
&(\uR_{\mathfrak{q},L_{\mathbb{R}}\cap K_{\mathbb{R}}}^{\mathfrak{p},M_{\mathbb{R}}\cap K_{\mathbb{R}}})^j(\Gamma_{\lambda}\otimes e^{\rho(\mathfrak{n})})=0\ \text{ for $j\neq s$},\\
&(\uR_{\mathfrak{p},M_{\mathbb{R}}\cap K_{\mathbb{R}}}^{\mathfrak{g},K_{\mathbb{R}}})^j=0\ \text{ for $j\neq 0$},\\
&(\uR_{\mathfrak{q},L_{\mathbb{R}}\cap K_{\mathbb{R}}}^{\mathfrak{g}, K_{\mathbb{R}}})^j
 (\Gamma_{\lambda}\otimes e^{\rho(\mathfrak{n})})
\simeq 
\begin{cases} 
 (\uR_{\mathfrak{p},M_{\mathbb{R}}\cap K_{\mathbb{R}}}^{\mathfrak{g},K_{\mathbb{R}}})^0
 (\uR_{\mathfrak{q},L_{\mathbb{R}}\cap K_{\mathbb{R}}}^{\mathfrak{p},M_{\mathbb{R}}\cap K_{\mathbb{R}}})^s(\Gamma_{\lambda}\otimes e^{\rho(\mathfrak{n})})\ \text{ for $j=s$},\\
 0\ \text{ for $j\neq s$}.
\end{cases}
\end{align*}
Note that $\pi(\mathfrak{l}_{\mathbb{R}},\Gamma_{\lambda})$
 has infinitesimal character $\lambda+\rho_{\mathfrak{l}}$, where 
we choose positive roots $\Delta^+(\mathfrak{l},\mathfrak{j})\subset \Delta(\mathfrak{l},\mathfrak{j})$
 and write $\rho_{\mathfrak{l}}=\frac{1}{2}\sum_{\Delta^+(\mathfrak{l},\mathfrak{j})}\alpha$.
By \cite[Theorem 5.99 and Proposition~11.65]{KV95}, 
 $\pi(\mathfrak{l}_{\mathbb{R}},\Gamma_{\lambda})$ can be also constructed by the functor
 $\uL$:
\begin{align*}
&(\uR_{\mathfrak{q},L_\R\cap K_\R}^{\mathfrak{g}, K_\R})^j
 (\Gamma_{\lambda}\otimes e^{\rho(\mathfrak{n})})
\simeq 
(\uL_{\sigma(\mathfrak{q}), L_\R\cap K_\R}^{\mathfrak{g}, K_\R})_j
 (\Gamma_{\lambda}\otimes e^{\rho(\theta(\mathfrak{n}))}).
\end{align*}
Here, $e^{\rho(\theta(\mathfrak{n}))}$ is 
 the character defined in \cite[Chapitre I]{Du82a}
 associated with the Lagrangian subspace
 $\theta(\mathfrak{n})\subset T_{\lambda}{\mathcal{O}}$.

Following \cite[Appendix A]{HO20} (cf.\ also \cite[Theorem 2.2.3]{Ma04}),
 we define a virtual $(\mathfrak{g},K_\R)$-module
\[\pi(\mathcal{O},\Gamma, \mathfrak{q})
 :=\sum_j (-1)^j
 (\uR_{\mathfrak{q},L_{\mathbb{R}}\cap K_{\mathbb{R}}}^{\mathfrak{g},K_{\mathbb{R}}})^{s+j}
 (\Gamma_{\lambda}\otimes e^{\rho(\mathfrak{n})})\]
 for any polarization $\mathfrak{q}$.
Note that the functor $\uR$ here is denoted by $I$ in \cite{HO20}.
Then \cite[Theorem A.1]{HO20} says 
 $\pi(\mathcal{O},\Gamma, \mathfrak{q})$ does not depend on
  the choice of polarization $\mathfrak{q}$
 as long as $\mathfrak{q}$ is admissible.
In the same way, we can prove that a virtual module 
\[\pi'(\mathcal{O},\Gamma, \mathfrak{q})
 :=\sum_j (-1)^j
 (\uL_{\sigma(\mathfrak{q}),L_{\mathbb{R}}\cap K_{\mathbb{R}}}^{\mathfrak{g},K_{\mathbb{R}}})_{s-j}(\Gamma_{\lambda}\otimes e^{\rho(\theta(\mathfrak{n}))})\]
 does not depend on the choice of admissible polarization $\mathfrak{q}$.
Since $\pi(\mathcal{O},\Gamma, \mathfrak{q})=\pi'(\mathcal{O},\Gamma, \mathfrak{q})$
 for a maximally real admissible polarization $\mathfrak{q}$,
 the same is true for any admissible polarization,
 namely we have
\[
\pi(\mathcal{O},\Gamma, \mathfrak{q})=\pi'(\mathcal{O},\Gamma, \mathfrak{q})
=\pi(\mathcal{O},\Gamma)
\]
as a virtual $(\mathfrak{g},K_{\mathbb{R}})$-module
 for any admissible polarization $\mathfrak{q}$.

By the Beilinson-Bernstein localization,
 this representation can be also realized as global sections on the flag variety.
For an admissible polarization $\mathfrak{q}$,
 let $Q$ be the parabolic subgroup of $G$ with Lie algebra $\mathfrak{q}$,
 let $Y:=G/\sigma(Q)$ be the partial flag variety, the collection
 of all parabolic subgroups which are conjugate to $\sigma(Q)$
 and let $S=K/(\sigma(Q)\cap K)$ be the $K$-orbit through the base point in $Y$.
Let $\mathscr{D}_{Y,\lambda}$ be the sheaf of rings of twisted differential operators
 on $Y$ corresponding to the parameter $\lambda$ (see e.g.\ \cite{Bie90}).
Then we have a spectral sequence of $(\mathfrak{g},K_\R)$-modules
 (see e.g.\ \cite[Theorem 5.4]{Kit12}, \cite[(6.3)]{Osh13})
\[
H^{p}(Y,R^qi_+\mathcal{L})\Rightarrow
(\uL_{\sigma(\mathfrak{q}),L_{\mathbb{R}}\cap K_{\mathbb{R}}}^{\mathfrak{g},K_{\mathbb{R}}})_{s-p-q}
(\Gamma_{\lambda}\otimes e^{\rho(\theta(\mathfrak{n}))}).
\]
Here, $i\colon S\to Y$ is the natural immersion.
$\mathcal{L}$ is the $K$-equivariant line bundle (i.e.\ invertible $\mathcal{O}$-module)
  on $S$ given by $K\times_{(\sigma(Q)\cap K)} \tau$
 for an algebraic character $\tau$ of $\sigma(Q)\cap K$
 whose restriction to $L_{\mathbb{R}}\cap K_{\mathbb{R}}$ is 
\[
\Gamma_{\lambda}\otimes e^{-\rho(\theta(\mathfrak{n}))}
\otimes \bigwedge^{\rm top} (\mathfrak{k}/(\mathfrak{l}\cap\mathfrak{k})).
\]
Then $\mathcal{L}$ can be viewed as a twisted $\mathscr{D}$-module on $S$
 and its (higher) direct images $R^qi_+\mathcal{L}$
 are defined as $\mathscr{D}_{Y,\lambda}$-modules.
Our assumption on $\lambda$ implies $Y$ is $\mathscr{D}_{Y,\lambda}$-affine
 so that $H^{p}(Y,R^qi_+\mathcal{L})=0$ for $p>0$.
Hence the above spectral sequence collapses and we have 
\[
\Gamma(Y,R^qi_+\mathcal{L})\simeq
(\uL_{\sigma(\mathfrak{q}),L_{\mathbb{R}}\cap K_{\mathbb{R}}}^{\mathfrak{g},K_{\mathbb{R}}})_{s-q}
(\Gamma_{\lambda}\otimes e^{\rho(\theta(\mathfrak{n}))}).
\]
We therefore have
\begin{equation}\label{eq:localization}
\sum_{q}(-1)^q \Gamma(Y,R^qi_+\mathcal{L})=\pi(\mathcal{O},\Gamma).
\end{equation}

We end this section by giving the Langlands parameter of $\pi(\mathcal{O},\Gamma)$
 when $\mathcal{O}$ is in the good range.
In order to do this, we need to write a one-dimensional representation of $L_{\mathbb{R}}$
 as a quotient of standard module.
Let $J_{\mathbb{R}}$ be the maximally noncompact Cartan subalgebra of $L_{\mathbb{R}}$
 and let $J_{\mathbb{R}}=T_{\mathbb{R}}A^1_{\mathbb{R}}$ be its Cartan decomposition
 with respect to $\theta$,
 namely, $T_{\mathbb{R}}=J_{\mathbb{R}}^{\theta}$
 and $A^1_{\mathbb{R}}$ is the connected subgroup of $L_\R$ with Lie algebra
 $\mathfrak{a}^1_{\mathbb{R}}=\mathfrak{j}_{\mathbb{R}}^{-\theta}$.
Take a Borel subalgebra $\mathfrak{b}_{\mathfrak{l}}$ of $\mathfrak{l}$ such that
 $\mathfrak{b}_{\mathfrak{l}}\supset \mathfrak{j}$ and 
 $\mathfrak{b}_{\mathfrak{l}}+(\mathfrak{l}\cap\mathfrak{k})=\mathfrak{l}$.
Write $\mathfrak{n}_{\mathfrak{l}}$ for the nilradical of $\mathfrak{b}_{\mathfrak{l}}$.
Define a character $e^{2\rho(\mathfrak{n}_{\mathfrak{l}})'}$ of $J_{\mathbb{R}}$ by 
\[e^{2\rho(\mathfrak{n}_{\mathfrak{l}})'}(ta)=
 \det(\op{Ad}(t)|_{\mathfrak{n}_{\mathfrak{l}}\cap\mathfrak{k}})
 \cdot \det(\op{Ad}(a)|_{\mathfrak{n}_{\mathfrak{l}}})\]
for $t\in T_{\mathbb{R}}$ and $a\in A^1_{\mathbb{R}}$,
 which is the same as the character $\C_{2\rho(\mathfrak{n}_{\mathfrak{l}})'}$
 defined in \cite[(11.111)]{KV95}.
The differential of $e^{2\rho(\mathfrak{n}_{\mathfrak{l}})'}$
 equals $2\rho(\mathfrak{n}_{\mathfrak{l}})$, but it may not be equal to 
 $\det(\op{Ad}(ta)|_{\mathfrak{n}_{\mathfrak{l}}})$ when $T_{\R}$ is disconnected.
The trivial representation of $L_{\mathbb{R}}$ is the irreducible quotient
 of the standard module
 $(I^{\mathfrak{l},L_{\mathbb{R}}\cap K_{\mathbb{R}}}_{\mathfrak{b}_{\mathfrak{l}},T_{\mathbb{R}}})^{s_L}
 (e^{2\rho(\mathfrak{n}_{\mathfrak{l}})'})$,
where $s_L:=\dim_{\mathbb{C}}(\mathfrak{n}_{\mathfrak{l}}\cap\mathfrak{k})$.
By induction in stages, it turns out that $\pi(\mathcal{O},\Gamma)$
 is the unique irreducible quotient of 
 the standard module 
\[(\uR^{\mathfrak{g},K_{\mathbb{R}}}_{\mathfrak{b}_{\mathfrak{l}}+\mathfrak{n},T_{\mathbb{R}}})^{s+s_L}
 (\Gamma_{\lambda}\otimes e^{\rho(\mathfrak{n})}\otimes e^{2\rho(\mathfrak{n}_{\mathfrak{l}})'}).\]

In the notation of \cite{ALTV} (cf.\ also \cite[\S XI.9]{KV95}),
 the irreducible admissible representations of $G_{\mathbb{R}}$ are parametrized by
 data $(J_{\mathbb{R}},\gamma,\Delta_{i\mathbb{R}}^+)$, where $J_{\mathbb{R}}\subset G_{\mathbb{R}}$ is a Cartan subgroup with Lie algebra $\mathfrak{j}_{\mathbb{R}}$, $\gamma$ is a level one character of the $\rho_{\text{abs}}$ double cover of $J_{\mathbb{R}}$ (see Section 5 of \cite{ALTV} for an explanation), and $\Delta_{i\mathbb{R}}^+$ is a choice of positive roots among the set of imaginary roots for $\mathfrak{j}_{\mathbb{R}}$ in $\mathfrak{g}_{\mathbb{R}}$ for which $d\gamma\in \mathfrak{j}^*$ is weakly dominant. This triple must satisfy a couple of other technical assumptions
 (see Theorem 6.1 of \cite{ALTV}). 
The above argument shows that 
 the irreducible representation $\pi(\mathcal{O},\Gamma)$ corresponds to
 the parameter $(J_{\mathbb{R}}, \gamma, \Delta_{i\mathbb{R}}^+)$,
 where $\gamma$ is the character of $\rho_{\rm abs}$-cover of $J_{\mathbb{R}}$
 such that
\[\gamma\otimes \rho_{\rm abs} \simeq \Gamma_{\lambda}\otimes e^{\rho(\mathfrak{n})}\otimes e^{2\rho(\mathfrak{n}_{\mathfrak{l}})'}.\]
 $\Delta_{i\mathbb{R}}^+$ and $\rho_{\rm abs}$ are
 defined by the positive system for the Borel subalgebra
 $\mathfrak{b}_{\mathfrak{l}}+ \mathfrak{n}$.


\section{Annihilator ideas of induced representations}\label{sec:annihilator}

In this section we will study annihilator ideals of
 irreducible subrepresentations of $C^{\infty}(G_{\R}/H_0)$.

First, we need the following fact on algebraic subgroups.
See \cite[Theorems 4 and 8]{BBHM63}.

\begin{fact}\label{Fact:observable}
Let $G$ be a complex algebraic group and $H$ an algebraic subgroup.
The following three conditions are equivalent.
\begin{enumerate}
\item $G/H$ is quasi-affine.
\item \label{cond:extend}
 Every finite-dimensional rational $H$-module is a $H$-submodule
 of a finite-dimensional rational $G$-module.
\item \label{cond:stabilizer}
 There exists a vector $w$ in a rational $G$-module
 such that $H$ is the stabilizer subgroup of $w$.
\end{enumerate}
\end{fact}

When one (or all) of the conditions in Fact~\ref{Fact:observable}
 is satisfied, $H$ is said to be \emph{observable} in $G$.

Let $G$ be a connected, complex reductive group with real form 
$(G^{\sigma})_e\subset G_{\mathbb{R}}\subset G^{\sigma}$
for an antiholomorphic involution $\sigma$ of $G$. 
Suppose that a connected, complex algebraic subgroup $H$ of $G$
 is defined over $\mathbb{R}$, namely $\sigma(H)=H$.
Write $\mathfrak{h}_\R=\mathfrak{h}^{\sigma}$
 for the real form of $\mathfrak{h}$.
Let $H_0\subset G_{\mathbb{R}}$ be a closed subgroup whose Lie algebra
 $\mathfrak{h}_0$ is equal to the Lie algebra $\mathfrak{h}_{\mathbb{R}}$.
Here, the closedness of $H_0$ in $G_{\mathbb{R}}$ is considered in the classical topology and 
$H_0$ is not necessarily algebraic.
In particular, we allow $H_0$ to have infinitely many connected components.

\begin{lemma}\label{lem:unimodular_observable} 
If $H_0\subset G_{\mathbb{R}}$ is a unimodular subgroup, then $H$ is an observable subgroup of $G$.
\end{lemma}

\begin{proof} 
Let $d:=\dim \mathfrak{h}$.
If $H_0\subset G_{\mathbb{R}}$ is a unimodular subgroup of $G_{\mathbb{R}}$, then the identity component $(H_0)_e$ of $H_0$ acts trivially on $\bigwedge^{d}\mathfrak{h}_0$. Since $\mathfrak{h}_0=\mathfrak{h}_{\mathbb{R}}$, the complexification $\mathfrak{h}$ annihilates $\bigwedge^{d}\mathfrak{h}$.
This implies that $H\subset G$ is a unimodular subgroup.

Let  $W:=\bigwedge^{d} \mathfrak{g}$
 with the $G$-action $\bigwedge^{d} \Ad$.
Take a nonzero vector $w$ in $\bigwedge^{d} \mathfrak{h}\subset \bigwedge^{d} \mathfrak{g}$.
Define $S$ to be the stabilizer subgroup of $w$ in $G$. 
By definition of $S$ and Fact~\ref{Fact:observable} \eqref{cond:stabilizer}, 
 $S$ is observable in $G$.
Since $H$ is unimodular, $H\subset S$.
Moreover, $S$ normalizes $H$ and hence $H$ is observable in $S$ by \cite[Theorem 2]{BBHM63}.
The transitivity of the condition \eqref{cond:extend} in Fact~\ref{Fact:observable} implies
 that $H$ is observable in $G$.
\end{proof}

In the following we assume that $H_0$ is unimodular.

We now use the local structure theorem for $X=G/H$ (see \cite[Theorem 2.3, Proposition 2.4, Lemma 3.1]{Kno94}).  
The theorem states that there exist a parabolic subgroup $Q_X$ of $G$ with Levi factor $L_X$
 and an $L_X$-stable subvariety $Z\subset X$ such that
\begin{itemize}
\item  the natural map $Q_X\times^{L_X} Z\to X$
 is an open immersion, and 
\item if $L_X^0$ denotes the kernel of $L_X\to\op{Aut}(Z)$, then $L_X^0$
 contains a commutator subgroup $[L_X,L_X]$.
\end{itemize}
Let $A_X=L_X/L_X^0$ with Lie algebra $\mathfrak{a}_X$, which is a torus. 
It follows from the proof of \cite[Theorem 2.3, Proposition 2.4, Lemma 3.1]{Kno94}
 that $\mathfrak{a}_X^*$ intersects $Z(\mathfrak{l}_X)^*_{\mathrm{reg}}$. 
Hence $\mathfrak{l}_X=\{Y\in\mathfrak{g}\mid \ad^*(Y)(\mathfrak{a}_X^*)=0\}$.

Next, fix a Cartan subgroup $J\subset L_X$
 and a Borel subgroup $B$ of $G$ such that $J\subset B\subset Q_X$.
Note that there are natural inclusions
 $\mathfrak{a}_X^*\subset
 (\mathfrak{l}_X/[\mathfrak{l}_X,\mathfrak{l}_X])^*=Z(\mathfrak{l}_X)^*
 \subset \mathfrak{j}^*$.
Fix a positive system $\Delta^+(\mathfrak{g},\mathfrak{j})$
 as the roots for $B$, and let $F_{\lambda}$ denote the irreducible, finite-dimensional representation of $G$ with highest weight $\lambda\in \mathfrak{j}^*$. Let $R(G/H)$ denote the space of regular functions on $G/H$.

\begin{lemma} \label{lem:orbit_rest} If $\lambda\in \mathfrak{j}^*$
 is a dominant integral weight and $F_{\lambda}$ occurs in
 the irreducible decomposition of $R(G/H)$,
 then $\lambda\in \mathfrak{a}_X^*$. 
\end{lemma}

\begin{proof}
Suppose $F_{\lambda}\subset R(G/H)$.
If $f\in F_{\lambda}\subset R(G/H)$ is a highest weight vector,
 then $f(b^{-1}x)=b^\lambda f(x)$ for $b\in B$, $x\in X$.
Observe that $Q_X\times^{L_X} Z\simeq B\times^{B\cap L_X} Z$, which can be
 identified with an open subvariety of $X$.
Therefore, $f|_Z\not\equiv 0$.  Since $J\cap L_X^0$ acts trivially on $Z$,
 $\lambda=0$ on $\mathfrak{j}\cap \mathfrak{l}_X^0$, namely, $\lambda\in \mathfrak{a}_X^*$.
\end{proof}

Differentiating the action of $G_{\mathbb{R}}$ on $G_{\mathbb{R}}/H_0$ and the action of $G$ on $G/H$ we obtain maps
\[\mathcal{U}(\mathfrak{g})\stackrel{\Phi_0}\longrightarrow \op{Diff}(G_{\mathbb{R}}/H_0),\quad
 \mathcal{U}(\mathfrak{g})\stackrel{\Phi}\longrightarrow \op{Diff}(G/H)\]
of the universal enveloping algebra into the algebras of differential operators.
Here, $\op{Diff}(G_{\mathbb{R}}/H_0)$ (resp.\ $\op{Diff}(G/H)$)
 denotes the algebra of $\C$-valued real analytic differential operators on $G_{\mathbb{R}}/H_0$
 (resp.\ complex algebraic differential operators on $G/H$).
Since the complexificiation of $\mathfrak{h}_0$ is $\mathfrak{h}$, 
 the map $G_{\R}/H_0 \ni g H_0\mapsto gH \in G/H$
 is \emph{locally} well-defined and
 the image of this map is a totally real submanifold of $G/H$. 
The differential operators in $\op{Im}\Phi$ can be viewed
 as holomorphic differential operators on the connected complex manifold $G/H$.
Hence such operators are zero if and only if 
 their restrictions to a totally real submanifold are zero. 
This implies $\op{Ker}\Phi=\op{Ker}\Phi_{0}$.

Finally, we have the composition
\[\mathcal{U}(\mathfrak{g})\stackrel{\Phi}\rightarrow \op{Diff}(G/H)\stackrel{\psi}\rightarrow \op{End}R(G/H).\]
Recall that $H\subset G$ observable means that $G/H$ is quasi-affine, i.e.\ 
 $G/H$ is isomorphic to an open subset of an affine variety. Since no nonzero differential operator on an affine variety annihilates all regular functions on that space, the map $\psi$ is injective. Therefore, $\op{Ker}\Phi=\op{Ker}(\psi\circ \Phi)$. 
Now, we may decompose by the Peter-Weyl theorem 
\[R(G/H)= \bigoplus_{{F_{\lambda}\subset R(G/H)}}F_{\lambda}\otimes (F_{\lambda}^*)^H\]
and we note
\[\op{Ann}_{\mathcal{U}(\mathfrak{g})}R(G/H)=
\bigcap_{F_{\lambda}\subset R(G/H)} \op{Ann}_{\mathcal{U}(\mathfrak{g})}(F_{\lambda}).\]
Therefore, we have
\begin{equation}\label{eq:kernel}
\op{Ker}\Phi_0=\op{Ker}\Phi
=\bigcap_{F_{\lambda}\subset R(G/H)}
  \op{Ann}_{\mathcal{U}(\mathfrak{g})}(F_{\lambda})
\supset \bigcap_{\lambda\in\mathfrak{a}_X^*}
  \op{Ann}_{\mathcal{U}(\mathfrak{g})}(F_{\lambda})
\end{equation} 
where the last inclusion follows from Lemma \ref{lem:orbit_rest}. 
Here and in what follows, we assume $\lambda$ is dominant and integral
 whenever we write $F_{\lambda}$.

The cotangent bundle of $X$ is
 $T^* X \simeq \{(gH, \xi)\mid \xi\in (\mathfrak{g}/\op{Ad}(g)\mathfrak{h})^*\}$
 and the moment map is given by 
\[\mu\colon T^*X\rightarrow \mathfrak{g}^*,\quad (x,\xi)\mapsto \xi\in \mathfrak{g}^*.\]
As we stated in Theorem~\ref{thm:moment_image}, 
 \cite[Lemma 3.1 and Corollary 3.3]{Kno94} give the image of the moment map
 in terms of $\mathfrak{a}_X^*$:
\[\overline{\mu(T^*X)}=\overline{G\cdot \mathfrak{a}_X^*}.\]
In particular, the image of the moment map
 contains a dense subset of semisimple elements.

Let $\mathfrak{q}_X\subset \mathfrak{g}$ be the Lie algebra of $Q_X$
 with Levi decomposition 
\[ Q_X=L_XN_X,\qquad 
\mathfrak{q}_X=\mathfrak{l}_X\oplus \mathfrak{n}_X,\qquad
 \mathfrak{n}_X=\bigoplus_{\alpha\in \Delta(\mathfrak{n}_X,\mathfrak{j})}\mathfrak{g}_{\alpha}.\]
Define
\[
Q_X^0:=L_X^0N_X,\qquad
J_{\mathfrak{a}_X}:=\op{Ker}\bigl(\mathcal{U}(\mathfrak{g})\rightarrow \op{Diff}(G/Q_X^0)\bigr).\]
The following fact is the Corollary on page 453 of \cite{BB82}.

\begin{fact}[Borho-Brylinski] \label{lem:BB}
We have 
\begin{align}\label{eq:annihilator}
J_{\mathfrak{a}_X}=\op{Ann}_{\mathcal{U}(\mathfrak{g})}(\mathcal{U}(\mathfrak{g})\otimes_{\mathcal{U}(\mathfrak{q}_X^0)}\mathbb{C}) 
=\bigcap_{\lambda\in \mathfrak{a}_X^*} \op{Ann}_{\mathcal{U}(\mathfrak{g})}(\mathcal{U}(\mathfrak{g})\otimes_{\mathcal{U}(\mathfrak{q}_X)}\mathbb{C}_{\lambda}).
\end{align}
Here, $\mathbb{C}$ is the trivial $\mathcal{U}(\mathfrak{q}_X^0)$-module,
 and $\mathbb{C}_{\lambda}$ is the one-dimensional $\mathcal{U}(\mathfrak{q}_X)$-module
 on which $Z(\mathfrak{l}_X)$ acts by $\lambda$.
\end{fact}

Since each $F_{\lambda}$ for
 $\lambda\in \mathfrak{a}_X^*$
 is a quotient of $\mathcal{U}(\mathfrak{g})\otimes_{\mathcal{U}(\mathfrak{q}_X)}
 \mathbb{C}_{\lambda}$, we deduce 
\[\op{Ann}_{\mathcal{U}(\mathfrak{g})}(\mathcal{U}(\mathfrak{g})
 \otimes_{\mathcal{U}(\mathfrak{q}_X)}\mathbb{C}_{\lambda})
 \subset \op{Ann}_{\mathcal{U}(\mathfrak{g})}F_{\lambda}.\]
Together with \eqref{eq:kernel}, and \eqref{eq:annihilator}, this implies
\begin{equation}\label{eq:contain}
J_{\mathfrak{a}_X}\subset \bigcap_{\lambda\in \mathfrak{a}_X^*} \op{Ann}_{\mathcal{U}(\mathfrak{g})}(F_{\lambda})\subset \op{Ker}\Phi_0.
\end{equation}

The following lemma simplifies the statement of our later result:
\begin{lemma}\label{lem:rho_n}
$\rho(\mathfrak{n}_X)\in \mathfrak{a}_X^*$.
\end{lemma}

\begin{proof}
Since $H$ is unimodular, $X=G/H$ has a $G$-invariant differential form of top degree.
By restriction, it gives a $Q_X$-invariant form on $Q_X\times^{L_X} Z$.
Therefore, the line bundle
\[(\bigwedge^{\dim X} T^* X)|_Z\simeq
 \bigwedge^{\dim Z} T^*Z \otimes \bigwedge^{\dim \mathfrak{n}_X} T^*_Z (Q_X\times^{L_X} Z)\]
 has a nonzero $L_X$-invariant section, and hence in particular an $L_X^0$-invariant section.
Recall that $L_X^0$ acts trivially on $Z$ and on $T^*Z$.
On the other hand, the fibers of $T^*_Z (Q_X\times^{L_X} Z)$ are identified with
 $(\mathfrak{q}_X/\mathfrak{l}_X)^*$.
As a result, $L_X^0$ must act trivially on
 $\bigwedge^{\dim \mathfrak{n}_X} (\mathfrak{q}_X/\mathfrak{l}_X)^*$,
 which implies $\rho(\mathfrak{n}_X)$ is zero on $\mathfrak{l}_X^0$
 and $\rho(\mathfrak{n}_X)\in \mathfrak{a}_X^*$. 
\end{proof}

Suppose that $V$ is an irreducible $(\mathfrak{g}, K)$-module
 and suppose there exists an injective linear map 
\[V \hookrightarrow C^{\infty}(G_{\mathbb{R}}/H_0)\]
 which respects actions of $\mathfrak{g}$ and $K_{\R}$.
The enveloping algebra $\mathcal{U}(\mathfrak{g})$ acts on $V$ via the map $\Phi_0$
 together with the restriction of the action of $\op{Diff}(G_{\mathbb{R}}/H_0)$
 on $C^{\infty}(G_{\mathbb{R}}/H_0)$ to $V$. In particular, we have
 $\op{Ann}_{\mathcal{U}(\mathfrak{g})}(V)\supset \op{Ker}\Phi_0$.
By \eqref{eq:contain}, we obtain the following proposition.
\begin{proposition}\label{prop:Ann}
If $V$ is an irreducible $(\mathfrak{g}, K)$-module
 and there exists an injective linear map 
 $V \hookrightarrow C^{\infty}(G_{\mathbb{R}}/H_0)$
 which respects actions of $\mathfrak{g}$ and $K_{\R}$, 
 then 
\begin{equation*}
\op{Ann}_{\mathcal{U}(\mathfrak{g})}(V)\supset J_{\mathfrak{a}_X}.
\end{equation*}
\end{proposition}

For an infinitesimal character
 $\xi\colon Z(\mathcal{U}(\mathfrak{g}))\rightarrow \mathbb{C}$,
 define
\[I_{\xi}:=\mathcal{U}(\mathfrak{g})\cdot \op{Ker}(Z(\mathcal{U}(\mathfrak{g}))\stackrel{\xi}\rightarrow \mathbb{C}).\]
Let $W$ be the Weyl group for $\Delta(\mathfrak{g},\mathfrak{j})$.
Recall that there exists a natural algebra isomorphism (so-called the Harish-Chandra isomorphism)
 $\gamma\colon Z(\mathcal{U}(\mathfrak{g}))\simeq S(\mathfrak{j})^W$.
If $\xi\colon Z(\mathcal{U}(\mathfrak{g}))\rightarrow \mathbb{C}$
 is the infinitesimal character of $V$,
 then we may compose with $\gamma^{-1}$
 to give an element of $\mathfrak{j}^*/W$
 or a representative $\xi\in \mathfrak{j}^*$.

\begin{lemma}\label{lem:inf_character} 
Suppose that $V$ is an irreducible $(\mathfrak{g}, K)$-module
 with infinitesimal character $\xi\in \mathfrak{j}^*$ and
  $\op{Ann}_{\mathcal{U}(\mathfrak{g})}(V)\supset J_{\mathfrak{a}_X}$.
Then
\[(W\cdot \xi)\cap (\mathfrak{a}_X^*+\rho_{\mathfrak{l}_X})\neq \emptyset,\]
where we put
\[\rho_{\mathfrak{l}_X}:=
\frac{1}{2}
 \sum_{\alpha\in \Delta(\mathfrak{l}_X,\mathfrak{j})\cap \Delta^+(\mathfrak{g},\mathfrak{j})} \alpha.\]
\end{lemma}
\begin{proof}
Suppose $z\in Z(\mathcal{U}(\mathfrak{g}))$ with 
\[\gamma(z)|_{\mathfrak{a}_X^*+\rho_{\mathfrak{l}_X}}=0.\]
Recall that $F_{\lambda}$ has infinitesimal character
 $\lambda+\rho=\lambda+\rho_{\mathfrak{l}_X}+\rho(\mathfrak{n}_X)$.
In view of Lemma~\ref{lem:rho_n},
 $z\in \op{Ann}_{\mathcal{U}(\mathfrak{g})}(F_{\lambda})$
 for all $\lambda\in \mathfrak{a}_X^*$, and by \eqref{eq:contain},
 $z\in J_{\mathfrak{a}_X}$. 

Now, assume that the conclusion of Lemma \ref{lem:inf_character} is false. 
That is, assume that
 $(W\cdot \xi)\cap (\mathfrak{a}_X^*+\rho_{\mathfrak{l}_X})=\emptyset$. 
Then we may choose a polynomial $p\in \op{Pol}(\mathfrak{j}^*)^W$
 such that $p(w\cdot \xi)\neq 0$ for all $w\in W$ but 
\[p|_{\mathfrak{a}_X^*+\rho_{\mathfrak{l}_X}}=0.\]
Identify $\op{Pol}(\mathfrak{j}^*)^W\simeq S(\mathfrak{j})^W$ in the usual way and write
 $z=\gamma^{-1}(p) \in Z(\mathcal{U}(\mathfrak{g}))$. 
Then $z\in J_{\mathfrak{a}_X}$ by the above argument.
Since $z-\gamma(z)(\xi)\in I_{\xi}$ by the definition of $I_{\xi}$,  
\[\gamma(z)(\xi)=z-(z-\gamma(z)(\xi))\in J_{\mathfrak{a}_X}+I_{\xi}.\]
But, then $\gamma(z)(\xi)\neq 0$ implies $1\in J_{\mathfrak{a}_X}+I_{\xi}$.
On the other hand, 
 $\op{Ann}_{\mathcal{U}(\mathfrak{g})}(V)\supset J_{\mathfrak{a}_X}+I_{\xi}$
 by our assumption.
Hence we must have $V=0$, which is a contradiction.
\end{proof}

For $\lambda\in Z(\mathfrak{l}_X)^*$ define the two-sided ideal
\begin{equation*}
J_{\lambda}:=
\op{Ann}_{\mathcal{U}(\mathfrak{g})}
(\mathcal{U}(\mathfrak{g})\otimes_{\mathcal{U}(\mathfrak{q}_X)}
 \mathbb{C}_{\lambda-\rho(\mathfrak{n}_X)}).
\end{equation*}
Note that $J_{\lambda}\supset I_{\lambda+\rho_{\mathfrak{l}_X}}$,
 or equivalently, the generalized Verma module
 $\mathcal{U}(\mathfrak{g})\otimes_{\mathcal{U}(\mathfrak{q}_X)}
 \mathbb{C}_{\lambda-\rho(\mathfrak{n}_X)}$
 has the infinitesimal character $\lambda+\rho_{\mathfrak{l}_X}$.

\begin{lemma}\label{lem:Ann}
Suppose that $V$ is an irreducible $(\mathfrak{g}, K_{\mathbb{R}})$-module
 and $\op{Ann}_{\mathcal{U}(\mathfrak{g})}(V)\supset J_{\mathfrak{a}_X}$.
Then there exists $\lambda\in \mathfrak{a}_X^*$
 such that 
\[\op{Ann}_{\mathcal{U}(\mathfrak{g})}(V)\supset J_{\lambda}.\]
\end{lemma}

\begin{proof}
Let $\xi\in \mathfrak{j}^*$ be the infinitesimal character of $V$.
By Lemma~\ref{lem:inf_character},
 there exists a finite, nonempty collection 
 $\{\lambda_1,\ldots,\lambda_m\}\subset \mathfrak{a}_X^*$ for which 
\begin{equation*}
(W\cdot \xi)\cap (\mathfrak{a}_X^*+\rho_{\mathfrak{l}_X})
=\{\lambda_1+\rho_{\mathfrak{l}_X},\ldots,\lambda_m+\rho_{\mathfrak{l}_X}\}.
\end{equation*}
By an argument similar to the proof of \cite[Theorem 25]{Soe89}, we obtain
\begin{equation*}
\prod_{i=1}^m J_{\lambda_i}^N\subset J_{\mathfrak{a}_X}+I_{\xi}
\end{equation*}
for some large integer $N$.
Since $V$ is irreducible, our assumption
  $\op{Ann}_{\mathcal{U}(\mathfrak{g})}(V)\supset J_{\mathfrak{a}_X}+I_{\xi}$
 implies that
\begin{equation*}
\op{Ann}_{\mathcal{U}(\mathfrak{g})}(V)\supset J_{\lambda_i}
\end{equation*}
for some $i\in\{1,\dots,m\}$.
\end{proof}


\section{Reduction to quantizations of semisimple orbits}\label{sec:reduction_ss}

In the previous section we saw that 
 the annihilators of irreducible subrepresentations of $C^{\infty}(G_{\R}/H_0)$
 contain $J_{\lambda}$, the annihilator of a generalized Verma module.

We will show in Proposition~\ref{prop:partial_flag} that
 this statement of annihilators 
 implies that representations are
 realized as global sections of $\mathscr{D}$-modules on a partial flag variety
 unless the infinitesimal character is close to certain root hyperplanes.

Fix a holomorphic involution $\theta$ of $G$ that commutes with $\sigma$
 and restricts to a Cartan involution on $G_{\mathbb{R}}$. Let $K=G^{\theta}$.

If $\mathfrak{q}=\mathfrak{l}+\mathfrak{n}$ is a parabolic subalgebra
 of $\mathfrak{g}$ and $Y:=G/Q$ is the corresponding partial flag variety,
 we write $\mathscr{D}_{Y,\lambda}$ for the sheaf of
 twisted differential operators on $Y$ with parameter $\lambda\in Z(\mathfrak{l})^*$
 (see e.g.\ \cite{Bie90}).
Our normalization is that $\lambda=\rho(\mathfrak{n})$ corresponds to 
 ordinary (untwisted) differential operators.

We retain the notation of the previous section.
Recall that we defined 
 a Levi subalgebra $\mathfrak{l}_X$ and an ideal
 $J_{\mathfrak{a}_X}\subset \mathcal{U}(\mathfrak{g})$
 for a homogeneous space $X=G/H$.

\begin{proposition}\label{prop:partial_flag}
There exists a constant $d>0$ which depends only on $G$
 such that if $V$ is an irreducible $(\mathfrak{g}, K)$-module
 and if $\op{Ann}_{\mathcal{U}(\mathfrak{g})}(V)\supset J_{\mathfrak{a}_X}$,
 then at least one of the following holds:
\begin{enumerate}
\item \label{realization.D-mod}
There exist
 a parabolic subalgebra $\mathfrak{q}=\mathfrak{l}_X+\mathfrak{n}$,
 a parameter $\lambda\in \mathfrak{a}_X^*$ in the good range, and
 a $K$-equivariant $\mathscr{D}_{Y,\lambda}$-module $\mathcal{M}$ on $Y:=G/Q$
 such that $V\simeq \Gamma(Y,\mathcal{M})$.
Here, we say $\lambda$ is in the good range if
 $\langle\lambda+\rho_{\mathfrak{l}_X},\alpha^{\vee}\rangle\not\in \mathbb{R}_{\geq 0}$
 for every $\alpha\in \Delta(\mathfrak{n},\mathfrak{j})$.
\item \label{singular_inf.char.2}
There exist a representative $\xi\in \mathfrak{a}_X^*+\rho_{\mathfrak{l}_X}$
 of the infinitesimal character of $V$ and a root
 $\alpha\in\Delta(\mathfrak{g},\mathfrak{j})
 \setminus\Delta(\mathfrak{l}_X,\mathfrak{j})$ such that
 $|\langle\xi,\alpha^\vee\rangle| < d$.
\end{enumerate}
\end{proposition}

\begin{proof}
Take $d$ such that
 $d>\max_{\alpha\in\Delta(\mathfrak{g},\mathfrak{j})}
 |\langle\rho_{\mathfrak{l}_X},\alpha^\vee\rangle|$
 and suppose the condition \eqref{singular_inf.char.2} in Proposition~\ref{prop:partial_flag}
 does not hold.

By Lemma~\ref{lem:Ann}, there exists $\lambda\in\mathfrak{a}_X^*$
 such that $\op{Ann}_{\mathcal{U}(\mathfrak{g})}(V)\supset J_{\lambda}$.
Then, take a parabolic subalgebra $\mathfrak{q}=\mathfrak{l}_X+\mathfrak{n}$
 of $\mathfrak{g}$ such that
 $\langle\lambda,\alpha^{\vee}\rangle \not\in \mathbb{R}_{>0}$
 for $\alpha\in \Delta(\mathfrak{n},\mathfrak{j})$. 
For example, we may choose
\begin{equation*}
\Delta(\mathfrak{n},\mathfrak{j})
= \{\alpha\in \Delta(\mathfrak{g},\mathfrak{j})
\mid \RE\langle \lambda, \alpha^\vee \rangle < 0 \}
\cup
\{\alpha \in \Delta(\mathfrak{n}_X,\mathfrak{j})
  \mid \RE\langle \lambda, \alpha^\vee \rangle = 0 \}.
\end{equation*}
As we assumed that \eqref{singular_inf.char.2} does not hold, 
 $\langle\lambda,\alpha^{\vee}\rangle \not\in \mathbb{R}_{> -d}$
 for $\alpha\in \Delta(\mathfrak{n},\mathfrak{j})$
 and then our choice of $d$ shows 
 $\langle\lambda+\rho_{\mathfrak{l}_X},\alpha^{\vee}\rangle \not\in \mathbb{R}_{\geq 0}$
 namely, $\lambda$ is in the good range with respect to $\mathfrak{q}$.

We require the following fact which tells that annihilators of generalized Verma modules
 do not depend on the choice of polarizations.
\begin{fact}[{\cite[Corollar 15.27]{Jan83}}]\label{fact:ann_polarization}
Let $\mathfrak{q}=\mathfrak{l}+\mathfrak{n}$
 and $\mathfrak{q}'=\mathfrak{l}+\mathfrak{n}'$
 be two parabolic subalgebras with the same Levi factor.
Then we have
\[\op{Ann}(\mathcal{U}(\mathfrak{g})\otimes_{\mathcal{U}(\mathfrak{q})}
 \mathbb{C}_{\lambda-\rho(\mathfrak{n})})
=\op{Ann}(\mathcal{U}(\mathfrak{g})\otimes_{\mathcal{U}(\mathfrak{q}')}
 \mathbb{C}_{\lambda-\rho(\mathfrak{n}')})\]
for $\lambda\in Z(\mathfrak{l})^*$. 
\end{fact}

Let $Y:=G/Q$ and let $\mathscr{D}_{Y,\lambda}$ be the ring of
 twisted differential operators.
We have a natural homomorphism
\begin{equation*}
\phi\colon \mathcal{U}(\mathfrak{g})\to \Gamma(Y, \mathscr{D}_{Y,\lambda}).
\end{equation*}
The kernel of $\phi$ is 
 $\op{Ann}(\mathcal{U}(\mathfrak{g})\otimes_{\mathcal{U}(\mathfrak{q})}
 \mathbb{C}_{\lambda-\rho(\mathfrak{n})})$
 (see \cite[\S3.6, Corollary]{BB82} or \cite[Corollar 7]{Soe89}),
 which also equals $J_\lambda$
 by Fact~\ref{fact:ann_polarization} above.
Since $\lambda$ is in the good range, $\phi$ is surjective
 (see \cite[I.5.6 Proposition]{Bie90} for a proof).
Hence $\phi$ induces an isomorphism of algebras
\[\mathcal{U}(\mathfrak{g})/J_{\lambda}\simeq \Gamma(Y,\mathscr{D}_{Y,\lambda}).\]
Moreover, by \cite[I.6.3 Theorem]{Bie90},
\[V\mapsto V\otimes_{\mathcal{U}(\mathfrak{g})/J_{\lambda}}\mathscr{D}_{Y,\lambda}\]
 gives an equivalence of categories between
 $(\mathcal{U}(\mathfrak{g})/J_{\lambda})$-modules and $\mathscr{D}_{Y,\lambda}$-modules,
 whose inverse is given by taking the space of global sections.
Therefore,
 $\mathcal{M}:=V\otimes_{\mathcal{U}(\mathfrak{g})/J_{\lambda}}\mathscr{D}_{Y,\lambda}$
 satisfies the condition \eqref{realization.D-mod} of Proposition~\ref{prop:partial_flag}.

For $d$ to be independent of $V$ or $H_0$, 
 we may take the maximum of the above definition of $d$
 for $\mathfrak{l}_X$ running over all Levi subalgebras  of $\mathfrak{g}$.
\end{proof}

Let $\widehat{G}_{\mathbb{R}}$ denote the set consisting of irreducible, unitary representations of $G_{\mathbb{R}}$. 
Let $X_0=G_{\R}/H_0$.
Recall that we defined $\op{supp} L^2(X_0)$ to be the support of
 the Plancherel measure.
Then by \cite[\S 2.3]{Ber88}, for almost every
 $(\pi,V_{\pi})$ in $\op{supp}L^2(X_0)$,
 there exists an injective map 
\[(V_{\pi})_K\hookrightarrow C^{\infty}(X_0)\]
which respects actions of $\mathfrak{g}$ and $K_{\mathbb{R}}$.
Here, $(V_{\pi})_K$ denotes the underlying $(\mathfrak{g},K)$-module of $V_{\pi}$.
Then by Proposition~\ref{prop:Ann},
 we have $\op{Ann}_{\mathcal{U}(\mathfrak{g})}\bigl((V_{\pi})_K\bigr)\supset J_{\mathfrak{a}_X}$.

In a way similar to \cite[Th\'{e}or\`{e}m 1]{BD60}, we can show that
 the set of irreducible unitarizable $(\mathfrak{g},K)$-modules $V$ satisfying 
\begin{equation}\label{cond:AnnIdeal}
\op{Ann}_{\mathcal{U}(\mathfrak{g})}(V)\supset J_{\mathfrak{a}_X}
\end{equation}
 is closed in $\widehat{G}_{\mathbb{R}}$.
That is, \eqref{cond:AnnIdeal}
 is a closed condition in $\widehat{G}_{\mathbb{R}}$.
Therefore, \eqref{cond:AnnIdeal} is satisfied for every
 irreducible representation in $\op{supp} L^2(X_0)$.

Here is the main theorem in this section.

\begin{theorem}\label{thm:reduction_semisimple}
There exists a constant $d>0$ which depends only on $G$
 such that
 if $(\pi,V_{\pi})\in \op{supp}L^2(X_0)$,
 then at least one of the following holds:
\begin{enumerate}
\item 
There exist $\mathfrak{l}_{\mathbb{R}}\subset \mathfrak{g}_{\mathbb{R}}$
 and $(\mathcal{O}, \Gamma)$ such that 
 \begin{itemize}
 \item \label{item:regular}
 The complexification $\mathfrak{l}$ of $\mathfrak{l}_{\mathbb{R}}$ is
 $G$-conjugate to $\mathfrak{l}_X$,
 \item $(\mathcal{O}, \Gamma)$ is a semisimple orbital parameter
 such that $\pi\simeq \pi(\mathcal{O}, \Gamma)$, and 
 \item
 $\mathcal{O}$ intersects $\mathfrak{a}_X^*\cap
 \sqrt{-1}Z(\mathfrak{l}_{\mathbb{R}})^*_{\rm gr}$.
 \end{itemize}
\item \label{item:singular}
 We can take a representative $\xi\in\mathfrak{a}_X^*+\rho_{\mathfrak{l}}$
 of the infinitesimal character of $\pi$ and a root
 $\alpha\in\Delta(\mathfrak{g},\mathfrak{j})
 \setminus\Delta(\mathfrak{l},\mathfrak{j})$ such that
 $|\langle\xi,\alpha^\vee\rangle|<d$.
\end{enumerate}
\end{theorem}

\begin{proof}
As mentioned above, we have
 $\op{Ann}_{\mathcal{U}(\mathfrak{g})}\bigl((V_{\pi})_K\bigr)\supset J_{\mathfrak{a}_X}$.
By Proposition~\ref{prop:partial_flag}, 
 we may assume that Proposition~\ref{prop:partial_flag} \eqref{realization.D-mod} holds, namely,
there exist
 a parabolic subalgebra $\mathfrak{q}=\mathfrak{l}_X+\mathfrak{n}$,
 a parameter $\lambda\in \mathfrak{a}_X^*$ in the good range, and
 a $K$-equivariant $\mathscr{D}_{Y,\lambda}$-module $\mathcal{M}$ on $Y:=G/Q$
 such that $(V_{\pi})_K\simeq \Gamma(Y, \mathcal{M})$.

Let $\tilde{Y}$ be the complete flag variety for $G$
 and let $p\colon \tilde{Y}\twoheadrightarrow Y$
 be the natural projection.
Then we have natural isomorphisms
\begin{equation*}
p_* p^*\mathcal{M} \simeq p_*p^*\mathcal{O}_Y\otimes_{\mathcal{O}_Y} \mathcal{M}
\simeq \mathcal{M},
\end{equation*}
where $p_*$ denotes the direct image of $\mathcal{O}$-modules.
Hence \[(V_{\pi})_K\simeq \Gamma(\tilde{Y}, p^*\mathcal{M}).\]
It is easy to see that this isomorphism respects $(\mathfrak{g},K)$-actions.

The pull-back $p^*\mathcal{M}$ is a twisted $\mathscr{D}_{\tilde{Y}}$-module.
More precisely, it is a $K$-equivariant $\mathscr{D}_{\tilde{Y},\lambda+\rho_{\mathfrak{l}}}$-module.
Let $\tilde{S}$ be a dense $K$-orbit in
 $p^{-1}(\op{supp}\mathcal{M})=\op{supp}p^*\mathcal{M}$.
Since $\tilde{Y}$ is $\mathscr{D}_{\tilde{Y},\lambda+\rho_{\mathfrak{l}}}$-affine
 and $\Gamma(\tilde{Y}, p^*\mathcal{M})$ is an irreducible $(\mathfrak{g},K)$-module,
 $p^*\mathcal{M}$ is a minimal extension of
 a $K$-equivariant line bundle with connection $\tilde{\mathcal{L}}$
 on $\tilde{S}$.
Fix a point $o\in \tilde{S}$ and let $B$ be the stabilizer of $o$ in $G$.
We may assume that $B$ contains a $\theta$-stable and $\sigma$-stable Cartan subgroup $J$.
Write $J_\mathbb{R}=T_\mathbb{R}A_\mathbb{R}$ for the Cartan decomposition of the real form of $J$.
By replacing $Q$ with its $G$-conjugate, we may assume that $Q$
 is the stabilizer of the point $p(o)\in Y$.
Let $Q=LN$ be the Levi decomposition such that $L\supset J$.
Note that $L$ is $G$-conjugate to $L_X$.

Then by the correspondence between the Langlands classification
 and the Beilinson-Bernstein classification of $(\mathfrak{g},K)$-modules
 (see \cite{Sch91}, \cite[Chapter XI]{KV95}),
 the Langlands parameter of $(V_{\pi})_K$ is given as
 $(J_{\mathbb{R}}, \gamma, \Delta_{i\mathbb{R}}^+)$ in the notation of \cite{ALTV} such that
 $d\gamma=\lambda+\rho_{\mathfrak{l}}(\in \mathfrak{j}^*)$ and 
 $\Delta_{i\mathbb{R}}^+$ is the set of imaginary roots which are not the roots in $\mathfrak{b}$.
We note that $\lambda+\rho_{\mathfrak{l}}$ is regular as we assumed Proposition~\ref{prop:partial_flag} \eqref{realization.D-mod}.

Write $d\gamma=\RE(d\gamma)+\sqrt{-1}\IM(d\gamma)$,
 where $\RE(d\gamma), \IM(d\gamma)\in
 \Hom(\mathfrak{j}_{\mathbb{R}},\mathbb{R})\subset \mathfrak{j}^*$
 and write 
 $\overline{d\gamma}^{h}=-\RE(d\gamma)+\sqrt{-1}\IM(d\gamma)$
 for the Hermitian dual.
We use similar notation for any vector in $\mathfrak{j}^*$.

We want to prove that $\RE (\lambda)=0$.
Since the $\rho$-shift of $d\gamma|_{\mathfrak{t}_{\mathbb{R}}}$ is a differential of a character of $T_{\mathbb{R}}$,
 the compact part of the Cartan subgroup, we have $\RE(d\gamma)|_{\mathfrak{t}}=0$.
Moreover, since $(\pi, V_{\pi})$ is unitary, $\pi$ is isomorphic to its Hermitian dual.
Hence by uniqueness in the Langlands classification,
 $d\gamma$ and $\overline{d\gamma}^{h}$ lie in the same Weyl group orbit.
Let $w \in W(\Delta(\mathfrak{g},\mathfrak{j}))$
 such that $w\cdot d\gamma = \overline{d\gamma}^{h}$.
The Weyl group 
$W(\Delta(\mathfrak{g},\mathfrak{j}))$ preserves 
 the real span of roots so it preserves
 $\sqrt{-1}\mathfrak{t}_{\mathbb{R}}\oplus \mathfrak{a}_{\mathbb{R}}$.
Hence
 $w\cdot \IM(d\gamma)|_{\mathfrak{a}}
 =\IM(d\gamma)|_{\mathfrak{a}}$.
Put 
\begin{align*}
\Delta_1:=\{\alpha\in \Delta(\mathfrak{g},\mathfrak{j})\mid
 \langle\IM(d\gamma)|_{\mathfrak{a}},\alpha^{\vee}\rangle=0\}.
\end{align*}
Then $w\in W(\Delta_1)$.

We note that
\begin{align*}
&d\gamma=\bigl(\RE(d\gamma)|_{\mathfrak{a}}
 +\sqrt{-1}\IM(d\gamma)|_{\mathfrak{t}}\bigr)
+\sqrt{-1}\IM(d\gamma)|_{\mathfrak{a}},\\
&\lambda=\bigl(\RE(d\gamma)|_{\mathfrak{a}}
 +\sqrt{-1}\IM(d\gamma)|_{\mathfrak{t}}
 -\rho_{\mathfrak{l}}\bigr)
+\sqrt{-1}\IM(d\gamma)|_{\mathfrak{a}},
\end{align*}
and we have
\begin{align*}
&\langle\RE(d\gamma)|_{\mathfrak{a}}
 +\sqrt{-1}\IM(d\gamma)|_{\mathfrak{t}},\,\alpha^{\vee}\rangle
\in \mathbb{R},\quad 
\langle\rho_{\mathfrak{l}},\alpha^{\vee}\rangle\in \mathbb{R},\\ 
&\langle\sqrt{-1}\IM(d\gamma)|_{\mathfrak{a}},\,\alpha^{\vee}\rangle\in \sqrt{-1}\mathbb{R}.
\end{align*}
Hence for $\alpha\in\Delta(\mathfrak{g},\mathfrak{j})$,
\[
\alpha\in \Delta_1 \Leftrightarrow
 \langle d\gamma,\alpha^{\vee}\rangle\in \mathbb{R}
\Leftrightarrow
 \langle \lambda,\alpha^{\vee}\rangle\in \mathbb{R}.
\]
Since $\langle\lambda,\alpha^{\vee}\rangle=0$ for
 $\alpha\in\Delta(\mathfrak{l})$,
 we have $\Delta(\mathfrak{l})\subset \Delta_1$.
As we assumed Proposition~\ref{prop:partial_flag} \eqref{realization.D-mod},
  $\lambda$ is in the good range. 
Hence for $\alpha\in \Delta_1$, 
\begin{equation}\label{eq:ln.root.cond}
\langle \lambda, \alpha^\vee\rangle < 0
 \ (\text{resp.}\ =0,\ >0) \text{ if }
\alpha\in \Delta(\mathfrak{n})
 \ (\text{resp.}\ \alpha\in \Delta(\mathfrak{l}),\ -\Delta(\mathfrak{n})).
\end{equation}

Since $(\pi, V_{\pi})$ is unitary,
 $\RE(d\gamma)$ lies in a certain bounded region
 (see \cite[Chapter XVI, \S 5]{Kna86}).
Suppose that the condition \eqref{item:singular} of Theorem~\ref{thm:reduction_semisimple}
 does not hold for the constant $d$ greater than
\[\max\{\langle\rho_{\mathfrak{l}},\beta^{\vee}\rangle
  \mid\beta\in \Delta_1\}
+\max\{\langle 2\RE(d\gamma),\beta^{\vee}\rangle
  \mid\beta\in \Delta_1\}.\]
Combining with \eqref{eq:ln.root.cond},
 we have for $\alpha\in \Delta_1$,
\begin{equation}\label{eq:l.root.cond} 
\begin{split}
&\alpha\in \Delta(\mathfrak{n})\Leftrightarrow
 \langle d\gamma,\alpha^{\vee}\rangle\leq -d, \qquad
\alpha\in -\Delta(\mathfrak{n})\Leftrightarrow
 \langle d\gamma,\alpha^{\vee}\rangle\geq d, \\
&\alpha\in \Delta(\mathfrak{l})\Leftrightarrow
 |\langle d\gamma,\alpha^{\vee}\rangle|
  \leq \max\{\langle\rho_{\mathfrak{l}},\beta^{\vee}\rangle
  \mid\beta\in \Delta_1\}.
\end{split}
\end{equation}
If $\alpha\in \Delta_1\cap \Delta(\mathfrak{n})$, then
\begin{align*}
\langle d\gamma,w^{-1}\cdot\alpha^{\vee}\rangle
&=\langle w\cdot d\gamma,\alpha^{\vee}\rangle\\
&=\langle \overline{d\gamma}^h,\alpha^{\vee}\rangle\\
&=\langle  d\gamma-2\RE(d\gamma),\alpha^{\vee}\rangle\\
&<-\max\{\langle\rho_{\mathfrak{l}},\beta^{\vee}\rangle
  \mid\beta\in \Delta_1\},
\end{align*}
where the last inequality follows from our choice of $d$
 and $\langle d\gamma,\alpha^{\vee}\rangle\leq -d$.
Therefore, $w\cdot (\Delta_1\cap \Delta(\mathfrak{n}))
 = \Delta_1\cap \Delta(\mathfrak{n})$ by \eqref{eq:l.root.cond}
 and hence $w\in W(\Delta(\mathfrak{l}))$.
If $\alpha\in\Delta(\mathfrak{l})$, then
\begin{equation*}
|\langle  d\gamma,(\overline{\alpha}^h)^\vee \rangle|
 =|\langle \overline{ d\gamma}^h,\alpha^\vee \rangle|
 =|\langle w\cdot  d\gamma, \alpha^\vee\rangle|
 =|\langle  d\gamma, (w^{-1}\cdot \alpha)^\vee\rangle|
 < d
\end{equation*}
and hence $\overline{\alpha}^h\in\Delta(\mathfrak{l})$,
 namely, $\Delta(\mathfrak{l})$ is preserved by the Hermitian dual.
In addition, $w\in \Delta(\mathfrak{l})$ implies that
 $\RE(d\gamma)=\frac{1}{2}(d\gamma-\overline{d\gamma}^h)
=\frac{1}{2}(d\gamma-w\cdot d\gamma)$
is a linear combination of $\Delta(\mathfrak{l})$.
Therefore, in view of the decomposition
\begin{equation*}
\RE(d\gamma)=\RE(\lambda)+\RE(\rho_{\mathfrak{l}})
 \in Z(\mathfrak{l})^* \oplus ([\mathfrak{l},\mathfrak{l}]\cap \mathfrak{j})^*,
\end{equation*}
 we obtain $\RE(\lambda)=0$.
Since $\mathfrak{l}_{\mathbb{R}}:=\mathfrak{l}\cap\mathfrak{g}_\mathbb{R}$ is
 a real form of $\mathfrak{l}$,
 we have proved that $\lambda\in \sqrt{-1}Z(\mathfrak{l}_\mathbb{R})^*$.
As we assumed \eqref{realization.D-mod} of Proposition~\ref{prop:partial_flag}, we have
 $\lambda\in
 \mathfrak{a}_X^*\cap \sqrt{-1}Z(\mathfrak{l}_\mathbb{R})^*_{\rm gr}$.

Recall that $\mathcal{M}$ is an irreducible
 $K$-equivariant $\mathscr{D}_{Y,\lambda}$-module on $Y=G/Q$
 such that $(V_{\pi})_K\simeq \Gamma(Y, \mathcal{M})$.
Let $S$ be the $K$-orbit in $Y$ containing $p(o)$.
Then $S=p(\tilde{S})$ and $\op{supp} \mathcal{M}=\overline{S}$, the closure of $S$.
Let $i\colon S \hookrightarrow Y$ and $\tilde{i}\colon  p^{-1}(S) \hookrightarrow \tilde{Y}$
 denote the natural inclusion maps.
We have the following commutative diagram:
\begin{align*}
\xymatrix{
\tilde{S} \ar[r]  & p^{-1}(S) \ar[r]^{\tilde i} \ar[d]_{p} &  \tilde{Y} \ar[d]_{p}
\\
& S \ar[r]^{i} & Y
}
\end{align*}
Let $i^\dagger:=Li^*[\dim S-\dim Y]$ denotes the shifted inverse image functor
 for $\mathscr{D}$-modules as in \cite{HTT08}.
Since $\op{supp}\mathcal{M}=\overline{S}$, the complex
 $i^{\dagger}\mathcal{M}$ is concentrated in one degree, namely,
 $H^q(i^{\dagger}\mathcal{M})=0$ for $q\neq 0$.
Let $\mathcal{L}:=H^0(i^{\dagger}\mathcal{M})$, which is a $K$-equivariant
 twisted $\mathscr{D}$-module on $S$.
By an isomorphism $p^* i^{\dagger}\mathcal{M}\simeq \tilde{i}^{\dagger}p^* \mathcal{M}$,
 we have $p^*\mathcal{L}|_{\tilde{S}}\simeq \tilde{\mathcal{L}}$.
Hence $\mathcal{L}$ must be a $K$-equivariant line bundle.

Next, decompose the map $i$ into $i=j\circ k$:
\[
S\xrightarrow{k} Y\setminus(\overline{S}\setminus S)
\xrightarrow{j} Y
\]
so $j$ is an open immersion and $k$ is a closed immersion.
By the definition of $\mathcal{L}$, we have
 $k^{\dagger}(j^{-1}\mathcal{M})\simeq \mathcal{L}$.
Since $j^{-1}\mathcal{M}$ is supported on $S$, 
 there is an isomorphism $j^{-1}\mathcal{M}\simeq k_+\mathcal{L}$
 by Kashiwara's equivalence.
Then we get a nonzero element in 
\[\Hom(j^{-1}\mathcal{M}, k_+\mathcal{L})
\simeq \Hom(\mathcal{M}, j_*k_+\mathcal{L}).\]
Hence $\Hom(\mathcal{M}, i_+\mathcal{L})\neq 0$.
Write the $K$-equivariant line bundle $\mathcal{L}$ on $S$ as
 $\mathcal{L}=K\times_{(Q\cap K)} \tau$ 
 for a character of $Q\cap K$.
As in Section~\ref{sec:quantization_semisimple},
 we define a unitary character $\Gamma_{\lambda}$ of $\widetilde{L_{\mathbb{R}}}$
 such that $d\Gamma_{\lambda}=\lambda\in\sqrt{-1}Z(\mathfrak{l}_{\mathbb{R}})^*$ and 
\[
(\Gamma_{\lambda}
 \otimes e^{-\rho(\theta\sigma(\mathfrak{n}))})|_{L_{\mathbb{R}}\cap K_{\mathbb{R}}}
\otimes
 \wedge^{\rm top}(\mathfrak{k}/(\mathfrak{l}\cap \mathfrak{k}))
\simeq \tau|_{L_{\mathbb{R}}\cap K_{\mathbb{R}}}.
\]
Notice that the roles of $Q$ and $\sigma(Q)$ here are interchanged
 from Section~\ref{sec:quantization_semisimple}.
Then by \eqref{eq:localization} 
\[
\sum_q(-1)^q \Gamma(Y, R^qi_+\mathcal{L})=\pi(\mathcal{O},\Gamma).
\]
We saw above that $\Hom(\mathcal{M}, R^0i_+\mathcal{L})\neq 0$
 and hence $\Hom_{\mathfrak{g},K}((V_{\pi})_K, \Gamma(Y,R^0i_+\mathcal{L}))\neq 0$.
On the other hand, $R^qi_+\mathcal{L}$ for $q>0$
 is supported on $\overline{S}\setminus S$.
Hence the irreducible $(\mathfrak{g},K)$-module $(V_{\pi})_K$
 does not appear in the composition series of $\Gamma(Y,R^qi_+\mathcal{L})$
 for $q>0$.
Since $\pi(\mathcal{O},\Gamma)$ is irreducible,
 we conclude that $\pi(\mathcal{O},\Gamma)\simeq (V_{\pi})_K$.
Thus, the condition \eqref{item:regular} in Theorem~\ref{thm:reduction_semisimple} holds.
\end{proof}

Note that if $\pi$ satisfies \eqref{item:regular} in Theorem~\ref{thm:reduction_semisimple},
 then $\pi\in \widehat{G}_{\mathbb{R}}^{\mathfrak{l}_X}$
 in the notation of Section~\ref{sec:introduction}.

Now we prove Corollary~\ref{cor:singular_representation}.

\begin{proof}[Proof of Corollary~\ref{cor:singular_representation}]
\eqref{AC.regular} is a direct consequence of Theorem~\ref{thm:main},
 which will be proved in Section~\ref{sec:proof}.

To prove \eqref{AC.singular}, recall 
 the Langlands classification of irreducible admissible representations
 of $G_{\mathbb{R}}$. 
In the notation of \cite{ALTV}, they are parametrized by
 triples $(J_{\mathbb{R}},\gamma,\Delta_{i\mathbb{R}}^+)$.
Write $\pi(J_{\mathbb{R}},\gamma,\Delta_{i\mathbb{R}}^+)$ for the irreducible representation
 of $G_{\mathbb{R}}$ corresponding to $(J_{\mathbb{R}},\gamma,\Delta_{i\mathbb{R}}^+)$.
Then the infinitesimal character of $\pi(J_{\mathbb{R}},\gamma,\Delta_{i\mathbb{R}}^+)$
 is given by the $W$-orbit through $d\gamma$.

Since there are finitely many Cartan subgroups $J_{\mathbb{R}}$ up to conjugation
 and the asymptotic cone commutes with finite union, 
 we may fix $J_{\mathbb{R}}$ and treat only representations
 of the form $\pi(J_{\mathbb{R}},\gamma,\Delta_{i\mathbb{R}}^+)$.
By replacing $\mathfrak{j}$ in the statement
 of Corollary~\ref{cor:singular_representation} with its conjugation,
 we may moreover assume that the complexified Lie algebra of our fixed $J_{\mathbb{R}}$
 is the same as $\mathfrak{j}$ in the statement.

Suppose that
 $\pi(J_{\mathbb{R}},\gamma,\Delta_{i\mathbb{R}}^+)\in
 \op{supp}L^2(X_0)\setminus \widehat{G}_{\mathbb{R}}^{\mathfrak{l}_X}$.
Then by Theorem~\ref{thm:reduction_semisimple},
 it satisfies \eqref{item:singular} in the theorem.
Hence there exist $w\in W$ and $\xi\in \mathfrak{j}^*$
 such that $d\gamma = w\cdot \xi$, 
 $\xi \in \mathfrak{a}_X^*+\rho_{\mathfrak{l}_X}$
 and $|\langle \xi,\alpha^\vee\rangle|< d$
 for some $\alpha\in\Delta(\mathfrak{g},\mathfrak{j})\setminus\Delta(\mathfrak{l}_X,\mathfrak{j})$.
Therefore,
\begin{equation}\label{eq:ac_inf.char.}
\begin{split}
&\op{AC}\bigl(
 \bigl\{d\gamma\mid \pi(J_{\mathbb{R}},\gamma,\Delta_{i\mathbb{R}}^+)\in \op{supp}L^2(X_0)
 \setminus\widehat{G}_{\mathbb{R}}^{\mathfrak{l}_X}\bigr\}\bigr) \\
&\subset 
\bigcup_{w\in W}
\bigcup_{\alpha\in\Delta(\mathfrak{g},\mathfrak{j})\setminus\Delta(\mathfrak{l}_X,\mathfrak{j})}
w\cdot \op{AC}(\{\xi\in \mathfrak{a}_X^*+\rho_{\mathfrak{l}_X}:|\langle \xi, \alpha^\vee\rangle|< d\})
 \\ 
&=\bigcup_{w\in W}
\bigcup_{\alpha\in\Delta(\mathfrak{g},\mathfrak{j})\setminus\Delta(\mathfrak{l}_X,\mathfrak{j})}
w\cdot ( \mathfrak{a}_X^*\cap\alpha^\perp).
\end{split}
\end{equation}

Consider the decomposition
 $d\gamma=\RE(d\gamma)+\sqrt{-1}\IM(d\gamma)$
 with $\RE(d\gamma), \IM(d\gamma)\in \mathfrak{j}_{\mathbb{R}}^*$.
If $\pi(J_{\mathbb{R}},\gamma,\Delta_{i\mathbb{R}}^+)$ is unitary, then $\RE(d\gamma)$ is bounded.
Hence the left hand side of \eqref{eq:ac_inf.char.}
 is contained in $\sqrt{-1}\mathfrak{j}_{\mathbb{R}}^*$.
Define
\[\sqrt{-1}\mathfrak{j}_{\mathbb{R},X,\mathrm{sing}}^*
:=\sqrt{-1}\mathfrak{j}_{\mathbb{R}}^*\cap
\bigcup_{w\in W}
\bigcup_{\alpha\in\Delta(\mathfrak{g},\mathfrak{j})\setminus\Delta(\mathfrak{l}_X,\mathfrak{j})}
w\cdot (\mathfrak{a}_X^*\cap\alpha^\perp).
\]
Then
\[\op{AC}\bigl(
 \bigl\{d\gamma\mid \pi(J_{\mathbb{R}},\gamma,\Delta_{i\mathbb{R}}^+)\in \op{supp}L^2(X_0)
 \setminus\widehat{G}_{\mathbb{R}}^{\mathfrak{l}_X}\bigr\}\bigr) 
\subset \sqrt{-1}\mathfrak{j}_{\mathbb{R},X,\mathrm{sing}}^*\]
and it is easy to see that
\[\dim_{\mathbb{R}} \sqrt{-1}\mathfrak{j}_{\mathbb{R},X,\mathrm{sing}}^*
< \dim_{\mathbb{C}}\mathfrak{a}_X.\]
Since
\begin{align*}
&\op{AC}\Bigl(
\bigcup_{\pi\in \op{supp}L^2(G_{\mathbb{R}}/H_0) \setminus
 \widehat{G}_{\mathbb{R}}^{\mathfrak{l}_X}}
 \chi_{\pi}\Bigr) \\
&= 
W\cdot \op{AC}\bigl(
 \bigl\{d\gamma\mid \pi(J_{\mathbb{R}},\gamma,\Delta_{i\mathbb{R}}^+)\in \op{supp}L^2(X_0)
 \setminus\widehat{G}_{\mathbb{R}}^{\mathfrak{l}_X}\bigr\}\bigr),
\end{align*}
Corollary~\ref{cor:singular_representation} \eqref{AC.singular} is proved.
\end{proof}

To describe representations of type
 \eqref{item:singular} in Theorem~\ref{thm:reduction_semisimple},
 we introduce some notation.
For a Levi subalgebra $\mathfrak{l}\subset\mathfrak{g}$,
 its Cartan subalgebra $\mathfrak{j}\subset \mathfrak{l}$
 and a constant $d>0$, 
 define subsets $\Xi(\mathfrak{l},d)\subset \mathfrak{j}^*$
 and $\widehat{G}_{\mathbb{R}}(\mathfrak{l},d)\subset \widehat{G}_{\mathbb{R}}$ by
\begin{align*}
&\Xi(\mathfrak{l},d):=
\{\xi\in Z(\mathfrak{l})^*+\rho_{\mathfrak{l}}\mid
 \exists\alpha\in\Delta(\mathfrak{g},\mathfrak{j})
 \setminus\Delta(\mathfrak{l},\mathfrak{j}) \text{ such that }
 |\langle\xi,\alpha^\vee\rangle|< d \},\\
&\widehat{G}_{\mathbb{R}}(\mathfrak{l},d):=
\{\pi\in \widehat{G}_{\mathbb{R}}\mid
 \text{The infinitesimal character of $\pi$
 has a representative in $\Xi(\mathfrak{l},d)$}\}.
\end{align*}

For the proof of main theorems in \S\ref{sec:proof}, 
 we need Lemma~\ref{lem:singular_spectrum},
 which states that the contribution to singular spectrum
 from representations of type \eqref{item:singular}
 in Theorem~\ref{thm:reduction_semisimple} is small.

For a unitary representation $(\Pi, V_{\Pi})$ of $G_{\mathbb{R}}$,
 define the \emph{wave front set} 
 and the \emph{singular spectrum} of $\Pi$ by
\[
\op{WF}(\Pi)
=\overline{\bigcup_{u,v\in V_{\Pi}}\op{WF}_e(\pi(g)u,v)},
\quad 
\op{SS}(\Pi)
=\overline{\bigcup_{u,v\in V_{\Pi}}\op{SS}_e(\pi(g)u,v)}.
\]
Here, $\op{WF}_e(\Pi(g)u,v)$ is the wave front set of the
 matrix coefficient function  $(\Pi(g)u,v)$ at $e\in G$,
Similarly, $\op{SS}_e(\Pi(g)u,v)$ is
 the singular spectrum (or the analytic wave front set)
 of $(\Pi(g)u,v)$ at $e$.
Both $\op{WF}(\Pi)$ and $\op{SS}(\Pi)$ are closed
 $G$-invariant subset of $\mathfrak{g}^*(\simeq T^*_e G)$.
We always have $\op{WF}(\Pi)\subset \op{SS}(\Pi)$.
See \cite{HHO16} for the equivalence with Howe's original definition \cite{How81}
 of the wave front set.
We note that a relationship between
 the singular spectrum of functions and the spectrum of representations
 was studied in Kashiwara-Vergne~\cite{KV79}.
Such a microlocal point of view also appeared in 
 Kobayashi's theory~\cite{Ko98b, Ko98c} on the admissibility of restrictions of representations.

\begin{lemma}\label{lem:singular_spectrum}
Let $\Pi$ be a unitary representation of $G_{\mathbb{R}}$
 and $\op{supp} \Pi \subset \widehat{G}_{\mathbb{R}}(\mathfrak{l},d)$.
Then $\op{WF}(\Pi)\cap (G\cdot Z(\mathfrak{l})^*_{\mathrm{reg}})=\op{SS}(\Pi)\cap (G\cdot Z(\mathfrak{l})^*_{\mathrm{reg}})=\emptyset$.
\end{lemma}

\begin{proof}
The proof follows the same line of arguments as in the proof of \cite[Theorem 1.1]{Har18}. 

To each Langlands parameter $\Gamma$,
 one defines the Langlands quotient $J(\Gamma)$,
 which is an irreducible representation of $G_{\mathbb{R}}$.
In \cite[Section 2]{Har18}, we associate a contour $C(\Gamma)\subset \mathfrak{g}^*$.
By \cite[Lemma 3.4, Lemma 3.5]{Har18}, it is enough to show:
\begin{equation*}
\op{AC}\Bigl( \bigcup_{J(\Gamma)\in \widehat{G}_{\mathbb{R}}(\mathfrak{l},d)} C(\Gamma) \Bigr)
\cap (G\cdot Z(\mathfrak{l})^*_{\mathrm{reg}})=\emptyset.
\end{equation*}
If $J(\Gamma)$ has infinitesimal character $\xi\in\mathfrak{j}^*$,
 then $C(\Gamma)\subset G \cdot \xi$ by the definition of $C(\Gamma)$.
Hence
\begin{equation*}
\op{AC}\Bigl( \bigcup_{J(\Gamma)\in \widehat{G}_{\mathbb{R}}(\mathfrak{l},d)} C(\Gamma) \Bigr)
 \subset 
\op{AC} ( G\cdot \Xi (\mathfrak{l},d)).
\end{equation*}
Since $\op{AC} ( G\cdot \Xi (\mathfrak{l},d))$ is $G$-stable, it is enough to show that
\begin{equation}\label{eq:asymp_Xi}
\op{AC} ( G\cdot \Xi (\mathfrak{l},d)) \cap  Z(\mathfrak{l})^*_{\mathrm{reg}}=\emptyset.
\end{equation}

Let $W=W(\mathfrak{g},\mathfrak{j})$ be the Weyl group which acts on $\mathfrak{j}^*$.
We claim that
\begin{align}\label{eq:asymp_cone_complex}
\op{AC}(G\cdot S) \cap \mathfrak{j}^*
= W\cdot \op{AC}(S)
\end{align}
for any subset $S\subset \mathfrak{j}^*$.
Indeed, we have
\begin{align*}
\op{AC}(G\cdot S) 
\supset \op{AC}(W\cdot S)
 = W \cdot \op{AC}(S).
\end{align*}
For the other inclusion, let $\xi\in \op{AC}(G\cdot S) \cap \mathfrak{j}^*$.
Then there exist $g_i\in G$, $s_i\in S$, and $t_i\in \R_{>0}$
 for $i\in \mathbb{N}$
 such that $t_i\to +\infty$ and
 $t_i^{-1} (g_i\cdot s_i)\to \xi$ when $i\to \infty$.
Let $p\colon \mathfrak{g}^*\to \mathfrak{j}^*/W$ be the map induced from
 the isomorphism $S(\mathfrak{g})^G \simeq S(\mathfrak{j})^W$.
By applying $p$ to the convergent sequence, 
 we obtain $t_i^{-1}p(s_i)\to \xi$ in $\mathfrak{j}^*/W$,
 which implies $\xi\in \op{AC}(W\cdot S)$.

Plugging $S=\Xi(\mathfrak{l},d)$
 into \eqref{eq:asymp_cone_complex}, we get
\[
 \op{AC}(G\cdot\Xi(\mathfrak{l},d))
 \cap Z(\mathfrak{l})^*_{\mathrm{reg}}
=\bigl(W\cdot \op{AC}(\Xi(\mathfrak{l},d))\bigr)
 \cap Z(\mathfrak{l})^*_{\mathrm{reg}}.
\]
If $\lambda\in \op{AC}(\Xi(\mathfrak{l},d))$,
 then $\mathfrak{g}(\lambda)\supsetneq\mathfrak{l}$.
Hence $\lambda\in W\cdot \op{AC}(\Xi(\mathfrak{l},d))$
 implies $\dim \mathfrak{g}(\lambda)> \dim \mathfrak{l}$.
Therefore,
 $\bigl(W\cdot \op{AC}(\Xi(\mathfrak{l},d))\bigr)
 \cap Z(\mathfrak{l})^*_{\mathrm{reg}}=\emptyset$
 and \eqref{eq:asymp_Xi} is proved.
\end{proof}


\section{Wave front sets of direct integrals for a Levi, part 1}\label{sec:WF1}

Let $\mathfrak{l}_{\mathbb{R}}$ be a Levi subalgebra of $\mathfrak{g}_{\mathbb{R}}$.
Define a subset
 $\widehat{G}_{\mathbb{R}}^{\mathfrak{l}_{\mathbb{R}}}
 \subset\widehat{G}_{\mathbb{R}}$ as
\[
\widehat{G}_{\mathbb{R}}^{\mathfrak{l}_{\mathbb{R}}}
=\{\pi\in \widehat{G}_{\mathbb{R}}\mid
 \exists \Gamma_{\lambda} \text{ such that }
 \pi\simeq \pi(\mathfrak{l}_{\mathbb{R}},\Gamma_{\lambda})
 \text{ and } \lambda\in \sqrt{-1}Z(\mathfrak{l}_\mathbb{R})^*_{\text{gr}}\}.
\]
The definition of $\pi(\mathfrak{l}_{\mathbb{R}},\Gamma_{\lambda})$
 was given in \S \ref{sec:quantization_semisimple}.
For a complex Levi subalgebra $\mathfrak{l}'\subset \mathfrak{g}$,
 we defined $\widehat{G}_{\mathbb{R}}^{\mathfrak{l}'}$
 in Section~\ref{sec:introduction}. By these definitions,
\[
\widehat{G}_{\mathbb{R}}^{\mathfrak{l}'}
=\bigcup_{\mathfrak{l}_{\mathbb{R}}}
 \widehat{G}_{\mathbb{R}}^{\mathfrak{l}_{\mathbb{R}}},
\]
where $\mathfrak{l}_{\mathbb{R}}$ runs over all Levi subalgebras
 of $\mathfrak{g}_{\mathbb{R}}$
 such that $\mathfrak{l}\sim \mathfrak{l}'$.

We want to prove the following theorem on
 the wave front set and the singular spectrum:
\begin{theorem}\label{thm:direct_integral} 
Let $\mathfrak{l}_{\mathbb{R}}$ be
 a Levi subalgebra of $\mathfrak{g}_{\mathbb{R}}$.
Suppose that $(\Pi,V_{\Pi})$
 is a unitary representation of $G_{\mathbb{R}}$
 which is isomorphic to a direct integral
 of representations in $\widehat{G}_{\mathbb{R}}^{\mathfrak{l}_{\mathbb{R}}}$:
\[\Pi\simeq\int^{\oplus}_{\pi\in \widehat{G}_{\mathbb{R}}^{\mathfrak{l}_{\mathbb{R}}}}
\pi^{\oplus n(\pi)} dm_{\Pi}.\]
Then 
\begin{align*}
\op{WF}(\Pi)\cap (G\cdot Z(\mathfrak{l})^*_{\rm reg})
&=
\op{SS}(\Pi)\cap (G\cdot Z(\mathfrak{l})^*_{\rm reg})\\
&= \op{AC}
\Biggl(\bigcup_{\pi(\mathfrak{l}_{\mathbb{R}},\Gamma_{\lambda})
\in \op{supp}(m_{\Pi})}G_{\mathbb{R}}\cdot \lambda \Biggr)
\cap (G\cdot Z(\mathfrak{l})^*_{\rm reg}).
\end{align*}
\end{theorem}

In this section, we prove the following inclusion.

\begin{lemma}\label{lem:inclusion1} 
In the setting of Theorem \ref{thm:direct_integral},
\begin{equation*}
\op{SS}(\Pi)\cap (G\cdot Z(\mathfrak{l})^*_{\rm reg})
\subset \op{AC}
\Biggl(\bigcup_{\pi(\mathfrak{l}_{\mathbb{R}},\Gamma_{\lambda})
\in \op{supp}(m_{\Pi})}G_{\mathbb{R}}\cdot\lambda \Biggr).
\end{equation*}
\end{lemma}

The proof of Theorem~\ref{thm:direct_integral}
 will be completed in the subsequent two sections.

Before starting the proof of Lemma~\ref{lem:inclusion1}, we see that
 $\widehat{G}_{\mathbb{R}}^{\mathfrak{l}_{\mathbb{R}}}$
 is a locally closed subset of $\widehat{G}_{\mathbb{R}}$
 with respect to the Fell topology.
Let $\{\pi^j\}_j$ be a sequence in $\widehat{G}_{\mathbb{R}}^{\mathfrak{l}_{\mathbb{R}}}$
 which converges to $\pi\in \widehat{G}_{\mathbb{R}}$.
Let $\pi^j=\pi(\mathcal{O}^j,\Gamma^j)$ and
 $\mathcal{O}^j=G_{\mathbb{R}}\cdot \lambda^j$.
Recall from Section~\ref{sec:quantization_semisimple} and \cite[\S 2]{HO20}
 that $\pi(\mathcal{O}^j,\Gamma^j)$ is defined as a unitary parabolic induction
 for a parabolic subgroup $P_{\mathbb{R}}=M_{\mathbb{R}}A_{\mathbb{R}}(N_P)_{\mathbb{R}}$. 
Since there are only finitely many possibilities for $P_{\mathbb{R}}$,
 we may assume that $P_{\mathbb{R}}$ does not depend on $j$ by passing to a subsequence.
We have a decomposition $\lambda^j=\lambda^j_c+\lambda^j_n$ and
 let $(\mathcal{O}^j)^{M_{\mathbb{R}}}=M_{\mathbb{R}}\cdot \lambda^j_c$.
Then we can define a semisimple orbital parameter 
 $\bigl((\mathcal{O}^j)^{M_{\mathbb{R}}},(\Gamma^j)^{M_{\mathbb{R}}}\bigr)$
 for $M_{\mathbb{R}}$ such that
 $\pi(\mathcal{O}^j,\Gamma^j)$ is induced from
 $\pi((\mathcal{O}^j)^{M_{\mathbb{R}}},(\Gamma^j)^{M_{\mathbb{R}}})$.
By \cite{BD60}, the map $\widehat{G}_{\mathbb{R}}\to \mathfrak{j}^*/W$
 sending an irreducible unitary representation to its infinitesimal character is continuous.
Therefore, the infinitesimal character of $\pi^j$ converges to that of $\pi$.
This implies that $\lambda_c^j$ is bounded and hence
 there are only finitely many possibilities for
 $\bigl((\mathcal{O}^j)^{M_{\mathbb{R}}},(\Gamma^j)^{M_{\mathbb{R}}}\bigr)$.
Passing to a subsequence, we may assume all parameters 
 $\bigl((\mathcal{O}^j)^{M_{\mathbb{R}}},(\Gamma^j)^{M_{\mathbb{R}}}\bigr)$
 are the same so let
 $(\mathcal{O}^{M_{\mathbb{R}}},\Gamma^{M_{\mathbb{R}}})=
 \bigl((\mathcal{O}^j)^{M_{\mathbb{R}}},(\Gamma^j)^{M_{\mathbb{R}}}\bigr)$
 and $\lambda_c=\lambda^j_c$.
We may also assume that $\lambda_n^j$ converges to
 $\lambda_n\in\sqrt{-1}\mathfrak{a}_{\mathbb{R}}^*$.
Then as noted in the proof of \cite[Corollary 8.9]{SRV98},
 $\pi$ is isomorphic to an irreducible constituent
 of $\op{Ind}_{P_{\mathbb{R}}}^{G_{\mathbb{R}}}
 \bigl((\mathcal{O}^{M_{\mathbb{R}}},\Gamma^{M_{\mathbb{R}}})
 \boxtimes e^{\lambda_n}\bigr)$.
If $\lambda_c+\lambda_n$ is in the good range, then
 the induced representation is irreducible and 
 $\pi\in \widehat{G}_{\mathbb{R}}^{\mathfrak{l}_{\mathbb{R}}}$.
Otherwise, $\lambda_c+\lambda_n+\rho_{\mathfrak{l}}$ is singular
 and $\pi$ has the singular infinitesimal character.
Since the set of representations with singular infinitesimal characters
 is closed in $\widehat{G}_{\mathbb{R}}$, 
 the above argument proves that
 $\widehat{G}_{\mathbb{R}}^{\mathfrak{l}_{\mathbb{R}}}$ is locally closed.

Let $\mathfrak{q}\subset \mathfrak{g}$ be a parabolic subalgebra with Levi factor $\mathfrak{l}$ and nilradical $\mathfrak{n}$. We may define $\sqrt{-1}Z(\mathfrak{l}_{\mathbb{R}})^{*,\mathfrak{q}}$ to be the subset of $\lambda\in \sqrt{-1}Z(\mathfrak{l}_{\mathbb{R}})_{\text{reg}}^*$ such that for all $\alpha\in \Delta(\mathfrak{n},\mathfrak{j})$, either
\[\operatorname{Im}\langle \lambda, \alpha^{\vee}\rangle>0 \]
or
\[\operatorname{Im}\langle \lambda, \alpha^{\vee}\rangle=0\ \text{and}\ \operatorname{Re}\langle \lambda,\alpha^{\vee}\rangle>0.\]
As noted in \cite{HO20}, in this case, $\mathfrak{q}$ defines a maximally real, admissible polarization of the coadjoint orbit $\mathcal{O}_{\lambda}:=G_{\mathbb{R}}\cdot \lambda$. Although this assignment of $\mathfrak{q}$ to $\lambda$ is not canonical, it is convenient for our argument to make such an assignment. 

Since there are finitely many parabolic subalgebras $\mathfrak{q}\subset \mathfrak{g}$ with Levi factor $\mathfrak{l}$,
 we have a finite disjoint union
\[\sqrt{-1}Z(\mathfrak{l}_{\mathbb{R}})^*_{\text{reg}}
=\bigsqcup_{\mathfrak{q}\subset \mathfrak{g}}
\sqrt{-1}Z(\mathfrak{l}_{\mathbb{R}})^{*,\mathfrak{q}},\quad
\sqrt{-1}Z(\mathfrak{l}_{\mathbb{R}})^*_{\text{gr}}
=\bigsqcup_{\mathfrak{q}\subset \mathfrak{g}}
\sqrt{-1}Z(\mathfrak{l}_{\mathbb{R}})^{*,\mathfrak{q}}_{\text{gr}},\]
where 
$\sqrt{-1}Z(\mathfrak{l}_{\mathbb{R}})^{*,\mathfrak{q}}_{\text{gr}}
:=\sqrt{-1}Z(\mathfrak{l}_{\mathbb{R}})^{*,\mathfrak{q}}
 \cap \sqrt{-1}Z(\mathfrak{l}_{\mathbb{R}})^*_{\text{gr}}$.

Next, let $\widehat{G}_{\mathbb{R}}^{(\mathfrak{l}_{\mathbb{R}},\mathfrak{q})}$
 denote the collection of representations
 $\pi(\mathfrak{l}_{\mathbb{R}},\Gamma_{\lambda})$
 such that $\lambda\in \sqrt{-1}Z(\mathfrak{l}_{\mathbb{R}})^{*,\mathfrak{q}}_{\rm gr}$.
Equivalently, $\widehat{G}_{\mathbb{R}}^{(\mathfrak{l}_{\mathbb{R}},\mathfrak{q})}$
 consists of $\pi(\mathcal{O}, \Gamma)$ such that
 $\mathcal{O}\cap \sqrt{-1}Z(\mathfrak{l}_{\mathbb{R}})^{*,\mathfrak{q}}_{\rm gr}\neq \emptyset$.
Therefore, we have a finite union
\begin{equation}\label{eq:union}
\widehat{G}_{\mathbb{R}}^{\mathfrak{l}_{\mathbb{R}}}
=\bigcup_{\mathfrak{q}\subset \mathfrak{g}}
 \widehat{G}_{\mathbb{R}}^{(\mathfrak{l}_{\mathbb{R}},\mathfrak{q})}.
\end{equation}
Note that the right hand side of \eqref{eq:union} may not be disjoint.
In the same way as above, we can show that
  $\widehat{G}_{\mathbb{R}}^{(\mathfrak{l}_{\mathbb{R}},\mathfrak{q})}$
 is a locally closed subset of $\widehat{G}_{\mathbb{R}}$.

The set $\widehat{G}_{\mathbb{R}}^{(\mathfrak{l}_{\mathbb{R}},\mathfrak{q})}$
 can be identified with the collection of $\Gamma_{\lambda}$
 with $\lambda\in \sqrt{-1}Z(\mathfrak{l}_{\mathbb{R}})^{*,\mathfrak{q}}_{\rm gr}$.
To see this, suppose that
 $\pi(\mathfrak{l}_{\mathbb{R}},\Gamma_{\lambda})
 \simeq \pi(\mathfrak{l}_{\mathbb{R}},\Gamma'_{\lambda'})$
 for $\lambda,\lambda'\in \sqrt{-1}Z(\mathfrak{l}_{\mathbb{R}})^{*,\mathfrak{q}}_{\rm gr}$,
Then by comparing the infinitesimal characters,
 $\lambda+\rho_{\mathfrak{l}}$ and $\lambda'+\rho_{\mathfrak{l}}$
 lie in the same Weyl group orbit.
By our assumption,
 $\lambda+\rho_{\mathfrak{l}}$ and $\lambda'+\rho_{\mathfrak{l}}$
 satisfy the same dominance condition imposed by $\mathfrak{q}$ and hence $\lambda=\lambda'$.
In view of the Langlands parameters of two representations
 (see the discussion at the end of Section~\ref{sec:quantization_semisimple}),
 we have $\Gamma=\Gamma'$.
Therefore,
 $\widehat{G}_{\mathbb{R}}^{(\mathfrak{l}_{\mathbb{R}},\mathfrak{q})}$
 is identified with the set of $\Gamma_{\lambda}$, or equivalently,
 the map
\[\bigl\{\text{$(\mathcal{O},\Gamma)$ : a semisimple orbital parameter}
 \mid \mathcal{O}\cap \sqrt{-1}Z(\mathfrak{l}_{\mathbb{R}})^{*,\mathfrak{q}}_{\rm gr}
 \neq \emptyset\bigr\}
\to \widehat{G}_{\mathbb{R}}^{(\mathfrak{l}_{\mathbb{R}},\mathfrak{q})}
\] 
given by $(\mathcal{O},\Gamma)\mapsto \pi(\mathcal{O},\Gamma)$ is bijective.

By writing the measure $m_\Pi$ as a finite sum of
 measures supported on
 $\widehat{G}_{\mathbb{R}}^{(\mathfrak{l}_{\mathbb{R}},\mathfrak{q})}$
 for various $\mathfrak{q}$,
 it is enough to prove Lemma~\ref{lem:inclusion1} when
 $m_\Pi$ is a measure on $\widehat{G}_{\mathbb{R}}^{(\mathfrak{l}_{\mathbb{R}},\mathfrak{q})}$
 for one parabolic subalgebra $\mathfrak{q}$.
We thus fix $\mathfrak{q}$ and suppose
 $\Pi$ is a direct integral of 
 representations in $\widehat{G}_{\mathbb{R}}^{(\mathfrak{l}_{\mathbb{R}},\mathfrak{q})}$
 in the rest of this section.

Next, we need to define what it means for a measure $m$ on
 $\widehat{G}_{\mathbb{R}}^{(\mathfrak{l}_{\mathbb{R}},\mathfrak{q})}$
 to be of at most polynomial growth. 
Observe that we have a finite to one map 
\begin{align*}
p\colon \widehat{G}_{\mathbb{R}}^{(\mathfrak{l}_{\mathbb{R}},\mathfrak{q})}
\ni \pi(\mathfrak{l}_{\mathbb{R}},\Gamma_{\lambda})
\mapsto \lambda \in \sqrt{-1}Z(\mathfrak{l}_{\mathbb{R}})^{*,\mathfrak{q}}_{\rm gr}.
\end{align*}
For a Borel measure $m$
 on $\widehat{G}_{\mathbb{R}}^{(\mathfrak{l}_{\mathbb{R}},\mathfrak{q})}$,
 let $p_*m$ denote the pushforward of $m$ under the above map. 
Fix a norm $|\cdot|$ on $\sqrt{-1}Z(\mathfrak{l}_{\mathbb{R}})^*$.  
We say that $m$ is of \emph{at most polynomial growth}
 if there exist a constant $M_0>0$ and a finite measure $m_f$ on
 $\sqrt{-1}Z(\mathfrak{l}_{\mathbb{R}})^{*,\mathfrak{q}}_{\rm gr}$ such that
\begin{equation}\label{eq:at_most_polynomial}
p_*m \leq (1+|\lambda|^2)^{M_0/2}m_f.
\end{equation}
Here, $m\leq m'$ for measures $m$ and $m'$ means that $m(E)\leq m'(E)$ for all measurable sets $E$.

Our proof of Lemma~\ref{lem:inclusion1} involves the Harish-Chandra distribution character of $\pi(\mathcal{O},\Gamma)$.
Let $\Theta(\mathcal{O},\Gamma)$ denote the Harish-Chandra character of the representation $\pi(\mathcal{O},\Gamma)$. Define  the analytic function $j_{G_{\mathbb{R}}}$ utilizing the relation
\[\exp^* (dg)=j_{G_{\mathbb{R}}}(X)dX\]
where $dg$ denotes a nonzero $G_{\mathbb{R}}$-invariant density on $G_{\mathbb{R}}$ and $dX$ denotes a nonzero translation invariant density on $\mathfrak{g}_{\mathbb{R}}$. Normalize $dg$ and $dX$ so that $j_{G_{\mathbb{R}}}(0)=1$, and let $j_{G_{\mathbb{R}}}^{1/2}$ be the unique analytic square root of $j_{G_{\mathbb{R}}}$ with $j_{G_{\mathbb{R}}}^{1/2}(0)=1$. Since $\Theta(\mathcal{O},\Gamma)$ is an analytic function on the subset of regular, semisimple elements in $G_{\mathbb{R}}$, we may define
\[\theta(\mathcal{O},\Gamma):=j_{G_{\mathbb{R}}}^{1/2}(X)\cdot \exp^*\Theta(\mathcal{O},\Gamma)\]
to be the Lie algebra analogue of the character of $\pi(\mathcal{O},\Gamma)$. Note $\theta(\mathcal{O},\Gamma)$ is an analytic function on the collection of regular, semisimple elements
 in $\mathfrak{g}_{\mathbb{R}}$.

Fix a choice of positive roots $\Delta^+(\mathfrak{l},\mathfrak{j})\subset \Delta(\mathfrak{l},\mathfrak{j})$, and define
$\rho_{\mathfrak{l}}:=\frac{1}{2}\sum_{\alpha\in \Delta^+(\mathfrak{l},\mathfrak{j})}\alpha$.
Given a semisimple orbital parameter $(\mathcal{O},\Gamma)$ with 
 $\lambda\in \mathcal{O}\cap \sqrt{-1}Z(\mathfrak{l}_{\mathbb{R}})^{*,\mathfrak{q}}_{\rm gr}\neq \emptyset$, we define a contour
\begin{equation*}
\mathcal{C}(\mathcal{O},\mathfrak{q}):=\left\{g\cdot \lambda+u\cdot \rho_{\mathfrak{l}}\mid g\in G_{\mathbb{R}},\ u\in U,\ \op{Ad}(g)\cdot \mathfrak{q}=\op{Ad}(u)\cdot \mathfrak{q}\right\}
\end{equation*}
in $\mathfrak{g}^*$. 
Here, $\sigma_c$ is an anti-holomorphic involution on $G$
 which commutes with $\sigma$
 such that $U:= G^{\sigma_c}$ is a compact real form of $G$.
For a coadjoint $G$-orbit $\Omega\subset \mathfrak{g}^*$, the Kirillov-Kostant-Souriau $G$-invariant, holomorphic $2$-form $\omega$ on $\Omega$ is defined by
\[\omega_{\xi}(\op{ad}^*(X)(\xi),\op{ad}^*(Y)(\xi)):=\xi([X,Y]).\]
Suppose that $\Omega$ is the regular, coadjoint $G$-orbit through
 $\xi=\lambda+\rho_{\mathfrak{l}}\in \mathfrak{j}^*\subset \mathfrak{g}^*$
 and put $n:=\frac{1}{2}\dim_{\mathbb{C}} \Omega$.
Then $\mathcal{C}(\mathcal{O},\mathfrak{q})$ is a real $2n$-dimensional closed
 submanifold of $\Omega$ (see \cite{HO20}).
Define the $2n$-form
\[
\nu:=\frac{\omega^{\wedge n}}{(2\pi \sqrt{-1})^n n!}.
\]
For a function $\varphi$ on $\mathfrak{g}_{\mathbb{R}}$,
 we define the (inverse) Fourier transform as the following functions
 on $\sqrt{-1}\mathfrak{g}_{\mathbb{R}}^*$:
\begin{align*}
\Hat{\varphi}(\eta):=\int_{\mathfrak{g}_{\mathbb{R}}}
 e^{-\langle \eta, X\rangle} \varphi(X) dX, \quad
\Check{\varphi}(\eta):=\int_{\mathfrak{g}_{\mathbb{R}}}
 e^{\langle \eta, X\rangle} \varphi(X) dX.
\end{align*}
The main result of \cite{HO20} is
\begin{equation}\label{eq:HO_paper}
\langle \theta(\mathcal{O},\Gamma), \varphi\rangle=
 \int_{\mathcal{C}(\mathcal{O},\mathfrak{q})} \Check{\varphi}\, \nu,
\end{equation}
where $\varphi\in C_c^{\infty}(\mathfrak{g}_{\mathbb{R}})$
 is a smooth, compactly supported function on $\mathfrak{g}_{\mathbb{R}}$.
Observe that $\Check{\varphi}$ extends to a holomorphic function
 on $\mathfrak{g}^*$. 
We remark that for any semisimple orbit $\mathcal{O}$ with
 $\mathcal{O}\cap \sqrt{-1}Z(\mathfrak{l}_{\mathbb{R}})^*_{\rm gr} \neq \emptyset$, 
 the contour $\mathcal{C}(\mathcal{O},\mathfrak{q})$
 and the forms $\omega, \nu$ are defined in the same way,
 even if it does not come from a semisimple orbital parameter
 $(\mathcal{O},\Gamma)$.

Fix a $K_{\mathbb{R}}$-invariant norm $|\cdot|$ on $\mathfrak{g}_{\mathbb{R}}^*:=\Hom_{\mathbb{R}}(\mathfrak{g}_{\mathbb{R}},\mathbb{R})$. If $\eta\in \mathfrak{g}^*:=\Hom_{\mathbb{C}}(\mathfrak{g},\mathbb{C})$, write 
\[\eta=\RE\eta+\sqrt{-1}\;\! \IM\eta\]
where $\RE\eta,\, \IM\eta \in \mathfrak{g}_{\mathbb{R}}^*$.
Extend $|\cdot|$ to a norm on $\mathfrak{g}^*$ by defining $|\eta|^2=|\RE \eta|^2+|\IM \eta|^2$.

Fix $d\in \mathbb{R}$ such that 
 $d>\max_{\alpha\in\Delta(\mathfrak{g},\mathfrak{j})}|
 \langle\rho_{\mathfrak{l}},\alpha^\vee\rangle|$.
Writing $m_{\Pi}$ as a sum of two measures according to the decomposition
\[\widehat{G}_{\mathbb{R}}^{(\mathfrak{l}_{\mathbb{R}},\mathfrak{q})}
=\bigl(\widehat{G}_{\mathbb{R}}^{(\mathfrak{l}_{\mathbb{R}},\mathfrak{q})}
 \cap\widehat{G}_{\mathbb{R}}(\mathfrak{l},d)\bigr) \cup
\bigl(\widehat{G}_{\mathbb{R}}^{(\mathfrak{l}_{\mathbb{R}},\mathfrak{q})}
 \setminus\widehat{G}_{\mathbb{R}}(\mathfrak{l},d)\bigr)
\]
and using Lemma~\ref{lem:singular_spectrum},
 it is enough to show Lemma~\ref{lem:inclusion1} when
 $\op{supp} m_\Pi \cap\widehat{G}_{\mathbb{R}}(\mathfrak{l},d)=\emptyset$.
This assumption makes it easier for us to estimate the integral \eqref{eq:HO_paper}
 as we see below.

\begin{lemma}\label{lem:SS}
Suppose that $m$ is a measure on
 $\widehat{G}_{\mathbb{R}}^{(\mathfrak{l}_{\mathbb{R}},\mathfrak{q})}$
 with at most polynomial growth
 and $\op{supp} m \cap\widehat{G}_{\mathbb{R}}(\mathfrak{l},d)=\emptyset$.
\begin{enumerate}[{\rm (i)}]
\item \label{SSpart1} 
Let $\alpha$ be a function on $\mathfrak{g}^*$, and assume $\alpha|_{\mathcal{C}(\mathcal{O},\mathfrak{q})}$ is measurable for all coadjoint orbits $\mathcal{O}$ with
 $\mathcal{O}\cap \sqrt{-1}Z(\mathfrak{l}_{\mathbb{R}})_{\rm gr}^{*,\mathfrak{q}}
 \neq \emptyset$. 
Assume that for every $N\in \mathbb{N}$ and every $b>0$ there exist constants $C_{N,b}>0$ such that
\begin{equation}\label{eq:alpha_bound}
\sup_{\substack{\eta\in \mathfrak{g}^*\\ |\RE\eta|\leq b}}(1+|\IM\eta|^2)^{N/2}|\alpha(\eta)|\leq C_{N,b}.
\end{equation}
Then the integral 
\begin{equation}\label{eq:integral}
\langle C(m),\alpha\rangle:=
 \int_{\pi(\mathcal{O},\Gamma)
 \in \widehat{G}_{\mathbb{R}}^{(\mathfrak{l}_{\mathbb{R}},\mathfrak{q})}}
 \left(\int_{\mathcal{C}(\mathcal{O},\mathfrak{q})}\alpha(\eta)\nu\right)dm
\end{equation}
converges absolutely.
\item \label{SSpart2}
If $\varphi\in C_c^{\infty}(\mathfrak{g}_{\mathbb{R}})$,
 then the integral
\begin{equation}\label{eq:Fourier_integral}
\langle C(m), \Check{\varphi} \rangle
 :=\int_{\pi(\mathcal{O},\Gamma)
 \in \widehat{G}_{\mathbb{R}}^{(\mathfrak{l}_{\mathbb{R}},\mathfrak{q})}}
 \left(\int_{\mathcal{C}(\mathcal{O},\mathfrak{q})}\Check{\varphi}\,\nu\right)dm
\end{equation}
converges absolutely.
The functional $\varphi\mapsto \langle C(m), \Check{\varphi} \rangle$
 is a well-defined distribution on $\mathfrak{g}_{\mathbb{R}}$,
 which is the integral  
\[\theta(m):=\int_{\pi(\mathcal{O},\Gamma)
 \in \widehat{G}_{\mathbb{R}}^{(\mathfrak{l}_{\mathbb{R}},\mathfrak{q})}}
\theta(\mathcal{O},\Gamma)dm.\]
\item \label{SSpart3}
For $\varphi\in C_c^{\infty}(\mathfrak{g}_{\mathbb{R}})$,
 the Fourier transform of $\theta(m)\varphi$ is given by 
\begin{equation*}
(\theta(m)\varphi)\sphat\, (\xi)
 = \langle C(m)_{\eta}, \Check{\varphi}(\eta-\xi)\rangle.
\end{equation*}
It is a smooth, polynomially bounded function
 on $\sqrt{-1}\mathfrak{g}_{\mathbb{R}}^*$.
\item \label{SSpart4}
We have
\begin{equation}\label{eq:SS}
\op{SS}_0\bigl(\theta(m)\bigr)\subset \op{AC}\Biggl(\bigcup_{\pi(\mathcal{O},\Gamma)\in \op{supp}m} \mathcal{O}\Biggr).
\end{equation}
\end{enumerate}
\end{lemma}

To prove part \eqref{SSpart1}, we need another lemma.
Define
\[\Lambda:=\bigl\{\lambda\in \sqrt{-1}Z(\mathfrak{l}_{\mathbb{R}})^{*} :
 |\langle \lambda+\rho_{\mathfrak{l}}, \alpha^{\vee} \rangle| \geq  d
 \ \bigl(  \forall \alpha\in
 \Delta(\mathfrak{g},\mathfrak{j})\setminus \Delta(\mathfrak{l},\mathfrak{j})\bigr)\bigr\}. \]
Then $\Lambda$ is a closed subset of
 $\sqrt{-1}Z(\mathfrak{l}_{\mathbb{R}})^{*}_{\rm gr}$.

\begin{lemma}\label{lem:integral_semialgebraic}
For any $M>0$, 
 there exist constants $k_M, C_M>0$ such that
\[
\int_{\eta\in\mathcal{C}(\mathcal{O}_{\lambda},\mathfrak{q})}
 (1+|\IM\eta|^2)^{-k_M/2} |\nu|
\leq C_M (1+|\lambda|^2)^{-M/2}
\]
for $\lambda\in\Lambda$.
Here, we write
 $\mathcal{O}_{\lambda}:= G_{\mathbb{R}}\cdot \lambda$.
\end{lemma}

\begin{proof}[proof of Lemma~\ref{lem:integral_semialgebraic}]
Fix any $\lambda_0\in \Lambda$.
The Euclidean metric on $\mathfrak{g}^*$ induces 
 a Riemannian metric on the submanifold
 $\mathcal{C}(\mathcal{O}_{\lambda_0},\mathfrak{q})$.
Let $\nu_E$ be the volume form of this Riemannian manifold
 $\mathcal{C}(\mathcal{O}_{\lambda_0},\mathfrak{q})$.

We first claim that
\begin{align}\label{integral_Euclid}
\int_{\xi\in \mathcal{C}(\mathcal{O}_{\lambda_0},\mathfrak{q})}
 (1+|\xi|^2)^{-N/2} \nu_E < \infty
\end{align}
for sufficiently large $N>0$.
To see this, we use an argument similar to \cite[(3.14)]{SV98}.
Consider the one point compactification of $\mathfrak{g}$, which is a sphere $S$. 
Let $\overline{\mathcal{C}(\mathcal{O}_{\lambda_0},\mathfrak{q})}$
 be the closure of $\mathcal{C}(\mathcal{O}_{\lambda_0},\mathfrak{q})$ in $S$.
With respect to a standard metric on the sphere $S$, its compact semialgebraic subset
 $\overline{\mathcal{C}(\mathcal{O}_{\lambda_0},\mathfrak{q})}$ has finite volume
 (see e.g.\ \cite{OS17}).
By comparing the standard metric on $S$ and the Euclidean metric on $\mathfrak{g}$,
 this can be restated as \eqref{integral_Euclid} for $N\geq 4n$.

Next, define a semialgebraic set
\[
\mathcal{C}(\Lambda):=\{(\lambda,\eta)\in\Lambda\times \mathfrak{g}^*:
 \eta\in \mathcal{C}(\mathcal{O}_{\lambda},\mathfrak{q})\}.
\]
For each $\lambda\in \Lambda$, there is an isomorphism
\[i\colon \mathcal{C}(\mathcal{O}_{\lambda},\mathfrak{q})
\xrightarrow{\sim}
\mathcal{C}(\mathcal{O}_{\lambda_0},\mathfrak{q}),
\quad g\cdot \lambda+u\cdot \rho_{\mathfrak{l}}
\mapsto g\cdot \lambda_0+u\cdot \rho_{\mathfrak{l}}.
\]
Define the semialgebraic functions
 $f((\lambda,\eta)):= 1+|\IM \eta|^2$ on $\mathcal{C}(\Lambda)$.
In the following, we will compare some other semialgebraic functions
 on $\mathcal{C}(\Lambda)$ with $f$.
On $\mathcal{C}(\mathcal{O}_{\lambda},\mathfrak{q})$,
 we have two volume forms $i^*\nu_E$ and $|\nu|$.
Define $h:=\frac{|\nu|}{i^*\nu_E}$,
 which is a semialgebraic function on $\mathcal{C}(\Lambda)$.
It is easy to see that
 the set $\{(\lambda,\eta) : |f(\lambda,\eta)|\leq t\}$ is compact for any $t>0$.
Then the function
\[
\overline{h}(t)=\sup \{ h((\lambda,\eta)) : |f(\lambda,\eta)|\leq t \}
\]
is defined for large $t>0$ and is semialgebraic.
By \cite[Theorem A.2.5]{Hor83b},
 $\overline{h}(t)\leq A_1\cdot t^{N_1}$ for some constants $A_1, N_1>0$.
Hence we get
\begin{align}\label{estimate_h}
h((\lambda,\eta))\leq A_1\cdot (1+|\IM \eta|^2)^{N_1}
\end{align}
for $(\lambda,\eta)\in\mathcal{C}(\Lambda)$.

The functions $(\lambda,\eta)\mapsto 1+|i(\eta)|^2$
 and $(\lambda,\eta)\mapsto 1+|\lambda|^2$ are also semialgebraic on $\mathcal{C}(\Lambda)$.
Hence we similarly have 
\begin{align}\label{estimate_eta}
1+|i(\eta)|^2 \leq A_2\cdot (1+|\IM \eta|^2)^{N_2}, \quad 
1+|\lambda|^2 \leq A_3\cdot (1+|\IM \eta|^2)^{N_3}
\end{align}
for some constants $A_2,N_2, A_3,N_3>0$.
The lemma follows from 
 an isomorphism $i\colon \mathcal{C}(\mathcal{O}_{\lambda_0},\mathfrak{q})
 \simeq \mathcal{C}(\mathcal{O}_{\lambda},\mathfrak{q})$
 and the estimates \eqref{integral_Euclid}, \eqref{estimate_h} and \eqref{estimate_eta}.
\end{proof}

\begin{proof}[proof of Lemma~\ref{lem:SS}] 
Since $m$ is of at most polynomial growth, $p_*m\leq (1+|\lambda|^2)^{M_0/2}m_f$
 for a finite measure $m_f$ and a constant $M_0>0$.
By our assumption on $m$, we may assume that $\op{supp} m_f$ is contained in $\Lambda$.
In addition, $|\RE \eta|$ for $\eta\in \mathcal{C}(\mathcal{O},\mathfrak{q})$
 is bounded by a constant.
Hence the absolute convergence of \eqref{eq:integral} follows from
 Lemma~\ref{lem:integral_semialgebraic} and \eqref{eq:alpha_bound}.

To prove part \eqref{SSpart2}, recall that for $\varphi\in C_c^{\infty}(\mathfrak{g}_{\mathbb{R}})$, the Paley-Wiener Theorem assures us that there exists a constant $B>0$ and for every $N\in \mathbb{N}$, there exists a constant $A_N>0$ such that
\[\left|\Check{\varphi}(\eta)\right|\leq \frac{A_Ne^{B |\RE\eta|}}{(1+|\IM\eta|^2)^{N/2}}.\]
Hence, we may plug in $\Check{\varphi}$ for $\alpha$ and the absolute convergence of \eqref{eq:Fourier_integral} follows from part \eqref{SSpart1}. 
Further, the constants that bound this integral can be shown to be bounded by seminorms on the space of smooth compactly supported densities
 on $\mathfrak{g}_{\mathbb{R}}$. 
Therefore, the integral the $\theta(m)$ defined in part \eqref{SSpart2} is given as
 a well-defined distribution
\[
\varphi\mapsto \langle C(m), \Check{\varphi}\rangle.
\]
By \eqref{eq:HO_paper}, this is the integral of $\theta(\mathcal{O},\Gamma)$.

Next, we prove part \eqref{SSpart3}. 
Let $\varphi\in C_c^{\infty}(\mathfrak{g}_{\mathbb{R}})$. 
Then $\theta(m)\varphi$ is a distribution with compact support.
Hence the Fourier transform $(\theta(m)\varphi)\sphat$ is
 a smooth, polynomially bounded function on $\sqrt{-1}\mathfrak{g}_{\mathbb{R}}^*$. 
The value of $(\theta(m)\varphi)\sphat$ at
 $\xi\in \sqrt{-1}\mathfrak{g}_{\mathbb{R}}^*$ is given as
\begin{align*}
(\theta(m)\varphi)\sphat\, (\xi)
=\langle \theta(m)\varphi, e^{-\langle \xi, \cdot \rangle}\rangle
=\langle \theta(m), e^{-\langle\xi,\cdot\rangle} \varphi \rangle
&=\langle C(m), (e^{-\langle\xi,\cdot\rangle} \varphi)\spcheck \rangle \\
&=\langle C(m)_{\eta}, \Check{\varphi}(\eta-\xi) \rangle.
\end{align*}
Thus, \eqref{SSpart3} is proved.

For part \eqref{SSpart4}, we require some additional notation. Choose a basis $\{X_1,\ldots,X_n\}$ of $\mathfrak{g}$, and define the differential operator
\[D^{\alpha}:=\partial_{X_1}^{\alpha_1}\cdots \partial_{X_n}^{\alpha_n}\]
for every multi-index $\alpha=(\alpha_1,\ldots,\alpha_N)\in \mathbb{N}^n$. In addition, define $|\alpha|=\alpha_1+\cdots+\alpha_n$.  If $0\in \mathcal{U}_1\subset \mathcal{U}_2\subset \mathfrak{g}$ are precompact, open subsets of $\mathfrak{g}$ with $\overline{\mathcal{U}_1}\subset \mathcal{U}_2$, then there exists a sequence $\{\varphi_{N,\mathcal{U}_1,\mathcal{U}_2}\}$ of functions indexed by $N\in \mathbb{N}$ and satisfying the following properties (see pages 25--26, 282 of \cite{Hor83a}):
\begin{enumerate}
\item $\varphi_{N,\mathcal{U}_1,\mathcal{U}_2}\in C_c^{\infty}(\mathcal{U}_2)$ for all $N\in \mathbb{N}$
\item $\varphi_{N,\mathcal{U}_1,\mathcal{U}_2}(x)=1$ if $x\in \mathcal{U}_1$
\item There exists a constant $C_{\alpha}>0$ for every multi-index $\alpha\in \mathbb{N}^n$ such that
\begin{equation*}
\sup_{x\in \mathcal{U}_2}|(D^{\alpha+\beta}\varphi_{N,\mathcal{U}_1,\mathcal{U}_2})(x)|\leq C_{\alpha}^{|\beta|+1}(N+1)^{|\beta|}
\end{equation*}
for every multi-index $\beta\in \mathbb{N}^n$ with $|\beta|\leq N$.
\end{enumerate}
For the sequel, we fix $\mathcal{U}_1$, $\mathcal{U}_2$, and 
 take a sequence of functions $\{\varphi_{N,\mathcal{U}_1,\mathcal{U}_2}\}$
 satisfying (1)--(3).
Write $\varphi_{N}:=\varphi_{N,\mathcal{U}_1,\mathcal{U}_2}$.

Fix 
\[\xi\notin \op{AC}\left(\bigcup_{\pi(\mathcal{O},\Gamma)
\in \op{supp}m} \mathcal{O}\right).\]
In order to prove \eqref{eq:SS},
 it is enough to show the following by \cite[\S 8.4]{Hor83a}: 
 there exists an open subset
 $\xi\in W\subset \sqrt{-1}\mathfrak{g}_{\mathbb{R}}^*$
 and a constant $C>0$ such that
\begin{equation}\label{eq:bound_C}
\left| \langle C(m)_{\eta},\, 
 \Check{\varphi}_{N}(\eta-t\xi') \rangle
 \right|\leq C^{N+1}\frac{(N+1)^N}{(1+t^2)^{N/2}}
\end{equation}
for all $\xi'\in W$ and $t>0$. 

Choose an open cone $\Psi\subset \sqrt{-1}\mathfrak{g}_{\mathbb{R}}^*$ such that
\[\xi\in \Psi\subset \overline{\Psi}\setminus \{0\}\subset \sqrt{-1}\mathfrak{g}_{\mathbb{R}}^*\setminus \op{AC}\left(\bigcup_{\pi(\mathcal{O},\Gamma)\in \op{supp}\mu} \mathcal{O}\right),\]
and define 
\[W:=\left\{\xi'\in \Psi \, \Bigl|\, \frac{|\xi|}{2}< |\xi'|< 2|\xi|\right\}.\]
We require a lemma.
\begin{lemma}\label{lem:cone}
There exist constants $D,\epsilon,\epsilon'>0$ such that
\begin{equation}\label{eq:cone}
|\sqrt{-1}\IM\eta-t\xi'|\geq \epsilon t
\end{equation}
and
\begin{equation}\label{eq:cone_2}
|\sqrt{-1}\IM\eta-t\xi'|\geq \epsilon' |\IM\eta|
\end{equation}
if 
\[\eta\in \bigcup_{\pi(\mathcal{O},\Gamma)\in \op{supp}m}
 \mathcal{C}(\mathcal{O},\mathfrak{q}),\ \ \ \xi'\in \overline{W},
 \ \ \text{and}\ \ t>D.\]
\end{lemma}
\begin{proof}[proof of Lemma~\ref{lem:cone}] 
Assume that \eqref{eq:cone} does not hold. 
Then we may find sequences $\{\xi_j\}\subset \overline{W}$, 
 $\{t_j\}\subset \mathbb{R}_{>0}$, and $\{\eta_j\}$
 with $\eta_j\in \mathcal{C}(\mathcal{O}_j,\mathfrak{q})$
 satisfying $\pi(\mathcal{O}_j,\Gamma_j)\in \op{supp}m$
 such that $|t_j\xi_j-\sqrt{-1}\IM\eta_j|<\frac{t_j}{j}$ and $t_j>j$. 
Further, we may write $\eta_j=\eta_j'+\eta_j''$
 where $\eta_j'\in \mathcal{O}_j$
 and $\eta_j''\in U\cdot \rho_{\mathfrak{l}}$. 
Since $\RE \eta_j$ and $\eta_j''$ are bounded,
$|t_j\xi_j-\eta_j'|<\frac{t_j}{j}+a$ for a constant $a>0$.
But, then $\{\eta_j'/t_j\}$ has a convergent subsequence
 which must therefore lie in both $\overline{W}$
 and $\op{AC}\left(\bigcup_{\pi(\mathcal{O},\Gamma)\in \op{supp}m}\mathcal{O}\right)$,
 which is a contradiction. 
This implies \eqref{eq:cone}.

Next, we utilize the triangle inequality to obtain
\[|t\xi'|\geq |\IM\eta|-|\sqrt{-1}\IM\eta-t\xi'|.\]
Combining with \eqref{eq:cone} yields
\[|\sqrt{-1}\IM\eta-t\xi'|\geq \epsilon t\geq
 \frac{\epsilon}{|\xi'|}\left(|\IM\eta|-|\sqrt{-1}\IM\eta-t\xi'|\right).\]
Recall $|\xi'|\leq 2|\xi|$, collect the $|\sqrt{-1}\IM\eta-t\xi'|$
 terms on one side of the equation, and put
$\epsilon':=(1+\frac{\epsilon}{2|\xi|})^{-1}\frac{\epsilon}{2|\xi|}$.
Then \eqref{eq:cone_2} follows.
\end{proof}

In order to prove \eqref{eq:bound_C}, for each $M>0$,
 we will first show the existence of
 a constant $C_{M}>0$ such that
\begin{equation}\label{eq:bound_C2}
\left|\int_{\eta\in \mathcal{C}(\mathcal{O}_{\lambda},\mathfrak{q})}
 \Check{\varphi}_{N} (\eta-t\xi')\, \nu \right|
\leq \frac{C_{M}^{N+1}}{(1+|\lambda|^2)^{M/2}}
 \frac{(N+1)^N}{(1+t^2)^{N/2}}
\end{equation}
for all $\xi'\in W$, for all $\pi(\mathcal{O},\Gamma)\in \op{supp} m$,
 for all $N\in \mathbb{N}$, and for $t>0$. 
In order to prove \eqref{eq:bound_C2}, we need an estimate of $\Check{\varphi}_{N}$.
By the proof of the Paley-Wiener Theorem
 (see for instance page 181 of \cite{Hor83a})
 and part (3) of the definition of $\{\varphi_{N,\mathcal{U}_1,\mathcal{U}_2}\}$,
 there exist constants $B,C'>0$ such that
\[|\Check{\varphi}_{N}(\eta)|\leq
 \frac{(C')^{N+1}(N+1)^Ne^{B\cdot|\RE{\eta}|}}{(1+|\IM\eta|^2)^{N/2}}.\]
Using that $|\RE\eta|$ is bounded by a constant $a>0$
 for $\eta\in \mathcal{C}(\mathcal{O},\mathfrak{q})$, 
 and putting $C:=C'e^{B\cdot a}$, we deduce
\begin{equation}\label{eq:Paley_Wiener_4}
|\Check{\varphi}_{N}(\eta-t\xi')|
 \leq \frac{C^{N+1}(N+1)^N}{(1+|\sqrt{-1}\IM\eta-t\xi'|^2)^{N/2}}
\end{equation}
whenever $\xi'\in\sqrt{-1}\mathfrak{g}_{\mathbb{R}}^*$ and 
 $\eta\in \mathcal{C}(\mathcal{O}_{\lambda},\mathfrak{q})$
 with $\lambda\in \sqrt{-1}Z(\mathfrak{l}_{\mathbb{R}})^{*,\mathfrak{q}}_{\rm gr}$. 
For fixed $M$, define
\[\varphi^M_{N}:=\varphi_{N+k_M},\]
where $k_M$ is the constant in Lemma~\ref{lem:integral_semialgebraic}.
Observe that for every $M$, the sequence $\varphi^M_{N}$ still satisfies the properties (1)--(3). 
Therefore, in order to verify part \eqref{SSpart4},
 we may replace $\varphi_{N}$ with $\varphi^M_{N}$. 
Utilizing Lemma~\ref{lem:integral_semialgebraic} and \eqref{eq:Paley_Wiener_4},
 we obtain for all $M,N\in \mathbb{N}$ and $\pi(\mathcal{O}_{\lambda},\Gamma)\in \op{supp} m$,
\begin{align}\label{eq:bound_C3}
&\phantom{=} \left|\int_{\eta\in \mathcal{C}(\mathcal{O}_{\lambda},\mathfrak{q})}
 (\varphi^M_{N})\spcheck(\eta-t\xi')\,\nu \right| \nonumber\\
&\leq  \frac{C_M}{(1+|\lambda|^2)^{M/2}}
 \sup_{\eta\in \mathcal{C}(\mathcal{O}_{\lambda},\mathfrak{q})}
 (1+|\IM\eta|^2)^{k_M/2}|(\varphi^M_{N})\spcheck(\eta-t\xi')| \nonumber\\
&\leq \frac{C_M}{(1+|\lambda|^2)^{M/2}}
 \sup_{\eta\in \mathcal{C}(\mathcal{O}_{\lambda},\mathfrak{q})}
 \frac{C^{N+k_M+1}(N+k_M+1)^{N+k_M}}{(1+|\sqrt{-1}\IM\eta-t\xi'|^2)^{(N+k_M)/2}}
 \cdot (1+|\IM\eta|^2)^{k_M/2}
\end{align}
for some constant $C_M>0$. 
For fixed $M$, if $N$ is sufficiently large, we have
\begin{equation}\label{eq:s^t}
\begin{split}
(N+k_M+1)^{N+k_M} &= (N+k_M+1)^N(N+k_M+1)^{k_M} \\
&\leq 2^N (N+1)^N k_M^{N+k_M+1}
\end{split}
\end{equation}
where we have used that $t^s>s^t$ for $s>t\geq 3$. For every fixed $M\in \mathbb{N}$ and sufficiently large $N$, we may utilize \eqref{eq:cone}, \eqref{eq:cone_2}
 and \eqref{eq:s^t} to bound \eqref{eq:bound_C3} by
\begin{align*}
&\leq \frac{C_{M}^{N+1}(N+1)^{N}}{(1+|\lambda|^2)^{M/2}} \cdot \frac{(1+|\IM\eta|^2)^{k_M/2}}{(1+(\epsilon'|\IM\eta|)^2)^{k_M/2}}\cdot \frac{1}{(1+(\epsilon t)^2)^{N/2}}\\
&\leq \frac{C_{M}^{N+1}(N+1)^N}{(1+|\lambda|^2)^{M/2}} \cdot \frac{1}{(1+t^2)^{N/2}}
\end{align*}
for the constant $C_{M}>0$ which we increased in each line. 
Thus, \eqref{eq:bound_C2} is proved.
Then
\begin{align*}
&\phantom{=} \left|\int_{\pi(\mathcal{O},\Gamma)\in
  \widehat{G}_{\mathbb{R}}^{(\mathfrak{l}_{\mathbb{R}},\mathfrak{q})}}
 \int_{\eta\in \mathcal{C}(\mathcal{O},\mathfrak{q})}
 (\varphi_{N}^M)\spcheck (\eta-t\xi')\,\nu dm\right|\\
&\leq  \frac{C_{M}^{N+1}(N+1)^N}{(1+t^2)^{N/2}}
 \int_{\pi(\mathcal{O}_{\lambda},\Gamma)\in
 \widehat{G}_{\mathbb{R}}^{(\mathfrak{l}_{\mathbb{R}},\mathfrak{q})}}
 \frac{1}{(1+|\lambda|^2)^{M/2}}dm\\
&\leq  \frac{C_{M}^{N+1}(N+1)^N}{(1+t^2)^{N/2}}
 \int_{\sqrt{-1}Z(\mathfrak{l}_{\mathbb{R}})^{*,\mathfrak{q}}_{\rm gr}}
 \frac{1}{(1+|\lambda|^2)^{M/2}} (1+|\lambda|^2)^{M_0/2}dm_f\\
&\leq  \frac{C_{M}^{N+1}(N+1)^N}{(1+t^2)^{N/2}}
 \int_{\sqrt{-1}Z(\mathfrak{l}_{\mathbb{R}})^{*,\mathfrak{q}}_{\rm gr}}
 \frac{1}{(1+|\lambda|^2)^{(M-M_0)/2}}dm_f.
\end{align*}
where we have increased the constant $C_{M}$ as necessary throughout the calculation. Since the final integral converges if $M\geq M_0$, we may absorb the value of
 the integral into the constant $C_{M}$ to bound the entire expression by
\[\leq \frac{C_{M}^{N+1}(N+1)^N}{(1+t^2)^{N/2}}.\]
Part \eqref{SSpart4} follows.
\end{proof}

The proof of Lemma \ref{lem:inclusion1} now proceeds exactly line by line the same as the proof of
 \cite[Proposition 7.1]{HHO16}
 except one must substitute \eqref{eq:SS} in for (7.1) of \cite{HHO16}.
For this argument, we only need \eqref{eq:SS} for a finite measure $m$.
Lemma~\ref{lem:SS} was stated more generally for a measure with at most polynomial growth
 because it will be necessary in the next section.


\section{Wave front sets of direct integrals for a Levi, part 2}\label{sec:WF2}

We retain the notation of the previous section.
The purpose of this section is to
 prove the following lemma using Lemma~\ref{lem:diff_form1}
 and Lemma~\ref{lem:form_lower_bound}.
The proof of these lemmas will be postponed in the next section.

\begin{lemma}\label{lem:inclusion2} 
In the setting of Theorem~\ref{thm:direct_integral},
\begin{equation}\label{eq:inclusion2}
\op{WF}(\Pi)\supset 
\op{AC}\Biggl(\bigcup_{
 \pi(\mathfrak{l}_{\mathbb{R}},\Gamma_{\lambda})\in \op{supp}m_{\Pi}}
 G_{\mathbb{R}}\cdot \lambda\Biggr)
\cap (G\cdot Z(\mathfrak{l})^*_{\rm reg}).
\end{equation}
\end{lemma}

Lemma \ref{lem:inclusion1} and Lemma \ref{lem:inclusion2} combine to imply Theorem \ref{thm:direct_integral} since $\op{WF}(\Pi)\subset \op{SS}(\Pi)$ for any unitary representation $\Pi$ of $G_{\mathbb{R}}$.

We first show the following:
\begin{lemma}\label{lem:asymptotic_cone}
For any subset $S\subset \sqrt{-1}Z(\mathfrak{l}_{\mathbb{R}})^*$,
\begin{equation}\label{eq:asymptotic_cone}
\op{AC}(G_{\mathbb{R}}\cdot S) \cap (G\cdot Z(\mathfrak{l})^*_{\rm reg})
= G_{\mathbb{R}}\cdot
 (\op{AC}(S) \cap \sqrt{-1}Z(\mathfrak{l}_{\mathbb{R}})^*_{\rm reg}).
\end{equation}
In addition, if $(G_{\mathbb{R}}\cdot S)\cap\sqrt{-1}Z(\mathfrak{l}_{\mathbb{R}})^* =S$ holds,
 then
\begin{equation}\label{eq:asymptotic_cone2}
\op{AC}(G_{\mathbb{R}}\cdot S) \cap \sqrt{-1}Z(\mathfrak{l}_{\mathbb{R}})^*_{\rm reg}
= \op{AC}(S) \cap \sqrt{-1}Z(\mathfrak{l}_{\mathbb{R}})^*_{\rm reg}.
\end{equation}
\end{lemma}

\begin{proof}
Since $\op{AC}(G_{\mathbb{R}}\cdot S)\supset \op{AC}(S)$
 and $\op{AC}(G_{\mathbb{R}}\cdot S)$ is
 $G_{\mathbb{R}}$-stable,
 we have $\op{AC}(G_{\mathbb{R}}\cdot S)\supset
 G_{\mathbb{R}}\cdot \op{AC}(S)$.
The inclusion 
\begin{equation*}
\op{AC}(G_{\mathbb{R}}\cdot S) \cap (G\cdot Z(\mathfrak{l})^*_{\rm reg})
\supset G_{\mathbb{R}}\cdot 
 (\op{AC}(S) \cap \sqrt{-1}Z(\mathfrak{l}_{\mathbb{R}})^*_{\rm reg})
\end{equation*}
then follows from 
$G\cdot Z(\mathfrak{l})^*_{\rm reg}\supset
 G_{\mathbb{R}}\cdot \sqrt{-1}Z(\mathfrak{l}_{\mathbb{R}})^*_{\rm reg}$.

To prove the other inclusion, take a vector $\xi$ 
 in the left hand side of \eqref{eq:asymptotic_cone}.
Then in particular
 $\xi\in\sqrt{-1}\mathfrak{g}_{\mathbb{R}}^*\cap
 (G\cdot Z(\mathfrak{l})^*_{\rm reg})$.
Therefore, if $\mathfrak{l}'_{\mathbb{R}}:=\mathfrak{g}_{\mathbb{R}}(\xi)$,
 then $\mathfrak{l}'$ is $G$-conjugate to $\mathfrak{l}$.
Consider the map
\[
a\colon G_{\mathbb{R}}\times \sqrt{-1}(\mathfrak{l}'_{\mathbb{R}})^*
\to \sqrt{-1}\mathfrak{g}_{\mathbb{R}}^*
\]
given by $(g,\eta)\mapsto g\cdot\eta$.
Identify $\sqrt{-1}(\mathfrak{l}'_{\mathbb{R}})^{*}\simeq \mathfrak{l}'_{\mathbb{R}}$
 in an $L_{\mathbb{R}}$-invariant way and define 
\[\sqrt{-1}(\mathfrak{l}'_{\mathbb{R}})^{*,o}
:=\{\eta\in \sqrt{-1}(\mathfrak{l}'_{\mathbb{R}})^*
 \simeq \mathfrak{l}'_{\mathbb{R}} \mid
 \det(\op{ad}(\eta)|_{\mathfrak{g}/\mathfrak{l}'})\neq 0\}.
\]
Then $a$ is submersive on the open set
 $G_{\mathbb{R}}\times \sqrt{-1}(\mathfrak{l}'_{\mathbb{R}})^{*,o}$.
We see that $\xi\in \sqrt{-1}(\mathfrak{l}'_{\mathbb{R}})^{*,o}$.
Take an open cone $C\subset \sqrt{-1}(\mathfrak{l}'_{\mathbb{R}})^{*,o}$ containing $\xi$
 and take a small neighborhood $e\in V\subset G_{\mathbb{R}}$.
Then $V\cdot C$ is an open cone in $\sqrt{-1}\mathfrak{g}_{\mathbb{R}}^*$
 containing $\xi$.
By $\xi\in \op{AC}(G_{\mathbb{R}}\cdot S)$
 and the definition of the asymptotic cone, 
\begin{equation*}
(G_{\mathbb{R}}\cdot S)\cap (V\cdot C)
 \text{ is unbounded.}
\end{equation*}
Since 
$(G_{\mathbb{R}}\cdot S)\cap (V\cdot C)
\supset V \cdot ((G_{\mathbb{R}}\cdot S) \cap C)$
 and $V$ is bounded, 
\begin{equation*}
(G_{\mathbb{R}}\cdot S)\cap C
 \text{ is unbounded.}
\end{equation*}
Hence there exists $\lambda\in S\subset \sqrt{-1}Z(\mathfrak{l}_{\mathbb{R}})^*_{\rm reg}$
 and $g\in G_{\mathbb{R}}$ such that
 such that $g\cdot \lambda\in \sqrt{-1}(\mathfrak{l}'_{\mathbb{R}})^{*,o}$.
Since $\eta\in \sqrt{-1}(\mathfrak{l}'_{\mathbb{R}})^{*,o}$ implies 
 $\mathfrak{g}(\eta)\supset \mathfrak{l}'$, we have
 $g\cdot \mathfrak{l}\supset \mathfrak{l}'$.
Combining with $\mathfrak{l}\sim \mathfrak{l}'$, we have
  $g\cdot \mathfrak{l}_{\mathbb{R}}=\mathfrak{l}'_{\mathbb{R}}$.

Replacing $\xi$ by $g^{-1} \cdot \xi$, we have
 $\xi\in \sqrt{-1}Z(\mathfrak{l}_{\mathbb{R}})^*_{\rm reg}$
 and $\mathfrak{l}_{\mathbb{R}}=\mathfrak{l}'_{\mathbb{R}}$.
Take a Cartan subalgebra
 $\mathfrak{j}_{\mathbb{R}}\subset \mathfrak{l}_{\mathbb{R}}$.
If two elements in
 $\sqrt{-1}Z(\mathfrak{l}_{\mathbb{R}})^*
 (\subset \sqrt{-1}\mathfrak{j}_{\mathbb{R}}^*)$
 are $G_{\mathbb{R}}$-conjugate, they lie in the same orbit for the Weyl group
 $W_{\mathbb{R}}=N_{G_{\mathbb{R}}}(\mathfrak{j}_{\mathbb{R}})
 /Z_{G_{\mathbb{R}}}(\mathfrak{j}_{\mathbb{R}})$.
Hence 
\[(G_{\mathbb{R}}\cdot S)\cap C
=(W_{\mathbb{R}}\cdot S)\cap C.\]
Since $W_{\mathbb{R}}$ is finite, 
 there exists $w\in W_{\mathbb{R}}$ such that
 $(w\cdot S) \cap C$, or equivalently, $S\cap (w^{-1}\cdot C)$ is unbounded
 for any $C$.
This shows $w^{-1}\cdot \xi \in \op{AC}(S)$
 and hence $\xi\in G_{\mathbb{R}}\cdot \op{AC}(S)$,
 which implies the desired inclusion in \eqref{eq:asymptotic_cone}.

To prove \eqref{eq:asymptotic_cone2},  take a vector 
 $\xi\in \op{AC}(G_{\mathbb{R}}\cdot S) \cap \sqrt{-1}Z(\mathfrak{l}_{\mathbb{R}})^*_{\rm reg}$.
Then by \eqref{eq:asymptotic_cone}, we may write $\xi=g\cdot \xi'$
 such that $g\in G_{\mathbb{R}}$ and $\xi'\in \op{AC}(S)$.
Since $\mathfrak{g}_{\mathbb{R}}(\xi)=\mathfrak{g}_{\mathbb{R}}(\xi')=\mathfrak{l}_{\mathbb{R}}$,
  $g$ normalizes $\mathfrak{l}_{\mathbb{R}}$.
By our assumption, $g\cdot S=S$ and $g\cdot \op{AC}(S)=\op{AC}(S)$.  
Hence $\xi\in \op{AC}(S)$.
This proves \eqref{eq:asymptotic_cone2}.
\end{proof}

By applying Lemma~\ref{lem:asymptotic_cone} to
 $S=\{\lambda\in \sqrt{-1}Z(\mathfrak{l}_{\mathbb{R}})^*_{\rm gr}\mid
 \pi(\mathfrak{l}_{\mathbb{R}},\Gamma_{\lambda})\in \op{supp}m_{\Pi}\}$, 
 the right hand side of \eqref{eq:inclusion2} equals
\begin{equation*}
G_{\mathbb{R}}\cdot 
\left(\op{AC}\bigl(\{\lambda\in \sqrt{-1}Z(\mathfrak{l}_{\mathbb{R}})^*_{\rm gr}\mid
 \pi(\mathfrak{l}_{\mathbb{R}},\Gamma_{\lambda})\in \op{supp}m_{\Pi}\}
 \bigr) \cap \sqrt{-1}Z(\mathfrak{l}_{\mathbb{R}})^*_{\rm reg}\right).
\end{equation*}
Since the wave front set $\op{WF}(\Pi)$ is $G_{\mathbb{R}}$-stable,
 it is enough to show
\begin{equation}
\label{eq:inclusion2a}
\op{WF}(\Pi)\supset 
\op{AC}\bigl(\{\lambda\in \sqrt{-1}Z(\mathfrak{l}_{\mathbb{R}})^*_{\rm gr}\mid
 \pi(\mathfrak{l}_{\mathbb{R}},\Gamma_{\lambda})\in \op{supp}m_{\Pi}\}
 \bigr) \cap \sqrt{-1}Z(\mathfrak{l}_{\mathbb{R}})^*_{\rm reg}.
\end{equation}

Recall the decompositions
\[
\sqrt{-1}Z(\mathfrak{l}_{\mathbb{R}})^*_{\rm gr}
=\bigsqcup_{\mathfrak{q}} \sqrt{-1}Z(\mathfrak{l}_{\mathbb{R}})_{\rm gr}^{*,\mathfrak{q}},
\quad 
\text{ and } \quad
\widehat{G}_{\mathbb{R}}^{\mathfrak{l}_{\mathbb{R}}}
=\bigcup_{\mathfrak{q}} \widehat{G}_{\mathbb{R}}^{(\mathfrak{l}_{\mathbb{R}},\mathfrak{q})}
\]
defined in the previous section.
Then since the asymptotic cone commutes with finite union,
 it is enough to show \eqref{eq:inclusion2a}
 when $m_{\Pi}$ is a measure on
 $\widehat{G}_{\mathbb{R}}^{(\mathfrak{l}_{\mathbb{R}},\mathfrak{q})}$
 for one parabolic subalgebra $\mathfrak{q}$.
Moreover, fix $d\in \mathbb{R}$ such that 
 $d>\max_{\alpha\in\Delta(\mathfrak{g},\mathfrak{j})}|
 \langle\rho_{\mathfrak{l}},\alpha^\vee\rangle|$ and 
 write $m_{\Pi}$ as a sum of two measures according to the decomposition
\[\widehat{G}_{\mathbb{R}}^{(\mathfrak{l}_{\mathbb{R}},\mathfrak{q})}
=\bigl(\widehat{G}_{\mathbb{R}}^{(\mathfrak{l}_{\mathbb{R}},\mathfrak{q})}
 \cap\widehat{G}_{\mathbb{R}}(\mathfrak{l},d)\bigr) \cup
\bigl(\widehat{G}_{\mathbb{R}}^{(\mathfrak{l}_{\mathbb{R}},\mathfrak{q})}
 \setminus\widehat{G}_{\mathbb{R}}(\mathfrak{l},d)\bigr).
\]
Since 
\begin{equation*}
\op{AC}\bigl(\{\lambda\in \sqrt{-1}Z(\mathfrak{l}_{\mathbb{R}})^*_{\rm gr} \mid
 \pi(\mathfrak{l}_{\mathbb{R}},\Gamma_{\lambda})
 \in \widehat{G}_{\mathbb{R}}(\mathfrak{l},d)\}\bigr)
 \cap Z(\mathfrak{l})^*_{\rm reg} = \emptyset, 
\end{equation*}
it is enough to show \eqref{eq:inclusion2a} when
 $\op{supp} m_\Pi \cap\widehat{G}_{\mathbb{R}}(\mathfrak{l},d)=\emptyset$.
We thus assume
 $m_{\Pi}$ is a measure on $\widehat{G}_{\mathbb{R}}^{(\mathfrak{l}_{\mathbb{R}},\mathfrak{q})}$
 and  $\op{supp} m_\Pi \cap\widehat{G}_{\mathbb{R}}(\mathfrak{l},d)=\emptyset$.

In order to prove \eqref{eq:inclusion2a}, we first show 
\begin{equation}\label{eq:WF_character_integral}
\op{WF}(\Pi)\supset \op{WF}_0\theta(m)
\end{equation}
if $m$ is a measure on
 $\widehat{G}_{\mathbb{R}}^{(\mathfrak{l}_{\mathbb{R}},\mathfrak{q})}$
 which is equivalent to $m_{\Pi}$ and 
 satisfies the condition \eqref{eq:measure_def} given below.
We will see later that \eqref{eq:measure_def} implies $m_{\Pi}$
 is of at most polynomial growth
 and hence $\theta(m)$ is defined as in Lemma~\ref{lem:SS}.

We next take $\xi$ in the right hand side of \eqref{eq:inclusion2a},
 and define a measure $m$ depending on $\xi$,
 which is equivalent to $m_{\Pi}$
 and satisfies the condition \eqref{eq:measure_def}.
Then prove that
\begin{equation}\label{eq:WF_equation}
\op{WF}_0\theta(m)\ni \xi.
\end{equation}

In the next few pages, we prove \eqref{eq:WF_character_integral}
 for $m$ with the condition \eqref{eq:measure_def}.
Let $(\mathcal{O},\Gamma)$ be a semisimple orbital parameter 
 with $\mathcal{O}=G_{\mathbb{R}}\cdot\lambda$
 and $\lambda\in \sqrt{-1}Z(\mathfrak{l}_{\mathbb{R}})^{*,\mathfrak{q}}_{\rm gr}$.
We decompose the unitary representation
 $(\pi(\mathcal{O},\Gamma),V_{(\mathcal{O},\Gamma)})\in \widehat{G}_{\mathbb{R}}^{(\mathfrak{l}_{\mathbb{R}},\mathfrak{q})}$ as
\[V_{(\mathcal{O},\Gamma)}= \widehat{\bigoplus_{\sigma\in \widehat{K}_{\mathbb{R}}}} V_{(\mathcal{O},\Gamma)}(\sigma)\]
where $K_{\mathbb{R}}:=G_{\mathbb{R}}^{\theta}\subset G_{\mathbb{R}}$ is a maximal compact subgroup. We wish to choose an orthonormal basis $\{e_{\sigma,j}(\mathcal{O},\Gamma)\}_{j}$ of $V_{(\mathcal{O},\Gamma)}(\sigma)$ for each $\pi(\mathcal{O},\Gamma)\in \widehat{G}_{\mathbb{R}}^{(\mathfrak{l}_{\mathbb{R}},\mathfrak{q})}$ and each $\sigma\in \widehat{K}_{\mathbb{R}}$. However, we must be careful to choose these bases in a consistent way across parameters $(\mathcal{O},\Gamma)$. To write down this condition correctly, we require additional notation.

Following Section~\ref{sec:quantization_semisimple} or \cite[Section 2]{HO20}, 
 define a parabolic subgroup
 $P_{\mathbb{R}}=M_{\mathbb{R}}A_{\mathbb{R}}(N_P)_{\mathbb{R}}$.
For each semisimple orbital parameter $(\mathcal{O},\Gamma)$
 with $\mathcal{O}=G_{\mathbb{R}}\cdot\lambda$,
 we decompose $\lambda=\lambda_c+\lambda_n$ and
 define an elliptic orbital parameter
 $(\mathcal{O}^{M_{\mathbb{R}}},\Gamma^{M_{\mathbb{R}}})$ for $M_{\mathbb{R}}$. 

For an elliptic orbital parameter
 $(\mathcal{O}_0,\Gamma_0)$ for $M_{\mathbb{R}}$, define
\[\widehat{G}_{\mathbb{R}}^{(\mathfrak{l}_{\mathbb{R}},\mathfrak{q})}(\mathcal{O}_0,\Gamma_0)
=\bigl\{\pi(\mathcal{O},\Gamma)\in
 \widehat{G}_{\mathbb{R}}^{(\mathfrak{l}_{\mathbb{R}},\mathfrak{q})}\mid
 (\mathcal{O}^{M_{\mathbb{R}}},\Gamma^{M_{\mathbb{R}}})
 = (\mathcal{O}_0,\Gamma_0) \bigr\}.\]
Then $\widehat{G}_{\mathbb{R}}^{(\mathfrak{l}_{\mathbb{R}},\mathfrak{q})}$
 is the disjoint union of 
 $\widehat{G}_{\mathbb{R}}^{(\mathfrak{l}_{\mathbb{R}},\mathfrak{q})}(\mathcal{O}_0,\Gamma_0)$
 for various $(\mathcal{O}_0,\Gamma_0)$.
In \cite[Sections 2.3 and 2.4]{HO20}, we give a unitary representation $(\pi(\mathcal{O}^{M_{\mathbb{R}}},\Gamma^{M_{\mathbb{R}}}),V_{(\mathcal{O}^{M_{\mathbb{R}}},\Gamma^{M_{\mathbb{R}}})})$ of $M_{\mathbb{R}}$ associated to $(\mathcal{O}^{M_{\mathbb{R}}},\Gamma^{M_{\mathbb{R}}})$. 
Then we form the bundle
\[\mathcal{V}:=G_{\mathbb{R}}\times_{P_{\mathbb{R}}} 
\bigl(V_{(\mathcal{O}^{M_{\mathbb{R}}},\Gamma^{M_{\mathbb{R}}})}
\boxtimes e^{\lambda_n+\rho(\mathfrak{n}_\mathfrak{p})}\bigr),\]
and we define
\[V_{(\mathcal{O},\Gamma)}:=L^2(G_{\mathbb{R}}/P_{\mathbb{R}},\mathcal{V}).\]
In order to study the action of $K_{\mathbb{R}}$ on $V_{(\mathcal{O},\Gamma)}$, it is convenient to use the compact model for the induced representation (see e.g.\ \cite[Chapter 7]{Kna86}) obtained by restricting the sections on $G_{\mathbb{R}}/P_{\mathbb{R}}$ to sections on $K_{\mathbb{R}}/(K_{\mathbb{R}}\cap M_{\mathbb{R}})$. This gives us an identification
\[L^2(G_{\mathbb{R}}/P_{\mathbb{R}},\mathcal{V})\stackrel{\sim}{\rightarrow} L^2(K_{\mathbb{R}}/(K_{\mathbb{R}}\cap M_{\mathbb{R}}),\mathcal{V}|_{K_{\mathbb{R}}/(K_{\mathbb{R}}\cap M_{\mathbb{R}})})\]
as unitary $K_{\mathbb{R}}$ representations. Notice that this compact picture only depends on the elliptic orbital parameter $(\mathcal{O}^{M_{\mathbb{R}}},\Gamma^{M_{\mathbb{R}}})$ since $\mathcal{V}|_{K_{\mathbb{R}}/(K_{\mathbb{R}}\cap M_{\mathbb{R}})}$ is independent of $\lambda_n$. Now, for every $\sigma\in \widehat{K}_{\mathbb{R}}$, we may fix an orthonormal basis for $L^2(K_{\mathbb{R}}/(K_{\mathbb{R}}\cap M_{\mathbb{R}}),\mathcal{V}|_{K_{\mathbb{R}}/(K_{\mathbb{R}}\cap M_{\mathbb{R}})})(\sigma)$, and we may pull this basis back to an orthonormal basis $\{e_{\sigma,j}(\mathcal{O}^{M_{\mathbb{R}}},\Gamma^{M_{\mathbb{R}}},\lambda_n)\}$ of $V_{(\mathcal{O},\Gamma)}(\sigma)$. Since the compact model for $\pi(\mathcal{O},\Gamma)$ agrees with the compact model for $\pi(\mathcal{O}',\Gamma')$ whenever $(\mathcal{O}^{M_{\mathbb{R}}},\Gamma^{M_{\mathbb{R}}})=((\mathcal{O}')^{M_{\mathbb{R}}},(\Gamma')^{M_{\mathbb{R}}})$, we note that the basis $\{e_{\sigma,j}(\mathcal{O}^{M_{\mathbb{R}}},\Gamma^{M_{\mathbb{R}}},\lambda_n)\}$ depends continuously on the parameter $(\mathcal{O},\Gamma)$. 

Let $(\widehat{M}_{\mathbb{R}})_{\text{ell}}^{\Pi}$ denote the collection of
 all elliptic semisimple orbital parameters
 $(\mathcal{O}_0,\Gamma_0)$ for $M_{\mathbb{R}}$ with
 $m_{\Pi}( \widehat{G}_{\mathbb{R}}^{(\mathfrak{l}_{\mathbb{R}},\mathfrak{q})}
 (\mathcal{O}_0,\Gamma_0))\neq 0$.
Fix a measure $m$ on $\widehat{G}_{\mathbb{R}}^{(\mathfrak{l}_{\mathbb{R}},\mathfrak{q})}$
 equivalent to $m_{\Pi}$ such that 
\begin{equation}\label{eq:measure_def}
m(\widehat{G}_{\mathbb{R}}^{(\mathfrak{l}_{\mathbb{R}},\mathfrak{q})}
 (\mathcal{O}_0,\Gamma_0))=1
\end{equation}
for every $(\mathcal{O}_0,\Gamma_0)\in (\widehat{M}_{\mathbb{R}})_{\text{ell}}^{\Pi}$. 
\eqref{eq:measure_def} implies that $m$ is of at most polynomial growth.
Indeed, we have
\begin{equation*}
\begin{split}
&\int_{\lambda\in \sqrt{-1}Z(\mathfrak{l}_{\R})_{\rm gr}^{*,\mathfrak{q}}}
 \frac{p_*dm}{(1+|\lambda|^2)^N} \\
&\leq 
\int_{\lambda\in \sqrt{-1}Z(\mathfrak{l}_{\R})_{\rm gr}^{*,\mathfrak{q}}}
 \frac{p_*dm}{(1+|\lambda_c|^2)^N} \\
&=\sum_{(\mathcal{O}_0,\Gamma_0)\in (\widehat{M}_{\mathbb{R}})_{\text{ell}}^{\Pi}}
 \int_{\widehat{G}_{\mathbb{R}}^{(\mathfrak{l}_{\mathbb{R}},\mathfrak{q})}(\mathcal{O}_0,\Gamma_0)}
 \frac{dm}{(1+|\lambda_c|^2)^N} \\
&=\sum_{(\mathcal{O}_0,\Gamma_0)\in (\widehat{M}_{\mathbb{R}})_{\text{ell}}^{\Pi}}
 \frac{1}{(1+|\lambda_c|^2)^N},
\end{split}
\end{equation*}
where $\mathcal{O}_0=M_{\mathbb{R}}\cdot \lambda_c$.
Since the last expression is a sum of over a lattice with uniformly bounded finite multiplicities,
 it converges for a sufficiently large $N$, showing that $m$ is of at most polynomial growth.

We now fix a multiplicity free subrepresentation
\[\int^{\oplus}_{\pi(\mathcal{O},\Gamma)\in \widehat{G}_{\mathbb{R}}^{(\mathfrak{l}_{\mathbb{R}},\mathfrak{q})}(\mathcal{O}_0,\Gamma_0)}V_{(\mathcal{O},\Gamma)}dm
 \simeq V^{(\mathcal{O}_0,\Gamma_0)}\subset V_{\Pi}\]
for every $(\mathcal{O}_0,\Gamma_0)\in (\widehat{M}_{\mathbb{R}})_{\text{ell}}^{\Pi}$. 
Here, $V_{\Pi}$ denotes the representation space of $\Pi$.
We may then view $\lambda_n\mapsto e_{\sigma,j}(\mathcal{O}_0,\Gamma_0,\lambda_n)$ as a vector in $V^{(\mathcal{O}_0,\Gamma_0)}$ which we will denote by $e_{\sigma,j}(\mathcal{O}_0,\Gamma_0)$. Now, since $|e_{\sigma,j}(\mathcal{O}_0,\Gamma_0,\lambda_n)|=1$ for all $\lambda_n$ and $m(\widehat{G}_{\mathbb{R}}^{(\mathfrak{l}_{\mathbb{R}},\mathfrak{q})}(\mathcal{O}_0,\Gamma_0))=1$, we deduce $|e_{\sigma,j}(\mathcal{O}_0,\Gamma_0)|=1$.

Define 
\[V':=\overline{\langle e_{\sigma,j}(\mathcal{O}_0,\Gamma_0) \mid \sigma\in \widehat{K}_{\mathbb{R}},\ (\mathcal{O}_0,\Gamma_0)\in (\widehat{M}_{\mathbb{R}})_{\text{ell}}^{\Pi}\rangle}\]
to be the closure of the span of the $e_{\sigma,j}(\mathcal{O}_0,\Gamma_0)$,
 and let $P\colon V_{\Pi}\rightarrow V'$ denote the orthogonal projection onto $V'$. 
Observe that for $g$ in a small neighborhood of $e\in G_{\mathbb{R}}$, we have
\begin{equation}\label{eq:Theta}
\begin{split}
&\phantom{=}\op{Tr}(\Pi(g)P) \\
&=\sum_{\pi(\mathcal{O}_0,\Gamma_0)\in (\widehat{M}_{\mathbb{R}})^{\Pi}_{\text{ell}}}\sum_{\sigma\in \widehat{K}_{\mathbb{R}}}\sum_j (\Pi(g)e_{\sigma,j}(\mathcal{O}_0,\Gamma_0),e_{\sigma,j}(\mathcal{O}_0,\Gamma_0)) \\
&= \sum_{\pi(\mathcal{O}_0,\Gamma_0)\in (\widehat{M}_{\mathbb{R}})^{\Pi}_{\text{ell}}}\sum_{\sigma\in \widehat{K}_{\mathbb{R}}}\sum_j\int_{\lambda_n}(\Pi(g)e_{\sigma,j}(\mathcal{O}_0,\Gamma_0,\lambda_n),e_{\sigma,j}(\mathcal{O}_0,\Gamma_0,\lambda_n))dm \\
&= \sum_{\pi(\mathcal{O}_0,\Gamma_0)\in (\widehat{M}_{\mathbb{R}})^{\Pi}_{\text{ell}}}\int_{\lambda_n}\sum_{\sigma\in \widehat{K}_{\mathbb{R}}}\sum_j(\Pi(g)e_{\sigma,j}(\mathcal{O}_0,\Gamma_0,\lambda_n),e_{\sigma,j}(\mathcal{O}_0,\Gamma_0,\lambda_n))dm \\
&= \sum_{\pi(\mathcal{O}_0,\Gamma_0)\in (\widehat{M}_{\mathbb{R}})^{\Pi}_{\text{ell}}}\int_{\lambda_n}\Theta_{\pi(\mathcal{O},\Gamma)}dm \\
&= \int_{\pi(\mathcal{O},\Gamma)\in \widehat{G}_{\mathbb{R}}^{(\mathfrak{l}_{\mathbb{R}},\mathfrak{q})}}\Theta_{\pi(\mathcal{O},\Gamma)}dm\\
&=\Theta(m).
\end{split}
\end{equation}
We have defined $\Theta(m)$ on the group in the same way that we defined $\theta(m)$ on the Lie algebra. It is a well-defined distribution in a sufficiently small neighborhood of the identity by Lemma \ref{lem:SS} and the fact that $\exp$ restricts to a diffeomorphism of a neighborhood of zero onto a neighborhood of $e\in G_{\mathbb{R}}$. 

Next, let $\Omega_K\in \mathcal{U}(\mathfrak{k})\subset \mathcal{U}(\mathfrak{g})$ denote the Casimir operator for $K$. We wish to show that $(I+\Omega_K)^{-N}P$ is a trace class operator on $V_{\Pi}$ for sufficiently large $N$. Let $T_{\mathbb{R}}\subset K_{\mathbb{R}}$ be a maximal torus with Lie algebra $\mathfrak{t}_{\mathbb{R}}$, and let $\mathcal{C}\subset \sqrt{-1}\mathfrak{t}_{\mathbb{R}}^*$ be a closed Weyl chamber in $\sqrt{-1}\mathfrak{t}_{\mathbb{R}}^*$. For each $(\sigma,W_{\sigma})\in \widehat{K}_{\mathbb{R}}$, let $\lambda_{\sigma}\in \mathcal{C}$ be the corresponding highest weight. Then there exists a norm $|\cdot|$ on the vector space $\sqrt{-1}\mathfrak{t}_{\mathbb{R}}^*$ such that 
$\Omega_K\cdot v=|\lambda_{\sigma}|^2 v$ for all $v\in W_{\sigma}$.
We calculate 
\begin{equation}\label{eq:K_types}
\begin{split}
&\op{Tr}((I+\Omega_K)^{-N}P) \\
=&\sum_{\pi(\mathcal{O}_0,\Gamma_0)\in (\widehat{M}_{\mathbb{R}})^{\Pi}_{\text{ell}}}\sum_{\sigma\in \widehat{K}_{\mathbb{R}}}\sum_{j}((I+\Omega_K)^{-N}e_{\sigma,j}(\mathcal{O}_0,\Gamma_0),e_{\sigma,j}(\mathcal{O}_0,\Gamma_0)) \\
=&\sum_{\pi(\mathcal{O}_0,\Gamma_0)\in (\widehat{M}_{\mathbb{R}})^{\Pi}_{\text{ell}}}\sum_{\sigma\in \widehat{K}_{\mathbb{R}}} \frac{n(\mathcal{O}_0,\Gamma_0,\sigma)}{(1+|\lambda_{\sigma}|^2)^N},
\end{split}
\end{equation}
where $n(\mathcal{O}_0,\Gamma_0,\sigma)$ denotes the multiplicity of $\sigma\in \widehat{K}_{\mathbb{R}}$ in $\pi(\mathcal{O},\Gamma)\in \widehat{G}_{\mathbb{R}}^{(\mathfrak{l}_{\mathbb{R}},\mathfrak{q})}(\mathcal{O}_0,\Gamma_0)$. Recall (see page 205 of \cite{Kna86}) that 
\begin{equation}\label{eq:mult_bound}
n(\mathcal{O}_0,\Gamma_0,\sigma)\leq \dim \sigma
\end{equation}
for all $\pi(\mathcal{O}_0,\Gamma_0)\in (\widehat{M}_{\mathbb{R}})_{\text{ell}}$ and all $\sigma\in \widehat{K}_{\mathbb{R}}$. Further, there exists a natural number $r\in \mathbb{N}$ and a constant $C>0$ such that
\begin{equation}\label{eq:mult_bound2}
\dim\sigma \leq C(1+|\lambda_{\sigma}|^2)^r.
\end{equation}
Therefore, utilizing \eqref{eq:mult_bound} and \eqref{eq:mult_bound2}, we have that \eqref{eq:K_types} is bounded by
\begin{align*}
&\leq \sum_{\pi(\mathcal{O}_0,\Gamma_0)\in (\widehat{M}_{\mathbb{R}})_{\text{ell}}}\sum_{\sigma\in \widehat{K}_{\mathbb{R}}} \frac{\dim\sigma}{(1+|\lambda_{\sigma}|^2)^N}\\
&\leq C \sum_{\pi(\mathcal{O}_0,\Gamma_0)\in (\widehat{M}_{\mathbb{R}})_{\text{ell}}}\sum_{\sigma\in \widehat{K}_{\mathbb{R}}} \frac{1}{(1+|\lambda_{\sigma}|^2)^{N-r}}.
\end{align*}
Since $\{\lambda_{\sigma}\}_{\sigma\in \widehat{K}_{\mathbb{R}}}$ form a subset of a lattice in $\sqrt{-1}\mathfrak{t}_{\mathbb{R}}^*$, we obtain the bound
\[\sum_{\sigma\in \widehat{K}_{\mathbb{R}}}
 \frac{1}{(1+|\lambda_{\sigma}|^2)^{N-r}}
 \leq C\max_{L^2(K_{\mathbb{R}}/M_{\mathbb{R}},\mathcal{V})(\sigma)\neq 0} \frac{1}{(1+|\lambda_{\sigma}|^2)^{N-r}}\]
for $N\geq r+\dim\mathfrak{t}_{\mathbb{R}}+1$ and for some $C>0$. 
Let $\xi_{(\mathcal{O}_0,\Gamma_0)}\in \sqrt{-1}\mathfrak{t}_{\mathbb{R}}^*$ denote the highest weight of the minimal $K$-type of $L^2(K_{\mathbb{R}}/M_{\mathbb{R}},\mathcal{V}_{(\mathcal{O}_0,\Gamma_0)})$. By Theorem 10.44 of \cite{KV95} and the definition of $\pi(\mathcal{O}_0,\Gamma_0)$ (see \cite[\S 2.3]{HO20}), we have $\xi_{(\mathcal{O}_0,\Gamma_0)}=\lambda_c-\rho_{(\mathfrak{n}\cap \mathfrak{k})}+\rho_{(\mathfrak{n}\cap \mathfrak{g}^{-\theta})}$ when $\mathcal{O}_0=M_{\mathbb{R}}\cdot \lambda_c$ ($\lambda_c\in\sqrt{-1}\mathfrak{t}_{\mathbb{R}}^*$). We observe 
\begin{align*}
&\sum_{\pi(\mathcal{O}_0,\Gamma_0)\in (\widehat{M}_{\mathbb{R}})_{\text{ell}}}\frac{1}{(1+|\xi_{(\mathcal{O}_0,\Gamma_0)}|^2)^{N-r}}\\
=&\sum_{\pi(\mathcal{O}_0,\Gamma_0)\in (\widehat{M}_{\mathbb{R}})_{\text{ell}}}\frac{1}{(1+|\lambda_c-\rho_{(\mathfrak{n}\cap \mathfrak{k})}+\rho_{(\mathfrak{n}\cap \mathfrak{g}^{-\theta})}|^2)^{N-r}}
\end{align*}
is a sum over a lattice, and we observe that each term occurs with uniformly bounded, finite multiplicity. By standard calculus arguments, we deduce that the sum converges for sufficiently large $N$. It follows that $(I+\Omega_K)^{-N}P$ is of trace class for sufficiently large $N$. Utilizing Howe's original definition of the wave front set of a Lie group representation (\cite{How81}, see also \cite[\S 2]{HHO16} for an exposition), we have
\begin{equation}\label{eq:Howe_def}
\op{WF}_e(\op{Tr}(\Pi(g)(I+\Omega_K)^{-N}P))\subset \op{WF}(\Pi)
\end{equation}
for sufficiently large $N$. Next, utilizing \eqref{eq:Theta},
 we compute for $\varphi\in C_c^{\infty}(G_{\mathbb{R}})$
\begin{align*}
\langle \Theta(m),\varphi\rangle
&=\op{Tr}(\Pi(\varphi)P)\\
&=\op{Tr}(\Pi(\varphi)(I+\Omega_K)^N(I+\Omega_K)^{-N}P)\\
&=\op{Tr}(\Pi(L_{(I+\Omega_K)^N}\varphi)(I+\Omega_K)^{-N}P)\\
&=L_{(I+\Omega_K)^N}\op{Tr}(\Pi(\varphi)(I+\Omega_K)^{-N}P).
\end{align*}
Since applying the differential operator $L_{(I+\Omega_K)^N}$ can only decrease the wave front set of the distribution $\op{Tr}(\Pi(\varphi)(I+\Omega_K)^{-N}P)$,  we conclude
\[\op{WF}_e(\Theta(m))\subset \op{WF}_e(\op{Tr}(\Pi(\varphi)(I+\Omega_K)^{-N}P)).\]
Combining with \eqref{eq:Howe_def}, we have
\[\op{WF}_e(\Theta(m))\subset \op{WF}(\Pi).\]
Finally, since $\theta(m)$ differs from $\exp^*\Theta(m)$ only by multiplication with a real analytic function, we conclude \eqref{eq:WF_character_integral}.

Next, we will define a measure $m$
 which is equivalent to $m_{\Pi}$
 and satisfies \eqref{eq:WF_equation} and \eqref{eq:measure_def}.
We fix a positive definite, $K_{\mathbb{R}}$-invariant bilinear form $(\cdot,\cdot)$
 on $\mathfrak{g}_{\mathbb{R}}$, which is extended by complex linearity to $\mathfrak{g}$. 
We may then use $(\cdot,\cdot)$ to give an isomorphism
 $\mathfrak{g}_{\mathbb{R}}\simeq \mathfrak{g}_{\mathbb{R}}^*$,
 and we write $(\cdot,\cdot)$ for the corresponding bilinear form on $\mathfrak{g}^*$,
 which is positive definite on $\mathfrak{g}_{\mathbb{R}}^*$
 and negative definite on $\sqrt{-1}\mathfrak{g}_{\mathbb{R}}^*$.
For $\xi\in \mathfrak{g}^*$, write 
 $\xi=\RE\xi+\sqrt{-1}\IM\xi$ with $\RE\xi,\IM\xi\in \mathfrak{g}_{\mathbb{R}}^*$.
We write $|\xi|:=((\RE\xi,\RE\xi)+(\IM\xi, \IM\xi))^{1/2}$ for $\xi\in \mathfrak{g}^*$.

Fix 
\[\xi\in \op{AC} \bigl(\bigl\{\lambda \in \sqrt{-1}Z(\mathfrak{l}_{\mathbb{R}})^*_{\rm gr}
\mid \pi(\mathfrak{l}_{\mathbb{R}},\Gamma_\lambda)\in
 \op{supp} m_{\Pi}\bigr\}\bigr)
 \cap \sqrt{-1}Z(\mathfrak{l}_{\mathbb{R}})^*_{\rm reg}.\]
Replacing $\xi$ by $|\xi|^{-1} \cdot \xi$ we may assume $|\xi|=1$.
Write $p_*m_{\Pi}$ for the pushforward of $m_{\Pi}$
 by the map
\[p\colon \widehat{G}_{\mathbb{R}}^{(\mathfrak{l}_{\mathbb{R}},\mathfrak{q})}
\ni \pi(\mathfrak{l}_{\mathbb{R}},\Gamma_{\lambda})
\mapsto \lambda \in \sqrt{-1}Z(\mathfrak{l}_{\mathbb{R}})^{*,\mathfrak{q}}_{\rm gr}.\]
Then we can take a sequence
 $\{\zeta_i\}_{i\in \mathbb{Z}_{>0}}\subset \sqrt{-1}Z(\mathfrak{l}_{\mathbb{R}})^{*,\mathfrak{q}}$
 and $\{t_i\}_{i\in \mathbb{Z}_{>0}} \subset \mathbb{R}_{>0}$ such that
\begin{equation*}
|\zeta_i|=1,\ \ 
\lim_{i\to \infty} \zeta_i =\xi,\ \ 
t_i >2^{i+1},  \text{ and }\ 
t_i\zeta_i \in \op{supp} p_*m_{\Pi}.
\end{equation*}
We now want a measure $m$ on
 $\widehat{G}_{\mathbb{R}}^{(\mathfrak{l}_{\mathbb{R}},\mathfrak{q})}$ satisfying
\begin{equation}\label{eq:measure_def2}
p_*m(B_{1}(t_i\zeta_i))\geq 2^{-i-1}
\end{equation}
for all $i$.
Here, $B_{1}(t_i\zeta_i)$ is the open ball in $\sqrt{-1}Z(\mathfrak{l}_{\mathbb{R}})^*$
 with radius $1$ and center $t_i\zeta_i$.
It is easy to see that there exists a measure $m$
 which is equivalent to $m_{\Pi}$
 and satisfies \eqref{eq:measure_def} and \eqref{eq:measure_def2}.
We fix such $m$.

In order to prove \eqref{eq:WF_equation}, we require a lemma.
 Suppose $W$ is a finite-dimensional, real vector space with a positive definite inner product. 
Let 
\[\mathcal{G}(x)=e^{-|x|^2/2},\quad \mathcal{G}_t(x)=e^{-t|x|^2/2}\]
denote the corresponding Gaussian and family of Gaussians on $W$ for $t>0$.
\begin{lemma}[\cite{Fo89}]\label{lem:FBI}  
Suppose $u$ is a tempered distribution on a finite-dimensional,
 real vector space $W$. 
Then a vector $\xi\in W^*$ belongs to $\op{WF}_0(u)$
 if there exists a sequence $\zeta_i\in W^*$
 and $t_i>0$ such that
\[\lim_{i\to \infty} \zeta_i =\xi,\qquad 
\lim_{i\to \infty} t_i =\infty\] 
 and there exist $N\in \mathbb{N}$ and $C>0$ such that 
\begin{equation}\label{eq:Fol_bound}
|(u\cdot \mathcal{G}_{t_i})\sphat\,(t_i\zeta_i)|\geq C\cdot (1+t_i^2)^{-N/2}
\end{equation}
for sufficiently large $i$.
\end{lemma}
Lemma \ref{lem:FBI} is half of \cite[Theorem 3.22]{Fo89}
 with $f$ replaced by $u$ and $\phi$ replaced by $\mathcal{G}$.
We will apply Lemma \ref{lem:FBI} in the case $W=\sqrt{-1}\mathfrak{g}_{\R}^*$
 and $u=\theta(m)\cdot \mathcal{G}$.
The bilinear form $(\cdot, \cdot)$ we fixed above
 is negative definite on $\sqrt{-1}\mathfrak{g}_{\R}^*$.
For $\zeta\in \sqrt{-1}Z(\mathfrak{l}_{\mathbb{R}})^{*}_{\rm reg}$
 and $t>0$,
 it follows from Lemma~\ref{lem:SS} that 
\begin{equation}\label{eq:contour_integral}
\begin{split}
&(\theta(m)\cdot \mathcal{G}_{t+1})\sphat\,(t\zeta)\\
&=\int_{\pi(\mathcal{O},\Gamma)\in \op{supp}m}
 \int_{\eta\in\mathcal{C}(\mathcal{O},\mathfrak{q})}
 (\mathcal{G}_{t+1})\spcheck (\eta-t\zeta)\nu dm \\
&=c\int_{\pi(\mathcal{O},\Gamma)\in \op{supp}m}
 \int_{\eta\in\mathcal{C}(\mathcal{O},\mathfrak{q})}
 \frac{1}{\sqrt{t+1}}e^{(t\zeta-\eta,t\zeta-\eta)/2(t+1)}\nu dm.
\end{split}
\end{equation}
where the constant $c\neq 0$ depends only on the bilinear form $(\cdot,\cdot)$. 
We will estimate this integral and set $\zeta=\zeta_i$ and $t=t_i$ to prove
 the inequality \eqref{eq:Fol_bound}.
Note that this integral converges absolutely by part \eqref{SSpart1} of Lemma \ref{lem:SS}. In addition, since we wish to bound this integral as in \eqref{eq:Fol_bound}, we may safely ignore the constant $c$  and the factor $\frac{1}{\sqrt{t+1}}$ in what follows.

We estimate the integral as $t\to \infty$
 uniformly when $\zeta$ varies in a compact subset of
 $\sqrt{-1}Z(\mathfrak{l}_{\mathbb{R}})^{*}_{\rm reg}$.
Fix a compact set $V\subset \sqrt{-1}Z(\mathfrak{l}_{\mathbb{R}})^*_{\rm reg}$
 and suppose $\zeta\in V$.
We break up the integral \eqref{eq:contour_integral} into two pieces\\
\begin{align}\label{eq:positive}
&\int_{\pi(\mathcal{O},\Gamma)\in \op{supp}m}\int_{\substack{\eta\in\mathcal{C}(\mathcal{O},\mathfrak{q})\\ |t\zeta-\eta|\leq \delta t}}
 e^{(t\zeta-\eta,t\zeta-\eta)/2(t+1)}\nu dm\\ \label{eq:small}
+&\int_{\pi(\mathcal{O},\Gamma)\in \op{supp}m}\int_{\substack{\eta\in\mathcal{C}(\mathcal{O},\mathfrak{q})\\ |t\zeta-\eta|> \delta t}}e^{(t\zeta-\eta,t\zeta-\eta)/2(t+1)}\nu dm.
\end{align}
for some $\delta>0$. First, we wish to show that for every $\delta>0$, the size of the integral \eqref{eq:small} decays faster than any rational function of $t$ as $t\rightarrow \infty$. Then we will show that for sufficiently small $\delta>0$ and sufficiently large $t$, the imaginary part of the integral \eqref{eq:positive} is small relative to the real part of the integral \eqref{eq:positive}. Finally, we will show that the real part of the integral \eqref{eq:positive} is positive and bounded below by a rational function of $t$.

To analyze these integrals,
 we put $\eta':=\sqrt{-1}\IM\eta\in\sqrt{-1}\mathfrak{g}^*_{\mathbb{R}}$, and we expand
\begin{equation}\label{eq:Gaussian_expansion}
e^{(t\zeta-\eta,t\zeta-\eta)/2(t+1)}
=e^{(t\zeta-\eta',t\zeta-\eta')/2(t+1)}\cdot e^{-(t\zeta-\eta',\RE\eta)/(t+1)}\cdot e^{(\RE\eta,\RE\eta)/2(t+1)}.
\end{equation}
Now, we consider the integral \eqref{eq:small}. 
We observe $(t\zeta-\eta',\RE\eta)/(t+1)$
 is an imaginary number. 
Hence
\begin{equation}\label{eq:exp_bound2}
|e^{(t\zeta-\eta',\RE\eta)/(t+1)}|=1.
\end{equation}
In addition, there exists a constant $B>0$
 such that $|\RE\eta|\leq B$
 for all $\eta\in \mathcal{C}(\mathcal{O},\mathfrak{q})$
 and all $(\mathcal{O},\Gamma)$. 
Therefore,
\begin{equation}\label{eq:exp_bound3}
|e^{(\RE\eta,\RE\eta)/2(t+1)}|\leq e^{B^2/2(t+1)}\leq e^{B^2}.
\end{equation}

Plugging \eqref{eq:Gaussian_expansion}, \eqref{eq:exp_bound2}, and \eqref{eq:exp_bound3}
 into the integral \eqref{eq:small}, we obtain
\begin{align}\label{eq:small_integral}
&\left|\int_{\pi(\mathcal{O},\Gamma)\in \op{supp}m}
\int_{\substack{\eta\in\mathcal{C}(\mathcal{O},\mathfrak{q})\\ 
|t\zeta-\eta|> \delta t}}e^{(t\zeta-\eta,t\zeta-\eta)/2(t+1)}
\nu dm\right|\nonumber\\
& \leq 
 c_1\int_{\pi(\mathcal{O},\Gamma)\in \op{supp}m}
\int_{\substack{\eta\in\mathcal{C}(\mathcal{O},\mathfrak{q})\\
  |t\zeta-\eta|> \delta t}}
|e^{-|t\zeta-\eta'|^2/2(t+1)} \nu |dm
\end{align}
for some constant $c_1>0$ independent of $\zeta$, $\delta$, and $t$. To bound this latter integral, we will apply part \eqref{SSpart1} of Lemma \ref{lem:SS}
with 
\[\alpha(\eta)=e^{-|t\zeta-\eta'|^2/2(t+1)}.\] 
In order to apply part \eqref{SSpart1} of Lemma \ref{lem:SS},
 we need a lemma bounding the growth of our $\alpha(\eta)$ as a function of $t$.

\begin{lemma}\label{lem:A_bound}
For every $N,k\in \mathbb{N}$ and every $\delta>0$,
 there exist constants $B_{N,k,\delta}>0$ and $t_0>0$ such that
\begin{equation*}
\sup_{\substack{\eta\in \mathcal{C}(\mathcal{O},\mathfrak{q})\\ |t\zeta-\eta|\geq \delta t}}
 (1+|\eta'|^2)^{N/2}e^{-|t\zeta-\eta'|^2/2(t+1)}\leq
 \frac{B_{N,k,\delta}}{(1+t^2)^{k/2}}
\end{equation*}
for $t>t_0$.
The constants $B_{N,k,\delta}$ and $t_0$
 do not depend on $\zeta\in V$ or $\mathcal{O}$.
\end{lemma}
\begin{proof} 
Since $\zeta$ and $\eta-\eta'$ lies in a bounded set,
 $|t\zeta-\eta|\geq \delta t$ implies
 that $|\eta'|$ is at most of order $t$ when $t\to \infty$.
On the other hand, $|t\zeta-\eta'|$ is at least of order $t$.
Hence $e^{-|t\zeta-\eta'|^2/2(t+1)}$ decays exponentially when $t\to \infty$.
This shows the existence of the constant $B_{N,k,\delta}$ as in the lemma.
\end{proof}

Now, to bound \eqref{eq:small_integral}, we apply the bound in Lemma \ref{lem:A_bound} to Lemma~\ref{lem:integral_semialgebraic},
 where we set $M$ in Lemma~\ref{lem:integral_semialgebraic}
 to be the exponent $M_0$ in the polynomial growth bound on the measure $m$
 (see \eqref{eq:at_most_polynomial}). 
We deduce that for every $\delta>0$ and $k\in \mathbb{N}$,
 there exists a constant $B_{k,\delta}>0$ such that 
\begin{equation*}
\int_{\pi(\mathcal{O},\Gamma)\in \op{supp}m}
\int_{\begin{subarray}{l}\,\eta\in\mathcal{C}(\mathcal{O},\mathfrak{q})\\ |t\zeta-\eta|> \delta t\end{subarray}}
\bigl| e^{-|t\zeta-\eta'|^2/2(t+1)}\nu \bigr| dm
\leq \frac{B_{k,\delta}}{(1+t^2)^{k/2}}.
\end{equation*}
Combining with \eqref{eq:small_integral}, we obtain
\begin{equation} \label{eq:first_integral_bound}
\left|\int_{\pi(\mathcal{O},\Gamma)\in \op{supp}m}
\int_{\begin{subarray}{l}\,\eta\in\mathcal{C}(\mathcal{O},\mathfrak{q})\\ 
|t\zeta-\eta|> \delta t \end{subarray}}e^{(t\zeta-\eta,t\zeta-\eta)/2(t+1)}
\nu dm\right|
\leq  \frac{c_1 B_{k,\delta}}{(1+t^2)^{k/2}}.
\end{equation}
The constant $c_1B_{k,\delta}$ does not depend on $\zeta\in V$.

Next, we focus on the integral \eqref{eq:positive}
\[ \int_{\pi(\mathcal{O},\Gamma)\in \op{supp}m}
 \int_{\begin{subarray}{l}\,\eta\in\mathcal{C}(\mathcal{O},\mathfrak{q})\\
	 |t\zeta-\eta|\leq \delta t\end{subarray}}
e^{(t\zeta-\eta,t\zeta-\eta)/2(t+1)}\nu dm.\]
There are two parts to the integral, the function $e^{(t\zeta-\eta,t\zeta-\eta)/2(t+1)}$
 and the differential form $\nu$. 
We must analyze both separately. 
We begin to analyze the function $e^{(t\zeta-\eta,t\zeta-\eta)/2(t+1)}$
 by expanding it into three terms as in \eqref{eq:Gaussian_expansion}. 
Since $\RE\eta$ is bounded, we see that given $\epsilon'>0$,
 there exists $t_0>0$ such that whenever $t>t_0$, we have
\begin{equation}\label{eq:3term}
|e^{(\RE\eta,\RE\eta)/2(t+1)}-1|<\epsilon'
\end{equation} 
for all $\eta\in \mathcal{C}(\mathcal{O},\mathfrak{q})$. 
This bounds the third term in the expansion \eqref{eq:Gaussian_expansion}. 
Choose $B>0$ such that $|\RE\eta|\leq B$ for all
 $\eta\in \mathcal{C}(\mathcal{O},\mathfrak{q})$. 
If $|t\zeta-\eta|\leq \delta t$, then
\[\frac{|(t\zeta-\eta',\RE\eta)|}{2(t+1)}
 \leq \frac{|t\zeta-\eta'||\RE\eta|}{2(t+1)}
 \leq \frac{\delta t B}{2(t+1)}\leq \delta B.\]
Therefore, given $\epsilon'>0$, we may choose $\delta>0$ sufficiently small
 such that we have
\begin{equation}\label{eq:2term}
|e^{(t\zeta-\eta',\RE\eta)/2(t+1)}-1|<\epsilon'
\end{equation}
whenever $|t\zeta-\eta'|\leq \delta t$ and
 $\eta\in \mathcal{C}(\mathcal{O},\mathfrak{q})$.
This bounds the second term in the expansion \eqref{eq:Gaussian_expansion}. 
Since $|t\zeta-\eta'|^2 \in \mathbb{R}$, we note
\begin{equation}\label{eq:1term}
e^{-|t\zeta-\eta'|^2/2(t+1)}\in \mathbb{R}_{>0}.
\end{equation}
Define
\[f_{t\zeta}(\eta):=e^{(t\zeta-\eta,t\zeta-\eta)/2(t+1)}.\] 
Write $f_{t\zeta}=\RE f_{t\zeta}+\sqrt{-1}\IM f_{t\zeta}$
 with $\RE f_{t\zeta}, \IM f_{t\zeta}\in \mathbb{R}$. 

\begin{lemma}\label{lem:positive_function} There exist $t_0>0$ and $\delta_0>0$ such that whenever $t>t_0$, $\delta_0>\delta>0$, $\zeta\in V$,
 $\eta\in \mathcal{C}(\mathcal{O},\mathfrak{q})$ and
 $|t\zeta-\eta'|\leq \delta t$, 
 we have
\[|\IM f_{t\zeta}|<\frac{1}{5}\RE f_{t\zeta}.\]  
\end{lemma}
Lemma \ref{lem:positive_function} follows from the expansion \eqref{eq:Gaussian_expansion} together with \eqref{eq:3term}, \eqref{eq:2term}, \eqref{eq:1term}. 
Lemma \ref{lem:positive_function} is half of our analysis of the integral \eqref{eq:positive}. 
The other half involves analyzing the differential form $\nu$.
In the next section, we define a new real-valued differential form $\nu^{\op{O}}$
 on $\mathcal{C}(\mathcal{O},\mathfrak{q})$.
Then we bound the size of the differential form $\nu-\nu^{\op{O}}$
 and prove the following lemma.

\begin{lemma} \label{lem:diff_form1} 
There exist $t_0>0$ and $\delta_0>0$ such that for $t>t_0$, $\delta_0>\delta>0$,
 $\zeta\in V$, we have
\[|\nu-\nu^{\op{O}}|\leq \frac{1}{5}|\nu^{\op{O}}|\]
on  $\mathcal{C}(\mathcal{O},\mathfrak{q})\cap B_{\delta t}(t\zeta)$.
\end{lemma}

In the above lemma, the inequality $|\nu-\nu^{\op{O}}|\leq \frac{1}{5}|\nu^{\op{O}}|$ means 
\[|(\nu-\nu^{\op{O}})(Z_1,\ldots,Z_{2n})|\leq \frac{1}{5}|\nu^{\op{O}}(Z_1,\ldots,Z_{2n})|\]
 for all bases $\{Z_1,\ldots,Z_{2n}\}$ of $T_{\eta}\mathcal{C}(\mathcal{O},\mathfrak{q})$. Now, we combine Lemma \ref{lem:positive_function} and Lemma \ref{lem:diff_form1} to estimate the integral \eqref{eq:positive}. Define 
\begin{equation*}
I_{\zeta,\delta,t}^{\op{O}}
 :=\int_{\pi(\mathcal{O},\Gamma)\in \op{supp}m}
 \int_{\begin{subarray}{l}\,\eta\in\mathcal{C}(\mathcal{O},\mathfrak{q})\\ |t\zeta-\eta|\leq \delta t\end{subarray}}
 (\RE f_{t\zeta})\nu^{\op{O}}dm.
\end{equation*}

\begin{lemma}\label{lem:integral_estimate} There exist $t_0>0$ and $\delta_0>0$ such that whenever $\delta_0>\delta>0$ and $t>t_0$, we have
\[|I_{\zeta,\delta,t}-I_{\zeta,\delta,t}^{\op{O}}|\leq \frac{1}{2}I_{\zeta,\delta,t}^{\op{O}}.\]
\end{lemma}
In the next section,
 we will see that $\nu^{\op{O}}$ is positive with respect to
 the given orientation of $\mathcal{C}(\mathcal{O},\mathfrak{q})$.
Using Lemma \ref{lem:positive_function} and Lemma \ref{lem:diff_form1}, we have the pointwise estimate
\begin{align*}
|f_{t\zeta}\nu-(\RE f_{t\zeta})\nu^{\op{O}}|&\leq |\IM f_{t\zeta}||\nu|+|\RE f_{t\zeta}(\nu-\nu^{\op{O}})|\\
&\leq \frac{1}{5}|\RE f_{t\zeta}| \cdot \frac{6}{5}|\nu^{\op{O}}|
 +|\RE f_{t\zeta}|\cdot \frac{1}{5}|\nu^{\op{O}}|\\
&\leq \frac{1}{2}|\RE f_{t\zeta}||\nu^{\op{O}}|.
\end{align*}
Combining this pointwise estimate with the positivity of $(\RE f_{t\zeta})\nu^{\op{O}}$ yields Lemma~\ref{lem:integral_estimate}.

Next, define
\begin{equation*}
I_{\zeta,\delta,t}^{\op{O},\frac{1}{2}}:=
\int_{\pi(\mathcal{O},\Gamma)\in \op{supp}m}
\int_{\begin{subarray}{l} \,\eta\in\mathcal{C}(\mathcal{O},\mathfrak{q})\\ |t\zeta-\eta|\leq \delta t^{1/2}\end{subarray}}
(\RE f_{t\zeta})\nu^{\op{O}}dm.
\end{equation*}
The following lemma will be proved in the next section.

\begin{lemma}\label{lem:form_lower_bound}
For any positive numbers $\delta>\delta'>0$, there exist $t_0$ and $C>0$ such that
\[\int_{\begin{subarray}{l}\,
 \eta\in\mathcal{C}(\mathcal{O}_{\lambda},\mathfrak{q})\\
 |t\zeta-\eta|\leq \delta t^{1/2}\end{subarray}} \nu^{\op{O}}
 \geq C \]
if $\zeta\in V$, $\lambda\in \sqrt{-1}Z(\mathfrak{l}_{\mathbb{R}})^{*,\mathfrak{q}}_{\rm gr}$,
 $|t\zeta-\lambda|< \delta' t^{1/2}$ and $t>t_0$.
\end{lemma}

We now complete the proof of Lemma~\ref{lem:inclusion2}.
Let $V$ be a compact neighborhood of $\xi$ in
 $\sqrt{-1}Z(\mathfrak{l}_{\mathbb{R}})^{*}_{\rm reg}$.
Then $\zeta_i\in V$ for sufficiently large $i$.
Take $\delta>0$ sufficiently small so it satisfies 
 $\delta<\delta_0$ in Lemma~\ref{lem:integral_estimate}.
To estimate $I_{\delta,t}^{\op{O},\frac{1}{2}}$,
 we see that $\RE f_{t\zeta}\geq C_{\delta}$
 if $|t\zeta-\eta|\leq \delta t^{1/2}$
 for a constant $C_{\delta}$.
Hence
\[
I_{\delta,t}^{\op{O},\frac{1}{2}}
\geq 
C_{\delta}
\int_{\pi(\mathcal{O},\Gamma)\in \op{supp}m}
\int_{\begin{subarray}{l}\,\eta\in\mathcal{C}(\mathcal{O},\mathfrak{q})\\ |t\zeta-\eta|\leq \delta t^{1/2}\end{subarray}}
\nu^{\op{O}}dm.
\]
Then by applying Lemma~\ref{lem:form_lower_bound} to $\zeta=\zeta_i$ and $t=t_i$, we have
\begin{equation*}
\int_{\pi(\mathcal{O},\Gamma)\in \op{supp}m}
\int_{\begin{subarray}{l}\,\eta\in\mathcal{C}(\mathcal{O},\mathfrak{q})\\
 |t_i\zeta_i-\eta|\leq \delta t_i^{1/2}\end{subarray}}
\nu^{\op{O}}dm
\geq 
C \int_{\begin{subarray}{l}\,\pi(\mathcal{O}_{\lambda},\Gamma)\in \op{supp}m\\
 |t_i\zeta_i-\lambda|<\delta't_i^{1/2}\end{subarray}}dm.
\end{equation*}
When $i$ is sufficiently large, we have
 $\delta't_i^{1/2}>1$.
Hence we have 
\begin{equation*}
\int_{\begin{subarray}{l}\,\pi(\mathcal{O}_{\lambda},\Gamma)\in \op{supp}m\\
 |t_i\zeta_i-\lambda|<\delta't_i^{1/2}\end{subarray}}dm
\geq \int_{\begin{subarray}{l}\, \pi(\mathcal{O}_{\lambda},\Gamma)\in \op{supp}m\\
 |t_i\zeta_i-\lambda|<1 \end{subarray}}dm
= p_*m(B_1(t_i\zeta_i))\geq 2^{-i-1}>t_i^{-1}
\end{equation*}
by \eqref{eq:measure_def2}.
Since $(\RE f_{t_i\zeta_i})\nu^{\op{O}}$ is positive, we have
$I_{\zeta_i,\delta,t_i}^{\op{O},\frac{1}{2}}\leq I_{\zeta_i,\delta,t_i}^{\op{O}}$.
Therefore, 
\[
I_{\zeta_i,\delta,t_i}^{\op{O}}\geq C_{\delta} C \cdot t_i^{-1}
\]
for sufficiently large $i$.
Combining with 
 \eqref{eq:contour_integral}, \eqref{eq:first_integral_bound}
 and Lemma~\ref{lem:integral_estimate},
 we deduce that there exists a constant $C>0$ such that
\[|(\theta(m)\cdot \mathcal{G}_{t_i+1})\sphat\,(t_i\zeta_i)|
 \geq C t_i^{-3/2}\] for sufficiently large $i$. 
By Lemma \ref{lem:FBI}, we have $\xi\in \op{WF}_0(\theta(m))$. 
Therefore, we obtain \eqref{eq:inclusion2a}
 and then Lemma \ref{lem:inclusion2}.


\section{Estimate of Kirillov-Kostant-Souriau form}\label{sec:KKSform}

The purpose of this section is to estimate the volume form
 on the contour $\mathcal{C}(\mathcal{O},\mathfrak{q})$
 defined by the Kirillov-Kostant-Souriau symplectic form
 and to prove Lemma~\ref{lem:diff_form1} and Lemma~\ref{lem:form_lower_bound}.

Recall that for an coadjoint orbit
 $\mathcal{O}_{\lambda}= G_{\mathbb{R}}\cdot \lambda$
 with $\lambda\in \sqrt{-1}Z(\mathfrak{l}_{\mathbb{R}})^*_{\rm gr}$, 
 and a polarization $\mathfrak{q}$, 
 the contour $\mathcal{C}(\mathcal{O}_{\lambda},\mathfrak{q})$ is defined as 
\[
\mathcal{C}(\mathcal{O}_{\lambda},\mathfrak{q})=
\{g\cdot \lambda+u\cdot \rho_{\mathfrak{l}}\mid
g\in G_{\mathbb{R}},\ u\in U,\ g\cdot \mathfrak{q}=u\cdot \mathfrak{q}\},
\]
which is a closed submanifold of the complex coadjoint orbit
 $G \cdot (\lambda+\rho_{\mathfrak{l}})$.
The tangent space of $\mathcal{C}(\mathcal{O}_{\lambda},\mathfrak{q})$ is given as
\[
T_{g\cdot \lambda+u\cdot \rho_{\mathfrak{l}}}
 \mathcal{C}(\mathcal{O}_{\lambda},\mathfrak{q})
=\{\op{ad}^*(X)(g\cdot \lambda)+\op{ad}^*(Y)(u\cdot \rho_{\mathfrak{l}})
 \mid X\in \mathfrak{g}_{\mathbb{R}},\ Y\in \mathfrak{u},\ X-Y\in g\cdot \mathfrak{q}\}.
\]
Then for each such $X$ and $Y$, there exists $Z\in \mathfrak{g}$
 such that 
\[
\op{ad}^*(X)(g\cdot \lambda)+\op{ad}^*(Y)(u\cdot \rho_{\mathfrak{l}})
=\op{ad}^*(Z)(g\cdot \lambda+u\cdot \rho_{\mathfrak{l}}).
\]
Recall that the Kirillov-Kostant-Souriau symplectic form $\omega$ on
 the complex coadjoint orbit $G\cdot(\lambda+\rho_{\mathfrak{l}})$ is defined by 
\[\omega_{\eta}(\op{ad}^*(Z)(\eta),\op{ad}^*(Z')(\eta)):=\eta([Z,Z'])\]
and then we defined a complex-valued $2n$-form 
$\nu:=(2\pi\sqrt{-1})^{-n}(n!)^{-1}\omega^{\wedge n}$,
 where $2n$ is the dimension of the orbit
 $G\cdot(\lambda+\rho_{\mathfrak{l}})$.

Let us define another $2$-form $\omega^O$
 on $\mathcal{C}(\mathcal{O}_{\lambda},\mathfrak{q})$.
Recall from \cite{HO20} that we have a fiber bundle structure
\begin{equation}\label{eq:fiber_bundle}
\varpi \colon
\mathcal{C}(\mathcal{O}_{\lambda},\mathfrak{q})
\ni g\cdot \lambda+u\cdot \rho_{\mathfrak{l}} \mapsto
 g\cdot \lambda\in \mathcal{O}_{\lambda}.
\end{equation}
The fiber over $\lambda$ is identified with
$(U\cap L)\cdot\rho_{\mathfrak{l}} \simeq (U\cap L)/(U\cap J)$.
For any $g_0\in G_{\mathbb{R}}$, there exists $u_0\in U$
 such that $g_0\cdot \mathfrak{q} = u_0\cdot \mathfrak{q}$.
Then the fiber $\varpi^{-1}(g_0\cdot \lambda)$ is identified with 
$(u_0(U\cap L))\cdot\rho_{\mathfrak{l}}$
and then with $(U\cap L)\cdot\rho_{\mathfrak{l}}$
 by the action of $u_0^{-1}$.

Let $\omega^{G_{\mathbb{R}}}_\lambda$
 (resp.\ $\omega^{U\cap L}_{\rho_{\mathfrak{l}}}$)
 denote the Kirillov-Kostant-Souriau form on the real coadjoint orbit
 $\mathcal{O}_{\lambda}=G_{\mathbb{R}}\cdot \lambda$
 (resp.\ $(U\cap L)\cdot \rho_{\mathfrak{l}}$).
To define $\omega^O$, we will decompose the tangent space 
$T_{\eta}\mathcal{C}(\mathcal{O}_{\lambda},\mathfrak{q})$
at $\eta=g\cdot \lambda+u\cdot \rho_{\mathfrak{l}}$ as
\[T_{\eta}\mathcal{C}(\mathcal{O}_{\lambda},\mathfrak{q})
= T_{\eta}^b\mathcal{C} \oplus T_{\eta}^f \mathcal{C}.\]
We define  $T_{\eta}^f \mathcal{C}$
 as the vectors that are tangent to the fiber of $\varpi$.
In other words,
\[
T_{\eta}^f \mathcal{C}
=\{\op{ad}^*(Y)(u\cdot \rho_{\mathfrak{l}})
 \mid Y\in \mathfrak{u}\cap (u\cdot \mathfrak{q})\}.
\]
To define $T_{\eta}^b \mathcal{C}$,
 consider the natural maps
\[
\mathfrak{g}_{\mathbb{R}}\to
\mathfrak{g} \to \mathfrak{g}/(g\cdot\mathfrak{q})
 \simeq \mathfrak{u}/(\mathfrak{u}\cap (g\cdot\mathfrak{q}))
\simeq (\mathfrak{u}\cap (g\cdot\mathfrak{q}))^{\perp},
\]
where $(\mathfrak{u}\cap (g\cdot\mathfrak{q}))^{\perp}$ is the orthogonal complement
 of $\mathfrak{u}\cap (g\cdot\mathfrak{q})$ in $\mathfrak{u}$
 with respect to an invariant form on $\mathfrak{u}$, which we fix now.
Write 
\[\varphi\colon \mathfrak{g}_{\mathbb{R}}\to (\mathfrak{u}\cap (g\cdot\mathfrak{q}))^{\perp}\]
for the composite map.
Then $X-\varphi(X)\in g\cdot \mathfrak{q}$ for any $X\in\mathfrak{g}_{\mathbb{R}}$.
Define
\[T_{\eta}^b \mathcal{C}
=\{\op{ad}^*(X)(g\cdot \lambda)+\op{ad}^*(\varphi(X))(u\cdot \rho_{\mathfrak{l}})
 \mid X\in \mathfrak{g}_{\mathbb{R}}\}.\]
$T_{\eta}^b \mathcal{C}$ can be identified with
 $T_{g\cdot\lambda} \mathcal{O}_{\lambda}$ via $\varpi$.
Define $\omega^O$ as the $2$-form on $\mathcal{C}(\mathcal{O}_{\lambda},\mathfrak{q})$ as 
\begin{equation*}
\omega^O|_{T_{\eta}^b \mathcal{C}} = \omega^{G_{\mathbb{R}}}_{\lambda},\quad
\omega^O|_{T_{\eta}^f \mathcal{C}} = \omega^{U\cap L}_{\rho_{\mathfrak{l}}},\quad
\omega^O(T_{\eta}^b \mathcal{C},T_{\eta}^f \mathcal{C})=0.
\end{equation*}
Here, we use the identifications
$T_{\eta}^b \mathcal{C}\simeq T_{g\cdot\lambda} \mathcal{O}_{\lambda}$
 and $\varpi^{-1}(g\cdot \lambda)\simeq \varpi^{-1}(\lambda)$.
Since $\lambda\in \sqrt{-1}\mathfrak{g}_{\mathbb{R}}^*$
 and $\rho_{\mathfrak{l}}\in \sqrt{-1}(\mathfrak{u}\cap\mathfrak{l}_{\R})^*$,
 the $2$-form $\omega^O$ is purely imaginary.
Then define a real-valued $2n$-form $\nu^O$ on
 $\mathcal{C}(\mathcal{O}_{\lambda},\mathfrak{q})$ as
\[
\nu^O:=\frac{(\omega^O)^{\wedge n}}{(2\pi\sqrt{-1})^n n!}.
\]
In \cite[Section 3.1]{HO20} an orientation on 
 $\mathcal{C}(\mathcal{O}_{\lambda},\mathfrak{q})$
 is defined in terms of symplectic forms $\omega^{G_{\mathbb{R}}}_{\lambda}$
 and $\omega^{U\cap L}_{\rho_{\mathfrak{l}}}$
 and the fiber bundle structure $\varpi$.
Then it directly follows from definition that $\nu^O$ is positive
 with respect to that orientation.

In the following, we estimate the differences $\omega-\omega^O$ and $\nu-\nu^O$
 to prove Lemma~\ref{lem:diff_form1}.

As in the previous section, we fix an inner product on $\mathfrak{g}$
 and let $|\cdot|$ denote the corresponding norm on $\mathfrak{g}$
 and on $\mathfrak{g}^*$.
For $A\in \op{End}(\mathfrak{g}^*)$
 let $\|A\|$ denote the corresponding operator norm.

We fix a compact set $V\subset \sqrt{-1}Z(\mathfrak{l}_{\mathbb{R}})^{*,\mathfrak{q}}$
 throughout this section.
We will estimate $\nu-\nu^O$ on
 $\mathcal{C}(\mathcal{O},\mathfrak{q})\cap B_{\delta t}(t\zeta)$,
 which is an open subset of $\mathcal{C}(\mathcal{O},\mathfrak{q})$, 
 for any $\zeta\in V$ and any $\mathcal{C}(\mathcal{O},\mathfrak{q})$
 when $\delta$ is sufficiently small and $t$ is sufficiently large.
Here, $B_{\delta t}(t\zeta)$ denotes the open ball with radius $\delta t$
 and center $t\zeta$ in $\mathfrak{g}$ with respect to our fixed norm on $\mathfrak{g}$.
For $\epsilon>0$, let
\begin{align*}
&B_{\epsilon}^{G_{\mathbb{R}}}
 :=\left\{g\in G_{\mathbb{R}}\mid \|\op{Ad}^*(g)-\it{id}_{\mathfrak{g}^*}\|<\epsilon\right\},\\ 
&B_{\epsilon}^{U}:=\left\{u\in U\mid \|\op{Ad}^*(u)-\it{id}_{\mathfrak{g}^*}\|<\epsilon\right\},\\
&B_{\epsilon}^{G}
 :=\left\{g\in G\mid \|\op{Ad}^*(g)-\it{id}_{\mathfrak{g}^*}\|<\epsilon\right\}.
\end{align*}

We need lemmas:
\begin{lemma}\label{lem:gu}
Given any $\epsilon>0$,
 there exist $\delta>0$ and $t_0>0$ such that the following holds:
 if $t>t_0$, $\zeta\in V$,
 $\lambda\in \sqrt{-1}Z(\mathfrak{l}_{\mathbb{R}})^{*,\mathfrak{q}}_{\rm gr}$, and
\[\eta\in \mathcal{C}(\mathcal{O}_{\lambda},\mathfrak{q})
 \cap B_{\delta t}(t\zeta),\]
then $|\lambda-t\zeta|<\epsilon t$ and there exist
 $g\in B_{\epsilon}^{G_{\mathbb{R}}}$,
 $u'\in B_{\epsilon}^{U}$ and $u_L\in U\cap L$
 such that 
\[g\cdot \mathfrak{q}=(u'u_L)\cdot \mathfrak{q} \ \text{ and }\ 
\eta = g \cdot \lambda + (u'u_L) \cdot \rho_{\mathfrak{l}}.\]
\end{lemma}

\begin{proof}
Consider the map
\[G_{\mathbb{R}}\times \sqrt{-1}\mathfrak{l}_{\mathbb{R}}^*
 \rightarrow \sqrt{-1}\mathfrak{g}_{\mathbb{R}}^*,
\quad (g,\eta)\mapsto g\cdot\eta, \]
 which is a submersion at $(e,t\zeta)$.
Define $\sqrt{-1}\mathfrak{l}_{\mathbb{R}}^{*,o}$
 as in the proof of Lemma~\ref{lem:asymptotic_cone}.

Take an open set $\widetilde{V}\subset \sqrt{-1}\mathfrak{l}_{\mathbb{R}}^{*,o}$
 which contains $V$. 
We claim that when $\widetilde{V}$ is sufficiently small, we have the following:
if $\lambda\in \sqrt{-1}Z(\mathfrak{l}_{\mathbb{R}})^{*,\mathfrak{q}}$, 
 $\lambda'\in \widetilde{V}\cap \sqrt{-1}Z(\mathfrak{l}_{\mathbb{R}})^*$
 and $g\cdot \lambda=\lambda'$ for some $g\in G_{\mathbb{R}}$,
 then $\lambda=\lambda'$.
Indeed, if this is not the case, we may find sequences
 $\lambda_j\in \sqrt{-1}Z(\mathfrak{l}_{\mathbb{R}})^{*,\mathfrak{q}}$,
 $\lambda'_j\in \sqrt{-1}Z(\mathfrak{l}_{\mathbb{R}})^*$
 and $w_j \in W_{\mathbb{R}}$
 such that $w_j\cdot \lambda_j=\lambda'_j$,
 $w_j|_{Z(\mathfrak{l}_{\R})^*}\neq 1$ and $\lambda'_j\to \lambda'\in V$.
Here, $W_{\mathbb{R}}=N_{G_{\mathbb{R}}}(\mathfrak{j}_{\mathbb{R}})
 /Z_{G_{\mathbb{R}}}(\mathfrak{j}_{\mathbb{R}})$ denotes the real Weyl group.
By taking a subsequence, we may assume $\lambda_j$ has a limit $\lambda$
 and that $w_j=w$ for all $j$.
Then we have $w\cdot\lambda=\lambda'$ with
  $\lambda \in \overline{\sqrt{-1}Z(\mathfrak{l}_{\mathbb{R}})^{*,\mathfrak{q}}}$, 
 $\lambda'\in \sqrt{-1}Z(\mathfrak{l}_{\mathbb{R}})^{*,\mathfrak{q}}$, and 
 $w|_{Z(\mathfrak{l}_{\R})^*}\neq 1$.
It is easy to see from the definition of 
 $\sqrt{-1}Z(\mathfrak{l}_{\mathbb{R}})^{*,\mathfrak{q}}$
 that this is not possible.
Thus, the claim is proved.

Take $\widetilde{V}$ that satisfies above claim. 
For any $\epsilon'>0$ with $\epsilon>\epsilon'$,
 there exists $\delta'>0$ such that 
\[ B_{\epsilon'}^{G_{\mathbb{R}}}\cdot \widetilde{V}\supset
 \bigcup_{\zeta\in V} B_{\delta'}(\zeta).\]
Scaling everything by $t$ yields
\[B_{\epsilon'}^{G_{\mathbb{R}}}\cdot (t \widetilde{V})\supset
 \bigcup_{\zeta\in V} B_{\delta' t}(t\zeta).\]
Let $c=\sup_{u\in U} |u\cdot \rho_{\mathfrak{l}}|$, and fix $0<\delta<\delta'$. 
Then we may find $t_0>0$ sufficiently large such that $c+\delta t < \delta' t$ if $t>t_0$. 

Now, suppose that 
$t>t_0$, $\zeta\in V$,
 $\lambda\in \sqrt{-1}Z(\mathfrak{l}_{\mathbb{R}})^{*,\mathfrak{q}}_{\rm gr}$, and
 $\eta\in \mathcal{C}(\mathcal{O}_{\lambda},\mathfrak{q})
 \cap B_{\delta t}(t\zeta)$.
Then by the definition of $\mathcal{C}(\mathcal{O}_{\lambda},\mathfrak{q})$,
 we may write $\eta=g\cdot \lambda+u\cdot \rho_{\mathfrak{l}}$
 such that $g\in G_{\mathbb{R}}$, $u\in U$, and $g\cdot \mathfrak{q}=u\cdot \mathfrak{q}$.
We have 
\[|g\cdot \lambda - t\zeta|\leq |u\cdot \rho_{\mathfrak{l}}|+ |\eta-t\zeta|< \delta' t.\]
Hence $g\cdot \lambda \in B_{\epsilon'}^{G_{\mathbb{R}}}\cdot(t \widetilde{V})$
 and we can write $g\cdot \lambda=g'\cdot \lambda'$ with $g'\in B_{\epsilon'}^{G_{\mathbb{R}}}$
 and $\lambda'\in t\widetilde{V}$.
Then $\mathfrak{g}(\lambda)(=\mathfrak{l})$ and $\mathfrak{g}(\lambda')$ are conjugate
 and $\mathfrak{g}(\lambda')\subset \mathfrak{l}$.
Therefore, $\mathfrak{g}(\lambda')= \mathfrak{l}$ and
 $\lambda'\in \sqrt{-1} Z(\mathfrak{l}_{\mathbb{R}})^*$. 
Since $\lambda\in \sqrt{-1} Z(\mathfrak{l}_{\mathbb{R}})^{*,\mathfrak{q}},
 \lambda'\in t\widetilde{V}\cap \sqrt{-1} Z(\mathfrak{l}_{\mathbb{R}})^{*}$
 and they are in the same $G_{\mathbb{R}}$-orbit,
 the claim at the beginning of the proof implies $\lambda=\lambda'$.
Then $g^{-1}g'\in L_{\mathbb{R}}$ and we may replace $g$ by $g'$.
We may thus assume $g\in B_{\epsilon'}^{G_{\mathbb{R}}}$.

Let $Y$ denote the partial flag variety,
 the collection of all parabolic subalgebras of $\mathfrak{g}$
 that are $G$-conjugate to $\mathfrak{q}$.
Then $B_{\epsilon'}^{G_{\mathbb{R}}}\cdot \mathfrak{q}$
 is a small open neighborhood of $\mathfrak{q}$ in $Y$.
If $\epsilon'$ is small enough, then we can take $u'\in B_{\epsilon}^{U}$
 such that $g\cdot \mathfrak{q}=u'\cdot \mathfrak{q}$.
Then $u_L:=(u')^{-1}\cdot u$ satisfies $u_L\cdot \mathfrak{q}=\mathfrak{q}$
 and hence $u_L\in U\cap L$.

Moreover,
\[|\lambda-t\zeta|\leq |\lambda-g\cdot \lambda|+|g\cdot \lambda - t\zeta|
 < \epsilon'|\lambda|+ \delta' t
 \leq \epsilon'|\lambda-t\zeta| + \epsilon' t|\zeta| + \delta' t.\]
By decreasing $\epsilon'$ and $\delta'$ if necessary 
 we deduce that $|\lambda-t\zeta|<\epsilon t$.
\end{proof}

Note that there exists $d>0$ such that
 if $\delta$ is sufficiently small
 and $t$ is sufficiently large, then
 $\lambda\in B_{\delta t}(t\zeta)$ with $\zeta\in V$
 and $\lambda\in\sqrt{-1}Z(\mathfrak{l}_{\mathbb{R}})^*$ implies that
\begin{equation}\label{eq:regular_condition}
|\langle \lambda+\rho_{\mathfrak{l}}, \alpha^{\vee}\rangle|\geq d|\lambda|\ 
 (\forall\alpha\in \Delta(\mathfrak{n},\mathfrak{j})) \ \text{ and }\ 
|\lambda|\geq 2|\rho_{\mathfrak{l}}|.
\end{equation}
Here, $\mathfrak{n}$ is the nilradical of $\mathfrak{q}$.
We fix such $d$.

\begin{lemma}\label{lem:g_c}
Let $0<\epsilon<1$ and
let $\lambda\in\sqrt{-1}Z(\mathfrak{l}_{\mathbb{R}})^*$
 which satisfies \eqref{eq:regular_condition}.
Let $g\in B_{\epsilon}^{G_{\mathbb{R}}}$,
 $u'\in B_{\epsilon}^{U}$ and $u_L\in U\cap L$ such that
\[ g\cdot \mathfrak{q}=(u'u_L)\cdot \mathfrak{q}\ \text{ and }\ 
\eta = g \cdot \lambda + (u'u_L) \cdot \rho_{\mathfrak{l}}.\]
Then there exist $d'>0$ and $g_c\in B_{d'\epsilon}^{G}$ such that 
\[\eta = (g_c u_L) \cdot (\lambda+\rho_{\mathfrak{l}}).\]
The constant $d'$ depends only on $d$
 and does not depend on $\lambda$ or $\epsilon$.
\end{lemma}

\begin{proof}
Let $Q$ be the parabolic subgroup of $G$ with Lie algebra $\mathfrak{q}$,
 or equivalently the normalizer of $\mathfrak{q}$.
By the assumption $g\cdot \mathfrak{q}=(u'u_L)\cdot \mathfrak{q}$,
 we have $(u'u_L)^{-1}g\in Q$.
Then $(u'u_L)^{-1}g\cdot \lambda-\lambda\in \mathfrak{n}$
 with the identification $\mathfrak{g}\simeq \mathfrak{g}^*$.
By our assumption on $g$ and $u'$,
 we have $(u'u_L)^{-1}gu_L = u_L^{-1} (u')^{-1} gu_L \in B_{c_1\epsilon}^G$
 for some constant $c_1>0$.
Then 
\[|(u'u_L)^{-1}g \cdot \lambda-\lambda|
=|(u'u_L)^{-1}gu_L\cdot \lambda-\lambda|< c_1 \epsilon|\lambda|.\]

Decompose $\mathfrak{n}$ into root spaces
\[
\mathfrak{n}=\bigoplus_{i=1}^{k}\mathfrak{n}_i\ \text{ and put }
\mathfrak{n}_{>j}:=\bigoplus_{i>j}\mathfrak{n}_i.
\]
The ordering is chosen to satisfy $[\mathfrak{n}_i,\mathfrak{n}]\subset\mathfrak{n}_{>i}$.
We claim that for any $1\leq i\leq k$,
 there exist a constant $d_i>0$ and $g_c^i\in B_{d_i\epsilon}^{G}$
such that 
\begin{equation}\label{eq:gci}
\bigl(g_c^i\cdot (\lambda+\rho_{\mathfrak{l}})-(\lambda+\rho_{\mathfrak{l}}\bigr))
-\bigl((u'u_L)^{-1}g\cdot \lambda-\lambda\bigr)
\in\mathfrak{n}_{>i}.
\end{equation}
This can be seen by induction on $i$.
Given $g_c^{i-1}$, we can find $g_c^i=\exp(N_i)g_c^{i-1}$ with $N_i\in \mathfrak{n}_i$ 
 which satisfies \eqref{eq:gci}.
Moreover, it follows from \eqref{eq:regular_condition}
 that $|N_i|$ is bounded by the product of $\epsilon$
 and a constant.
Hence we get $g_c^i\in B_{d_i\epsilon}^{G}$ for some constant $d_i$.

The claim for $i=k$ yields 
\[g_c^{k}\cdot (\lambda+\rho_{\mathfrak{l}})-(\lambda+\rho_{\mathfrak{l}})
=(u'u_L)^{-1}g\cdot \lambda-\lambda.
\]
Then putting $g_c:=u'u_Lg_c^{k}u_L^{-1}$ we get
\[g_cu_L \cdot (\lambda+\rho_{\mathfrak{l}})
=g\cdot \lambda+(u'u_L)\cdot \rho_{\mathfrak{l}}.
\]
By $u'\in B_{\epsilon}^{U}$ and $g_c^{k}\in B_{d_{k}\epsilon}^{G}$
 we can choose a constant $d'$ such that $g_c\in B_{d'\epsilon}^G$.
\end{proof}

Fix vectors $X_1^o,\dots,X_{2k}^o$ in $\mathfrak{g}_{\mathbb{R}}$
 which form a basis of $\mathfrak{g}_{\mathbb{R}}/\mathfrak{l}_{\mathbb{R}}$.
We have 
\[\op{ad}^*(g\cdot X_i^o)(g\cdot \lambda)
 +\op{ad}^*(\varphi(g\cdot X_i^o))(u\cdot \rho_{\mathfrak{l}})\in
 T_{\eta}^b\mathcal{C}, \]
where $\eta=g\cdot\lambda+u\cdot\rho_{\mathfrak{l}}$.
We take
$X_i\in\mathfrak{g}$ such that 
\begin{equation}\label{eq:Xi}
\op{ad}^*(g\cdot X_i^o)(g\cdot \lambda)
 +\op{ad}^*(\varphi(g\cdot X_i^o))(u\cdot\rho_{\mathfrak{l}})
=\op{ad}^*(X_i)(g\cdot\lambda+u\cdot\rho_{\mathfrak{l}})
\end{equation}
for $1\leq i\leq 2k$.

Next, fix vectors $Y_1^o,\dots,Y_{2l}^o$ in $\mathfrak{u}\cap \mathfrak{l}$
 which form a basis in $(\mathfrak{u}\cap \mathfrak{l})/(\mathfrak{u}\cap \mathfrak{j})$.
Then 
\[\op{ad}^*(u\cdot Y_i^o)(u\cdot \rho_{\mathfrak{l}})\in T_{\eta}^f\mathcal{C}.\]
We take $Y_i\in \mathfrak{g}$ such that
\begin{equation}\label{eq:Yi}
\op{ad}^*(u\cdot Y_i^o)(u\cdot\rho_{\mathfrak{l}})
=\op{ad}^*(Y_i)(g\cdot \lambda+u\cdot\rho_{\mathfrak{l}})
\end{equation}
for $1\leq i\leq 2l$.

Define $Z_i\in\mathfrak{g}$ for $1\leq i\leq 2k+2l=2n$ as
\[
Z_i:=X_i\ (1\leq i\leq 2k),\quad Z_{2k+i}:=Y_i\ (1\leq i\leq 2l).
\]
The vectors
$\op{ad}^*(Z_i)(g\cdot \lambda+u\cdot \rho_{\mathfrak{l}})$
 form a basis of the tangent space $T_{\eta}\mathcal{C}(\mathcal{O},\mathfrak{q})$.
Let $A$ be a $2n$ by $2n$ matrix whose $(i,j)$ entry is
 $\omega_{\eta}(\op{ad}^*(Z_i)(\eta),\op{ad}^*(Z_j)(\eta))=\eta([Z_i,Z_j])$.
Then $A$ is skew symmetric and 
 the $2n$-form $\nu=(2\pi\sqrt{-1})^{-n} (n!)^{-1}\omega^{\wedge n}$
 is given by
\[\nu(Z_1,\dots,Z_{2n})=(2\pi\sqrt{-1})^{-n} \op{Pf}(A),\]
where $\op{Pf}(A)$ denotes the Pfaffian of $A$.

We now estimate each entry of $A$:

\begin{lemma}\label{lem:form_estimate}
Let $\epsilon, d>0$.
Suppose that $\lambda\in \sqrt{-1}Z(\mathfrak{l}_{\mathbb{R}})^{*,\mathfrak{q}}_{\rm gr}$
 satisfying \eqref{eq:regular_condition}, 
 $g\in B_{\epsilon}^{G_{\mathbb{R}}}$,
 $u'\in B_{\epsilon}^{U}$, and 
 $u_L\in U\cap L$ such that $g\cdot \mathfrak{q}=u'u_L\cdot \mathfrak{q}$.
Define $X_i$ and $Y_i$ as above for $g$ and $u:=u'u_L$.
Then we have 
\begin{enumerate}
\item
$|(g\cdot \lambda+u\cdot \rho_{\mathfrak{l}})([X_i,X_j])
 -\lambda([X_i^o,X_j^o])|\leq C$,
\item
$|(g\cdot \lambda+u\cdot \rho_{\mathfrak{l}})([X_i,Y_j])|\leq C$,
\item
$|(g\cdot \lambda+u\cdot \rho_{\mathfrak{l}})([Y_i,Y_j])
 -\rho_{\mathfrak{l}}([Y_i^o,Y_j^o])|\leq \epsilon C$
\end{enumerate}
for some constant $C>0$.
Here, $C$ depends on $\epsilon$ and $d$, but does not depend on
 $\lambda$, $g$, $u'$ or $u_L$.
\end{lemma}

\begin{proof}
By Lemma~\ref{lem:g_c}, there exists $g_c\in B_{d' \epsilon}^G$ such that 
\[g_cu_L\cdot (\lambda+\rho_{\mathfrak{l}})
=g\cdot \lambda+u\cdot \rho_{\mathfrak{l}}.
\]
In the following proof,
 we say a vector in $\mathfrak{g}$ or an element in $G$ is \emph{bounded}
 if it lies in a compact set which depends only on $\epsilon$ and $d$. 
For instance $g$, $u'$, and $u_L$ are bounded,
 but $\lambda$ is not bounded.

Consider the equation
\begin{align}\label{eq:X'i}
-\op{ad}^*(g\cdot X_i^o)(u\cdot \rho_{\mathfrak{l}})
 +\op{ad}^*(\varphi(X_i^o))(u\cdot\rho_{\mathfrak{l}})
&=\op{ad}^*(X'_i)(g\cdot \lambda+u\cdot\rho_{\mathfrak{l}}) \\ \nonumber
\bigl(&=\op{ad}^*(X'_i)(g_cu_L\cdot (\lambda+\rho_{\mathfrak{l}}))\bigr).
\end{align}
If we put $X'_i:=X_i-g\cdot X_i^o$, then this is equivalent to \eqref{eq:Xi}.
In particular, \eqref{eq:X'i} is satisfied for at least one $X'_i$
 and hence the left hand side of \eqref{eq:X'i}
 is contained in $g_cu_L\cdot [\mathfrak{g},\mathfrak{j}]$
 with the identification $\mathfrak{g}\simeq \mathfrak{g}^*$.
Since the left hand side of \eqref{eq:X'i} is bounded and 
 $g_cu_L$ is bounded, the first condition of \eqref{eq:regular_condition}
 implies that there exists a bounded vector $X'_i$ which
 satisfies \eqref{eq:X'i}.
Then by putting $X_i=X'_i+g\cdot X_i^o$, we find a bounded vector $X_i$
 which satisfies \eqref{eq:Xi}.
Note that by \eqref{eq:X'i} again,
 $\op{ad}^*(X'_i)(g\cdot \lambda)$ is also bounded.

We may thus assume that $X_i$ are bounded vectors. 
To prove (1), it is enough to show that
 $(g\cdot \lambda) ([X_i,X_j])-\lambda([X_i^o,X_j^o])$ is bounded.
We calculate
\begin{align*}
&(g\cdot \lambda) ([X_i,X_j])-\lambda([X_i^o,X_j^o])\\
&=(g\cdot \lambda) ([g\cdot X_i^o+X'_i,\: g\cdot X_j^o+X'_j])- \lambda([X_i^o,X_j^o])\\
&=(g\cdot \lambda) ([X'_i,\: g\cdot X_j^o])
  +(g\cdot \lambda) ([g\cdot X_i^o, X'_j])
  +(g\cdot \lambda) ([X'_i,X'_j])\\
&=-\langle\op{ad}^*(X'_i)(g\cdot \lambda),\: g\cdot X_j^o \rangle
  +\langle\op{ad}^*(X'_j)(g\cdot \lambda),\: g\cdot X_i^o \rangle
  -\langle\op{ad}^*(X'_i)(g\cdot \lambda),\: g\cdot X'_j \rangle.
\end{align*}
The last three terms are all bounded and (1) is proved.

Since the left hand side of \eqref{eq:Yi} is bounded, we may assume that $Y_i$ is also bounded.
For example, if we take $Y_i$ from $g_cu_L\cdot [\mathfrak{g},\mathfrak{j}]$,
 then by \eqref{eq:regular_condition} $Y_i$ is bounded.
Moreover, we claim that 
\begin{equation}\label{eq:Yi_Yio}
\epsilon^{-1}|\op{ad}^*(Y_i-(g_c u_L)\cdot Y_i^o)(g\cdot \lambda+u\cdot\rho_{\mathfrak{l}})|
\end{equation}
is bounded.
Indeed, 
\begin{align*}
&\op{ad}^*(Y_i)(g\cdot \lambda+u\cdot\rho_{\mathfrak{l}})
 -\op{ad}^*((g_c u_L)\cdot Y_i^o)(g\cdot \lambda+u\cdot\rho_{\mathfrak{l}})\\
&=\op{ad}^*(u\cdot Y_i^o)(u\cdot \rho_{\mathfrak{l}})
 - (g_c u_L)\cdot \bigl(\op{ad}^*(Y_i^o)(\lambda+\rho_{\mathfrak{l}})\bigr)\\
&= (u'u_L)\cdot \bigl(\op{ad}^*(Y_i^o)(\rho_{\mathfrak{l}})\bigr)
 - (g_c u_L)\cdot \bigl(\op{ad}^*(Y_i^o)(\rho_{\mathfrak{l}})\bigr).
\end{align*}
Here, we used $\op{ad}^*(Y_i^o)(\lambda)=0$ which follows from $Y_i^o\in\mathfrak{l}$.
Then the claim follows from $g_c\in B_{c_1\epsilon}^{G}$ and $u'\in B_{\epsilon}^{U}$.

(2) follows from
\begin{align*}
(g\cdot \lambda+u\cdot \rho_{\mathfrak{l}})([X_i,Y_j])
&= \langle X_i, \op{ad}^*(Y_j)(g\cdot \lambda+u\cdot \rho_{\mathfrak{l}})\rangle \\
&= \langle X_i, \op{ad}^*(u\cdot Y_j^o)(u\cdot \rho_{\mathfrak{l}})\rangle.
\end{align*}

For (3), put $Y'_i:=Y_i-(g_c u_L)\cdot Y_i^o$.
Then
\begin{align*}
&(g\cdot \lambda+u\cdot \rho_{\mathfrak{l}})([Y_i,Y_j])\\
&=(g\cdot \lambda+u\cdot \rho_{\mathfrak{l}})
 ([Y'_i+(g_c u_L)\cdot Y_i^o,\: Y'_j+(g_c u_L)\cdot Y_j^o])\\
&= (g\cdot \lambda+u\cdot \rho_{\mathfrak{l}})
 ([(g_c u_L)\cdot Y_i^o,\:(g_c u_L)\cdot Y_j^o])
 -\langle (g_c u_L)\cdot Y_j^o,\: \op{ad}^*(Y'_i)(g\cdot \lambda+u\cdot\rho_{\mathfrak{l}})\rangle\\
&\qquad\qquad
 +\langle (g_c u_L)\cdot Y_i^o,\: \op{ad}^*(Y'_j)(g\cdot \lambda+u\cdot\rho_{\mathfrak{l}})\rangle
 +\langle Y'_i,\: \op{ad}^*(Y'_j)(g\cdot \lambda+u\cdot\rho_{\mathfrak{l}})\rangle.
\end{align*}
Since \eqref{eq:Yi_Yio} is bounded, 
 the last three terms are all bounded by $\epsilon C$ for some constant $C$.
The first term is calculated as 
\begin{align*}
&(g\cdot \lambda+u\cdot \rho_{\mathfrak{l}})
 ([(g_c u_L)\cdot Y_i^o,\: (g_c u_L)\cdot Y_j^o]) \\
&= ((g_cu_L)\cdot (\lambda+ \rho_{\mathfrak{l}}))
 ([(g_c u_L)\cdot Y_i^o,\: (g_c u_L)\cdot Y_j^o])\\
&= (\lambda+ \rho_{\mathfrak{l}})([Y_i^o,Y_j^o])\\
&= \rho_{\mathfrak{l}}([Y_i^o,Y_j^o]).
\end{align*}
(3) is thus proved.
\end{proof}

We now prove Lemma~\ref{lem:diff_form1}, namely, we prove
\[
\bigl|(\nu_{\eta}-\nu^O_{\eta})
\bigl(\op{ad}^*(Z_1)(\eta),\dots,\op{ad}^*(Z_{2n})(\eta)\bigr)\bigr|
 \leq \frac{1}{5}
 \bigl| \nu^O_{\eta}
 \bigl(\op{ad}^*(Z_1)(\eta),\dots,\op{ad}^*(Z_{2n})(\eta)\bigr)\bigr|
\]
on $\mathcal{C}(\mathcal{O},\mathfrak{q})\cap B_{\delta t}(t \zeta)$
 when $\delta$ is sufficiently small and 
 $t$ is sufficiently large, or equivalently,
 $|\lambda|$ is sufficiently large.
Since $\nu$ and $\nu^O$ are differential forms of top degree,
 it is enough to prove the inequality
 for our particular basis  $\op{ad}^*(Z_1)(\eta),\dots,\op{ad}^*(Z_{2n})(\eta)$
 of the tangent space chosen above.

Similarly to the matrix $A$, 
 let $A^O$ be a $2n$ by $2n$ matrix whose $(i,j)$ entry is
 $\omega^O_{\eta}(\op{ad}^*(Z_i)(\eta),\op{ad}^*(Z_j)(\eta))$.
We have
\[\nu^O_{\eta}\bigl(\op{ad}^*(Z_1)(\eta),\dots,\op{ad}^*(Z_{2n})(\eta)\bigr)
 =(2\pi\sqrt{-1})^{-n}\op{Pf}(A^O). \]
Hence it is enough to prove
\begin{equation}\label{eq:pfaffian}
|\op{Pf}(A)-\op{Pf}(A^O)|\leq \frac{1}{5}|\op{Pf}(A^O)|.
\end{equation}

By definition of $\omega^O$, the matrix $A^O$ is block diagonal and
 each entry does not depend on $\eta$. 
The upper left $2k$ by $2k$ part of $A^O$ is $\lambda([X_i^o,X_j^o])$.
The lower right $2l$ by $2l$ part is $\rho_{\mathfrak{l}}([Y_i^o,Y_j^o])$.
Since the the Kirillov-Kostant-Souriau form is nondegenerate, 
 the Pfaffian of $A^O$ does not vanish.
Assuming \eqref{eq:regular_condition}, 
 the $\op{Pf}(A^O)$ grows exactly of order $|\lambda|^k$, namely, 
 there exist constants $C_1,C_2>0$ such that
\[
C_1|\lambda|^k\leq |\op{Pf}(A^O)|\leq C_2|\lambda|^k.
\]
In light of the estimate of the entries of $A-A^O$ given in Lemma~\ref{lem:form_estimate}, 
there exist $C_3,C_4>0$ such that
\[|\op{Pf}(A)-\op{Pf}(A^O)|\leq C_3 |\lambda|^{k-1}+C_4\epsilon |\lambda|^k.\]
Therefore, for sufficiently small $\epsilon$ and sufficiently large $|\lambda|$
 we obtain \eqref{eq:pfaffian}.
We fix such sufficiently small $\epsilon>0$ and
 then Lemma~\ref{lem:gu} gives $\delta>0$.
By decreasing $\delta$ if necessary to have \eqref{eq:regular_condition},
 we conclude that the inequality in Lemma~\ref{lem:diff_form1}
 holds for sufficiently large $t$.

It remains to prove Lemma~\ref{lem:form_lower_bound}.
For this we use the fiber bundle structure
 $\varpi\colon \mathcal{C}(\mathcal{O}_{\lambda},\mathfrak{q})
  \to \mathcal{O}_{\lambda}$
 as \eqref{eq:fiber_bundle}.
We have a canonical volume form
\[\nu^{U\cap L}_{\rho_{\mathfrak{l}}}
 :=\frac{(\omega^{U\cap L}_{\rho_{\mathfrak{l}}})^{\wedge l}}{(2\pi\sqrt{-1})^{l} l!}
\]
 on the fiber $\varpi^{-1}(\lambda)$
 and then on any fiber by an isomorphism
 $\varpi^{-1}(g_0\cdot \lambda)\simeq \varpi^{-1}(\lambda)$.
The volume of the fiber with respect to this form is a constant,
 which we denote by $c$.
Then for an open subset $B\subset \mathcal{O}_{\lambda}$, we have
\begin{equation}\label{eq:volume_fiber}
\int_{\varpi^{-1}(B)} \nu^O 
= c \int_{B} \nu^{G_{\mathbb{R}}}_{\lambda},
\end{equation}
where we put 
\[\nu^{G_{\mathbb{R}}}_{\lambda} :=
 \frac{(\omega^{G_{\mathbb{R}}}_{\lambda})^{\wedge k}}{(2\pi\sqrt{-1})^{k} k!}.
\]

To study the volume form on the base $\mathcal{O}_{\lambda}$,
 we fix a constant $d>0$ and assume 
\begin{equation}\label{eq:regular_condition2}
|\langle \lambda, \alpha^{\vee}\rangle|\geq d|\lambda|\ 
 (\forall\alpha\in \Delta(\mathfrak{n},\mathfrak{j})) \ \text{ and }\ 
|\lambda|\geq 2|\rho_{\mathfrak{l}}|.
\end{equation}
Let $\mathfrak{l}_{\mathbb{R}}^{\perp}$ be the orthogonal complement
 of $\mathfrak{l}_{\mathbb{R}}$ in $\mathfrak{g}_{\mathbb{R}}$
 and fix a basis $X_1^o,\dots,X_{2k}^o$ of $\mathfrak{l}_{\mathbb{R}}^{\perp}$.
Let $x_1,\dots,x_{2k}$ be linear coordinate functions
 on $\mathfrak{l}_{\mathbb{R}}^{\perp}$ with respect to this basis. 
Then we have a natural map
\[\psi\colon
\mathfrak{l}_{\mathbb{R}}^{\perp}\to \mathcal{O}_{\lambda},
\quad X\mapsto \exp(X)\cdot \lambda.
\]
Under the assumption \eqref{eq:regular_condition2},
 there exists $0<\epsilon'<1$ which does not depend on $\lambda$ such that
 $\psi\colon B_{\epsilon'}\to \psi(B_{\epsilon'})$ 
 is a diffeomorphism,
 where $B_{\epsilon'}$ is the open ball in $\mathfrak{l}_{\mathbb{R}}^{\perp}$
 with center at origin and radius $\epsilon'$ with respect to our linear coordinate.
Decreasing $\epsilon'$ if necessary, we may further assume that
 $\psi$ restricted to some open set containing the closure of $B_{\epsilon'}$
 is a diffeomorphism onto its image.
Moreover, $\psi(B_{\epsilon'})\subset B_{C_1\epsilon' |\lambda|}(\lambda)$
 for some constant $C_1$.
We claim that
\[
|\psi^* \nu^{G_{\mathbb{R}}}_{\lambda}|
\geq C_2|\lambda|^k|dx_1\wedge \cdots \wedge dx_{2k}|
\]
on $\psi(B_{\epsilon'})$ for some constant $C_2>0$.
Indeed, we can find such $C_2$ when $|\lambda|$ is bounded.
Then the claim follows because
 $|\lambda|^{-k}|\psi^* \nu^{G_{\mathbb{R}}}_{\lambda}|$
 is invariant under the scaling $\lambda\to a\lambda\ (a>0)$.
Therefore, we have
\[
\int_{B_{C_1\epsilon' |\lambda|}(\lambda)} \nu^{G_{\mathbb{R}}}_{\lambda}
\geq 
C_2 |\lambda|^k \int_{B_{\epsilon'}} |dx_1\wedge \cdots \wedge dx_{2k}|
\geq 
C_3 (\epsilon')^{2k}|\lambda|^k
\]
for some constant $C_3>0$.

Combining with \eqref{eq:volume_fiber}, we obtain the following.
\begin{lemma}\label{lem:volume_estimate}
There exist positive numbers $\epsilon_0$ and $C$ such that
\[\int_{\varpi^{-1}(B_{\epsilon |\lambda|}(\lambda))} \nu^O
\geq C\epsilon^{2k}|\lambda|^k \]
for $0<\epsilon<\epsilon_0$ and
 any $\lambda$ satisfying \eqref{eq:regular_condition2}.
\end{lemma}

To prove Lemma~\ref{lem:form_lower_bound}, fix positive numbers
 $\delta>\delta'>0$.
If $t$ is sufficiently large,
 then $\zeta\in V$ and
 $|t\zeta-\lambda|< \delta' t^{1/2}$ imply that
 $\lambda$ satisfies \eqref{eq:regular_condition2}.
Moreover, $t^{-1}|\lambda|$ is bounded
 from below and from above by positive constants.
Define $\epsilon$ by the equation
\[
\delta t^{\frac{1}{2}}=\epsilon|\lambda|
 + \max_{u\in U} |u\cdot \rho_{\mathfrak{l}}| + \delta' t^{\frac{1}{2}}.
\]
When $t$ becomes larger,
 $|\lambda|$ is of order $t$ and $\epsilon$ is of order $t^{-\frac{1}{2}}$.
Hence if $t$ is sufficiently large, 
 then $\epsilon$ becomes arbitrarily small positive number.
By the inclusion $\varpi^{-1}(B_{\epsilon|\lambda|}(\lambda))
 \subset B_{\delta t^{1/2}}(t\zeta)$
 and by Lemma~\ref{lem:volume_estimate}, we have
\[
\int_{\begin{subarray}{l}\, \eta\in \mathcal{C}(\mathcal{O}_{\lambda},\mathfrak{q})\\
 |\eta-t\zeta|\leq \delta t^{1/2}\end{subarray}} \nu^O
\geq
\int_{\varpi^{-1}(B_{\epsilon |\lambda|}(\lambda))} \nu^O
\geq C\epsilon^{2k}|\lambda|^k. 
\]
Since $\epsilon^2|\lambda|$ is bounded from below by a positive constant,
 we obtain Lemma~\ref{lem:form_lower_bound}.


\section{Proof of main theorems}\label{sec:proof}

In this section, we prove Theorem~\ref{thm:main}, Theorem~\ref{thm:main2}
 and Theorem~\ref{thm:ds}.

Suppose that $X_0=G_{\mathbb{R}}/H_0$ is a locally algebraic
 homogeneous space with $G_{\mathbb{R}}$-invariant density.
Our proof depends on the following result of the wave front set of induced representation:
\begin{theorem}[{\cite[Theorem 2.1]{HW17}}]\label{thm:wavefront}
Let $\mu\colon T^*X_0\to \mathfrak{g}_{\mathbb{R}}^*$ be the moment map.
Then
\[\op{WF}(L^2(X_0))=\overline{\sqrt{-1}\mu(T^*X_0)}.\]
\end{theorem}

First, we prove Theorem~\ref{thm:main2}.
According to Theorem~\ref{thm:reduction_semisimple},
 we can divide the set $\op{supp} L^2(X_0)$ as
\begin{equation}\label{eq:incl_support}
\op{supp} L^2(X_0)
\subset \widehat{G}_{\mathbb{R}}(\mathfrak{l}_X,d)
 \cup \bigcup_{\mathfrak{l}_{\mathbb{R}}}
 \widehat{G}_{\mathbb{R}}^{\mathfrak{l}_{\mathbb{R}}}
\end{equation}
for some constant $d$,
where $\mathfrak{l}_{\mathbb{R}}$
 runs over representatives of all $G_{\mathbb{R}}$-conjugacy classes
 such that $\mathfrak{l}(=\mathfrak{l}_{\mathbb{R}}\otimes \mathbb{C})$
 is $G$-conjugate to $\mathfrak{l}_X$. 
If $d$ is large enough, 
 $\pi(\mathfrak{l}_{\mathbb{R}},\Gamma_{\lambda})\in 
 \widehat{G}_{\mathbb{R}}^{\mathfrak{l}_{\mathbb{R}}}\setminus
 \widehat{G}_{\mathbb{R}}(\mathfrak{l}_X,d)$
 implies $\lambda$ is far from
 $Z(\mathfrak{l}_{\mathbb{R}})^*\setminus Z(\mathfrak{l}_{\mathbb{R}})^*_{\rm reg}$.
In view of the Langlands parameter of
 $\pi(\mathfrak{l}_{\mathbb{R}},\Gamma_{\lambda})$
 in Section~\ref{sec:quantization_semisimple}, we have
\begin{equation}\label{eq:disjoint}
\bigl(\widehat{G}_{\mathbb{R}}^{\mathfrak{l}_{\mathbb{R}}}\setminus
 \widehat{G}_{\mathbb{R}}(\mathfrak{l}_X,d)\bigr)
 \cap
\bigl(\widehat{G}_{\mathbb{R}}^{\mathfrak{l}'_{\mathbb{R}}}\setminus
 \widehat{G}_{\mathbb{R}}(\mathfrak{l}_X,d)\bigr)
 =\emptyset
\end{equation}
 if $\mathfrak{l}_{\mathbb{R}}$ and $\mathfrak{l}'_{\mathbb{R}}$
 are not $G_{\mathbb{R}}$-conjugate and if $d$ is sufficiently large.
We fix $d$ satisfying \eqref{eq:incl_support} and \eqref{eq:disjoint}.
Then we obtain the decomposition of $\op{supp} L^2(X_0)$:
\begin{equation*}
\op{supp} L^2(X_0)
= \bigl(\op{supp} L^2(X_0)\cap \widehat{G}_{\mathbb{R}}(\mathfrak{l}_X,d)\bigr)
 \sqcup \bigsqcup_{\mathfrak{l}_{\mathbb{R}}}
 \bigl( (\op{supp} L^2(X_0)\cap
 \widehat{G}_{\mathbb{R}}^{\mathfrak{l}_\mathbb{R}})
 \setminus \widehat{G}_{\mathbb{R}}(\mathfrak{l}_X,d) \bigr)
\end{equation*}
In this decomposition, we note that 
 $\op{supp} L^2(X_0)\cap \widehat{G}_{\mathbb{R}}(\mathfrak{l}_X,d)$
 is open in $\op{supp} L^2(X_0)$ and 
 $(\op{supp} L^2(X_0)\cap
 \widehat{G}_{\mathbb{R}}^{\mathfrak{l}_\mathbb{R}})
 \setminus \widehat{G}_{\mathbb{R}}(\mathfrak{l}_X,d)$
 is closed in $\op{supp} L^2(X_0)$.

Let
\[
L^2(X_0)\simeq \int_{\widehat{G}_{\mathbb{R}}}^{\oplus} \pi^{\oplus n(\pi)} dm
\]
be the irreducible decomposition.  Define 
\[
V'=\int_{\widehat{G}_{\mathbb{R}}(\mathfrak{l}_X,d)}^{\oplus}
 \pi^{\oplus n(\pi)} dm,\qquad
V_{\mathfrak{l}_{\mathbb{R}}}
=\int_{\widehat{G}_{\mathbb{R}}^{\mathfrak{l}_\mathbb{R}}
 \setminus \widehat{G}_{\mathbb{R}}(\mathfrak{l}_X,d)}^{\oplus}
 \pi^{\oplus n(\pi)} dm,
\]
and regard them as subrepresentations of $L^2(X_0)$ so we have
\[
L^2(X_0)=
 V'\oplus \bigoplus_{\mathfrak{l}_{\mathbb{R}}} V_{\mathfrak{l}_{\mathbb{R}}},\qquad
\op{WF}(L^2(X_0))= \op{WF}(V')\cup\bigcup_{\mathfrak{l}_{\mathbb{R}}}
\op{WF}(V_{\mathfrak{l}_{\mathbb{R}}}).
\]
By Lemma~\ref{lem:singular_spectrum},
 we have 
\[
\op{WF}(L^2(X_0))\cap G\cdot Z(\mathfrak{l}_X)^*_{\rm reg}
=\bigcup_{\mathfrak{l}_{\mathbb{R}}}
 \bigl(\op{WF}(V_{\mathfrak{l}_{\mathbb{R}}})\cap G\cdot Z(\mathfrak{l}_X)^*_{\rm reg}\bigr).
\]
Hence Theorem~\ref{thm:direct_integral} and Theorem~\ref{thm:wavefront} imply
\[
\overline{\sqrt{-1}\mu(T^*X_0)}\cap G\cdot Z(\mathfrak{l}_X)^*_{\rm reg}
=\bigcup_{\mathfrak{l}_{\mathbb{R}}}
\op{AC}
\Biggl(\bigcup_{\pi(\mathfrak{l}_{\mathbb{R}},\Gamma_{\lambda})
 \in \op{supp} V_{\mathfrak{l}_{\mathbb{R}}}}G_{\mathbb{R}}\cdot \lambda \Biggr)
\cap G\cdot Z(\mathfrak{l}_X)^*_{\rm reg}.
\]
Since  $(\op{supp} L^2(X_0)\cap
 \widehat{G}_{\mathbb{R}}^{\mathfrak{l}_\mathbb{R}})
 \setminus \widehat{G}_{\mathbb{R}}(\mathfrak{l}_X,d)$
 is closed in $\op{supp} L^2(X_0)$, we have
\[\op{supp} V_{\mathfrak{l}_{\mathbb{R}}}
= (\op{supp} L^2(X_0) \cap  \widehat{G}_{\mathbb{R}}^{\mathfrak{l}_{\mathbb{R}}})
 \setminus \widehat{G}_{\mathbb{R}}(\mathfrak{l}_X,d).\]
As in \eqref{eq:asymp_Xi}, we can easily show that  
\begin{align*}
\op{AC}
\Biggl(\bigcup_{\pi(\mathfrak{l}_{\mathbb{R}},\Gamma_{\lambda})
 \in \widehat{G}_{\mathbb{R}}(\mathfrak{l}_X,d)}
 G_{\mathbb{R}}\cdot \lambda \Biggr)
\cap G\cdot Z(\mathfrak{l}_X)^*_{\rm reg}=\emptyset.
\end{align*} 
Hence
\begin{align*}
&\bigcup_{\mathfrak{l}_{\mathbb{R}}}\op{AC}
\Biggl(\bigcup_{\pi(\mathfrak{l}_{\mathbb{R}},\Gamma_{\lambda})
 \in \op{supp} V_{\mathfrak{l}_{\mathbb{R}}}}G_{\mathbb{R}}\cdot \lambda \Biggr)
\cap G\cdot Z(\mathfrak{l}_X)^*_{\rm reg}\\
&=\op{AC}
\Biggl(\bigcup_{\pi(\mathfrak{l}_{\mathbb{R}},\Gamma_{\lambda})\in \op{supp} L^2(X_0)}
 G_{\mathbb{R}}\cdot \lambda \Biggr)
\cap G\cdot Z(\mathfrak{l}_X)^*_{\rm reg}.
\end{align*} 
Therefore, putting 
\[
S_{\mathfrak{l}_{\mathbb{R}}}:=
 \{\lambda\in\sqrt{-1}Z(\mathfrak{l}_{\mathbb{R}})^*_{\rm gr}
 \mid \exists \Gamma_{\lambda} \text{ such that }
 \pi(\mathfrak{l}_{\mathbb{R}},\Gamma_{\lambda})
 \in \op{supp} L^2(X_0)\},
\]
we have
\begin{equation}\label{eq:wavefront_decomp}
\overline{\sqrt{-1}\mu(T^*X_0)}\cap G\cdot Z(\mathfrak{l}_X)^*_{\rm reg}
=\bigcup_{\mathfrak{l}_{\mathbb{R}}}
\bigl(\op{AC}
(G_{\mathbb{R}}\cdot S_{\mathfrak{l}_{\mathbb{R}}})
\cap G\cdot Z(\mathfrak{l}_X)^*_{\rm reg}\bigr).
\end{equation}
This proves the equation
\[\overline{\sqrt{-1}\mu(T^*X_0)}\cap G\cdot Z(\mathfrak{l}_X)^*_{\rm reg}
=\op{AC}
\Biggl(\bigcup_{\pi(\mathcal{O},\Gamma) \in \op{supp} L^2(X_0)} \mathcal{O} \Biggr)
\cap G\cdot Z(\mathfrak{l}_X)^*_{\rm reg}\]
in Theorem~\ref{thm:main2}.
To show the remaining equation in Theorem~\ref{thm:main2},
 we replace \eqref{eq:incl_support} by 
\begin{equation*}
\op{supp} L^2(X_0)
\subset \widehat{G}_{\mathbb{R}}(\mathfrak{l}_X,d)
 \cup\bigcup_{\mathfrak{l}_{\mathbb{R}}}
 \{\pi(\mathfrak{l}_{\mathbb{R}},\Gamma_{\lambda})
 \in \widehat{G}_{\mathbb{R}}^{\mathfrak{l}_{\mathbb{R}}} \mid
 \lambda\in\mathfrak{a}_X^*\},
\end{equation*}
which was proved in Theorem~\ref{thm:reduction_semisimple}.
Then the same argument shows 
\[\overline{\sqrt{-1}\mu(T^*X_0)}\cap G\cdot Z(\mathfrak{l}_X)^*_{\rm reg}
=\op{AC}
\Biggl(\bigcup_{\substack{\pi(\mathcal{O},\Gamma) \in \op{supp} L^2(X_0)\\
 (G\cdot \mathcal{O})\cap \mathfrak{a}_X \neq \emptyset}} \mathcal{O} \Biggr)
\cap G\cdot Z(\mathfrak{l}_X)^*_{\rm reg}.\]
This completes the proof of Theorem~\ref{thm:main2}.

\bigskip

Next, we prove \eqref{eq:main} in Theorem~\ref{thm:main}.
Fix a Levi subalgebra $\mathfrak{l}_{\mathbb{R}}$ with $\mathfrak{l}\sim\mathfrak{l}_X$.
Taking the intersection of $\sqrt{-1}Z(\mathfrak{l}_{\mathbb{R}})^*_{\rm reg}$
 and \eqref{eq:wavefront_decomp}, we have
\begin{equation}\label{eq:wavefront_decomp2}
\overline{\sqrt{-1}\mu(T^*X_0)}\cap \sqrt{-1} Z(\mathfrak{l}_{\mathbb{R}})^*_{\rm reg}
=\bigcup_{\mathfrak{l}'_{\mathbb{R}}}
\op{AC}(G_{\mathbb{R}}\cdot S_{\mathfrak{l}'_{\mathbb{R}}})
\cap \sqrt{-1} Z(\mathfrak{l}_{\mathbb{R}})^*_{\rm reg}.
\end{equation}
If $\mathfrak{l}_{\mathbb{R}}=\mathfrak{l}'_{\mathbb{R}}$, then
\begin{equation*}
\op{AC}(G_{\mathbb{R}}\cdot S_{\mathfrak{l}_{\mathbb{R}}})
\cap \sqrt{-1} Z(\mathfrak{l}_{\mathbb{R}})^*_{\rm reg}
=\op{AC}(S_{\mathfrak{l}_{\mathbb{R}}})
\cap \sqrt{-1} Z(\mathfrak{l}_{\mathbb{R}})^*_{\rm reg}
\end{equation*}
by applying \eqref{eq:asymptotic_cone2}.
If $\mathfrak{l}_{\mathbb{R}}$ and $\mathfrak{l}'_{\mathbb{R}}$
 are not $G_{\mathbb{R}}$-conjugate,
 then \eqref{eq:asymptotic_cone} gives
\begin{equation*}
\op{AC}(G_{\mathbb{R}}\cdot S_{\mathfrak{l}'_{\mathbb{R}}})
\cap \sqrt{-1} Z(\mathfrak{l}_{\mathbb{R}})^*_{\rm reg}
=G_{\mathbb{R}}\cdot \bigl(\op{AC}(S_{\mathfrak{l}'_{\mathbb{R}}})
\cap \sqrt{-1} Z(\mathfrak{l}'_{\mathbb{R}})^*_{\rm reg}\bigr)
\cap \sqrt{-1} Z(\mathfrak{l}_{\mathbb{R}})^*_{\rm reg}
=\emptyset
\end{equation*}
because
 $G_{\mathbb{R}}\cdot \sqrt{-1} Z(\mathfrak{l}'_{\mathbb{R}})^*_{\rm reg}$
 does not intersect $\sqrt{-1} Z(\mathfrak{l}_{\mathbb{R}})^*_{\rm reg}$.
Therefore, the right hand side of \eqref{eq:wavefront_decomp2} equals
\begin{align*}
&\op{AC}(S_{\mathfrak{l}_{\mathbb{R}}})
\cap \sqrt{-1} Z(\mathfrak{l}_{\mathbb{R}})^*_{\rm reg} \\
&=\op{AC}\bigl(
\bigl\{\lambda\in \sqrt{-1}Z(\mathfrak{l}_{\mathbb{R}})^*_{\mathrm{gr}}
\mid \pi(\mathfrak{l}_{\mathbb{R}},\Gamma_{\lambda})\in \op{supp}L^2(X_0)\bigr\}\bigr)
\cap \sqrt{-1}Z(\mathfrak{l}_{\mathbb{R}})_{\mathrm{reg}}^*.
\end{align*}
This prove the second equation of \eqref{eq:main}.
The other equation of \eqref{eq:main}
 can be proved in the same way.

To prove the remaining assertion of Theorem~\ref{thm:main},
 we may replace $\mu(T^*X_0)$ by $\mu(T^* X_{\mathbb{R}})$,
 where $X_{\mathbb{R}}:=G_{\mathbb{R}}/H_{\mathbb{R}}$.
Indeed, if $\mathfrak{h}_{\mathbb{R}}^{\perp}:=
 \{\xi\in \mathfrak{g}_{\mathbb{R}}^*\mid \xi|_{\mathfrak{h}_{\mathbb{R}}}=0\}$,
 then $\mu(T^*X_0)=G_{\mathbb{R}}\cdot \mathfrak{h}_{\mathbb{R}}^{\perp}
 =\mu(T^* X_{\mathbb{R}})$.
The manifold $X_{\mathbb{R}}$ may not be an algebraic variety
 but a union of connected components of the $\mathbb{R}$-valued points of $G/H$.
We have $X_{\mathbb{R}}\subset X$ and for $x\in X_{\mathbb{R}}$
 there is a natural decomposition
 $T_x X=T_x X_{\mathbb{R}}\oplus \sqrt{-1}T_x X_{\mathbb{R}}$.
Hence there exists a natural inclusion $T^*X_{\mathbb{R}}\subset T^*X$.
Put $n:=\dim_{\mathbb{C}} \mu(T^* X)$.
By Theorem~\ref{thm:moment_image}, 
\[n=\dim_{\mathbb{C}} G\cdot \mathfrak{a}_X^*
 = \dim_{\mathbb{C}} \mathfrak{a}_X^* +\dim_{\mathbb{C}} \mathfrak{g}/\mathfrak{l}.\]
Define
\begin{align*}
&(T^*X)^o:=\{(x,\xi) \in T^*X \mid \xi\in G\cdot Z(\mathfrak{l}_X)^*_{\rm reg}
 \text{ and } \op{rank} d\mu_{(x,\xi)}=n\},\\
&(T^*X_{\mathbb{R}})^o:= T^*X_{\mathbb{R}}\cap (T^*X)^o.
\end{align*}
Then $(T^*X)^o$ is a Zariski open dense set in $T^*X$.
Therefore, $(T^*X_{\mathbb{R}})^o$ is open and dense in $T^*X_{\mathbb{R}}$.

Observe that
\begin{equation}\label{eq:regular_decomp}
\bigl(G\cdot Z(\mathfrak{l}_X)^*_{\rm reg}\bigr)
  \cap \mathfrak{g}_{\mathbb{R}}^*
=\bigsqcup_{\mathfrak{l}_{\mathbb{R}}}
 G_{\mathbb{R}}\cdot Z(\mathfrak{l}_{\mathbb{R}})^*_{\rm reg}.
\end{equation}
Here, as in \eqref{eq:incl_support}, $\mathfrak{l}_{\mathbb{R}}$ runs over 
 representatives of all $G_{\mathbb{R}}$-conjugacy classes
 of Levi subalgebras of $\mathfrak{g}_{\mathbb{R}}$
 such that $\mathfrak{l}\sim \mathfrak{l}_X$.
Indeed, if $\xi$ is in the left hand side of \eqref{eq:regular_decomp},
 then $\mathfrak{g}_{\mathbb{R}}(\xi)$ is $G_{\mathbb{R}}$-conjugate to
 exactly one of $\mathfrak{l}_{\mathbb{R}}$ in the right hand side
 of \eqref{eq:regular_decomp}.
Then $\xi\in Z(\mathfrak{l}_{\mathbb{R}})^*_{\rm reg}$
 for this $\mathfrak{l}_{\mathbb{R}}$.
Let $S=\sqrt{-1}Z(\mathfrak{l}_{\mathbb{R}})^*$ and
 apply Lemma~\ref{lem:asymptotic_cone}.
Since $G_{\mathbb{R}}\cdot S$ is a cone,
 $\op{AC}(G_{\mathbb{R}}\cdot S)= \overline{G_{\mathbb{R}}\cdot S}$.
Then \eqref{eq:asymptotic_cone} multiplied by $\sqrt{-1}$ becomes
\[
\overline{G_{\mathbb{R}}\cdot Z(\mathfrak{l}_{\mathbb{R}})^*}\cap
  \bigl(G\cdot Z(\mathfrak{l}_X)^*_{\rm reg}\bigr)
= G_{\mathbb{R}}\cdot Z(\mathfrak{l}_{\mathbb{R}})^*_{\rm reg}. 
\]
This shows each
 $G_{\mathbb{R}}\cdot Z(\mathfrak{l}_{\mathbb{R}})^*_{\rm reg}$
 is closed and hence also open in
 $\bigl(G\cdot Z(\mathfrak{l}_X)^*_{\rm reg}\bigr)
 \cap \mathfrak{g}_{\mathbb{R}}^*$.

Fix $\mathfrak{l}_{\mathbb{R}}$.
Suppose first that $\mu((T^*X_{\mathbb{R}})^o)$
 intersects $G_{\mathbb{R}}\cdot Z(\mathfrak{l}_{\mathbb{R}})^*$,
Then since the rank of $\mu$ equals $n$ everywhere
 on $T^*X_{\mathbb{R}}\cap (T^*X)^o$, we have 
\[\dim_{\mathbb{R}}
 \bigl(\mu((T^*X_{\mathbb{R}})^o)\cap
 G_{\mathbb{R}}\cdot Z(\mathfrak{l}_{\mathbb{R}})^*_{\rm reg}\bigr)=n.\]
By $\overline{\mu((T^*X_{\mathbb{R}})^o)}=\overline{\mu(T^*X_{\mathbb{R}})}$,
we have
 $\dim_{\mathbb{R}}
 \bigl(\overline{\mu(T^*X_{\mathbb{R}})}\cap
 G_{\mathbb{R}}\cdot Z(\mathfrak{l}_{\mathbb{R}})^*_{\rm reg}\bigr)=n.$
Since $\overline{\mu(T^*X_{\mathbb{R}})}$ is $G_{\mathbb{R}}$-stable,
 $\dim_{\mathbb{R}}
 \bigl(\overline{\mu(T^*X_{\mathbb{R}})}\cap
  Z(\mathfrak{l}_{\mathbb{R}})^*_{\rm reg}\bigr)=\dim_{\mathbb{C}}\mathfrak{a}_X^*$.
Hence \eqref{eq:moment} follows.

Suppose next that
 $\mu((T^*X_{\mathbb{R}})^o)\cap G_{\mathbb{R}}\cdot Z(\mathfrak{l}_{\mathbb{R}})^*
 =\emptyset$.
Then since $\overline{\mu((T^*X_{\mathbb{R}})^o)}=\overline{\mu(T^*X_{\mathbb{R}})}$
 and since $G_{\mathbb{R}}\cdot Z(\mathfrak{l}_{\mathbb{R}})^*_{\rm reg}$
 is open in \eqref{eq:regular_decomp}, we have 
\[\overline{\mu(T^*X_{\mathbb{R}})}\cap
  G_{\mathbb{R}}\cdot Z(\mathfrak{l}_{\mathbb{R}})^*_{\rm reg}
= \overline{\mu(T^*X_{\mathbb{R}})}\cap
  Z(\mathfrak{l}_{\mathbb{R}})^*_{\rm reg}=\emptyset.\]
Finally, as $\mu((T^*X_{\mathbb{R}})^o)$ is nonempty
 and contained in the set \eqref{eq:regular_decomp},
 $\overline{\mu(T^*X_{\mathbb{R}})}$ intersects
 $G_{\mathbb{R}}\cdot Z(\mathfrak{l}_{\mathbb{R}})^*_{\rm reg}$
 for at least one $\mathfrak{l}_{\mathbb{R}}$.
We finish the proof of Theorem~\ref{thm:main}.

\bigskip

Let us prove Theorem~\ref{thm:ds}.
There exists a local isomorphism between $T^*X_0$ and $T^*X_{\mathbb{R}}$
 so we may replace the assumption of Theorem~\ref{thm:ds} by 
\[\text{$\mu(T^*X_{\mathbb{R}})\cap
 (\mathfrak{g}_{\mathbb{R}}^*)_{\mathrm{ell}}$
 contains a nonempty open subset of $\mu(T^*X_{\mathbb{R}})$.}
\]
Let us assume this.
Take a nonempty open subset $U\subset T^*X_{\mathbb{R}}$ such that
 $\mu(U)\subset (\mathfrak{g}_{\mathbb{R}}^*)_{\mathrm{ell}}$.
Define $(T^*X)^o$ and $(T^*X_{\mathbb{R}})^o$
 as in the proof of Theorem~\ref{thm:main} above.
Since $(T^*X_{\mathbb{R}})^o$ is open and dense in $T^*X_{\mathbb{R}}$,
 we may assume $U\subset (T^*X_{\mathbb{R}})^o$.
Then by shrinking $U$ if necessary, we may further assume that
 $\mu(U)$ is a real submanifold of
 $ \bigl(G\cdot Z(\mathfrak{l}_X)^*_{\rm reg}\bigr)\cap \mathfrak{g}_{\mathbb{R}}^*$
 of dimension $n:=\dim_{\mathbb{C}} \mu(T^* X)$.
On the other hand, 
 $\mu(U)\subset \overline{G\cdot\mathfrak{a}_X^*}\cap \mathfrak{g}_{\mathbb{R}}^*$
 by Theorem~\ref{thm:moment_image}.
Since $(G\cdot\mathfrak{a}_X^*)\cap \mathfrak{g}_{\mathbb{R}}^*$
 is a semialgebraic set of (real) dimension $n$,
 we can find a vector $\xi(\neq 0) \in \mu(U)$
 and an open neighborhood $V$ of $\xi$ in $\mathfrak{g}_{\mathbb{R}}^*$
 such that 
\[V\cap \mu(U)
 =V\cap (G\cdot\mathfrak{a}_X^*)\cap \mathfrak{g}_{\mathbb{R}}^*.\]
By our assumption, the left hand side is contained
 in $(\mathfrak{g}_{\mathbb{R}}^*)_{\mathrm{ell}}$. 
If we put $\tilde{V}:=\sqrt{-1}\mathbb{R}_{>0}\cdot V$, then $\tilde{V}$ is an open cone
 containing $\sqrt{-1}\xi$ and 
\[\tilde{V}\cap (G\cdot\mathfrak{a}_X^*)\cap \sqrt{-1}\mathfrak{g}_{\mathbb{R}}^*
 \subset \sqrt{-1}(\mathfrak{g}_{\mathbb{R}}^*)_{\mathrm{ell}}.\] 
Since $\sqrt{-1}\xi\in
 \overline{\sqrt{-1}\mu(T^*X_0)}\cap (G\cdot Z(\mathfrak{l}_X)^*_{\mathrm{reg}})$,
Theorem~\ref{thm:main2} yields 
\[
\sqrt{-1}\xi\in \op{AC}\Biggl(
\bigcup_{\substack{\pi(\mathcal{O},\Gamma)\in \op{supp} L^2(X_0) \\
 (G\cdot \mathcal{O})\cap \mathfrak{a}_X^* \neq \emptyset}}
 \mathcal{O}\Biggr).
\]
Hence there exist infinitely many semisimple orbital parameter
 $(\mathcal{O}_j,\Gamma_j)$ and $\lambda_j\in \mathcal{O}_j$ such that
\begin{align*}
&\pi(\mathcal{O}_j,\Gamma_j)\in \op{supp} L^2(X_0),\quad 
 \frac{\lambda_j}{|\lambda_j|}\to \frac{\sqrt{-1}\xi}{|\sqrt{-1}\xi|}\ (j\to \infty) \text{ and }\\
&(G\cdot \mathcal{O}_j)\cap \mathfrak{a}_X^* \cap Z(\mathfrak{l}_X)^*_{\rm reg}\neq \emptyset.
\end{align*}
For large enough $j$, we have $\lambda_j\in\tilde{V}$ and then
 $\lambda_j\in\sqrt{-1}(\mathfrak{g}_{\mathbb{R}})_{\mathrm{ell}}$.
Moreover, it is easy to see that $\pi(\mathcal{O}_j,\Gamma_j)$ is an isolated point in
 the set
\[
\{\pi(\mathcal{O},\Gamma)\mid (G\cdot \mathcal{O})\cap \mathfrak{a}_X^*\neq\emptyset\}
\]
with respect to the Fell topology.
Hence it is an isolated point in $\op{supp} L^2(X_0)$.
As a consequence, $\pi(\mathcal{O}_j,\Gamma_j)$ appears in the discrete spectrum
 of the decomposition in $L^2(X_0)$ for large $j$
 and therefore $L^2(X_0)$ has infinitely many discrete series. 
This completes the proof of Theorem~\ref{thm:ds}.


\section{Examples}\label{sec:example}

\subsection{\texorpdfstring{$\op{GL}(n,\mathbb{R})/(\op{GL}(m,\mathbb{R})\times \op{GL}(k,\mathbb{Z}))$}{GL(n,R)/(GL(m,R)xGL(k,Z))}}\label{subsec:GL}

\

Let $X_0=\op{GL}(n,\mathbb{R})/(\op{GL}(m,\mathbb{R})\times \op{GL}(k,\mathbb{Z}))$ for $m+k\leq n$, where $\op{GL}(m,\mathbb{R})\times \op{GL}(k,\mathbb{Z})$ is embedded as a subgroup of $\op{GL}(n,\mathbb{R})$ in a standard way.
 Below, we compute the image of the real moment map $\mu(\sqrt{-1}T^*X_0)$ for every $n,m,k$. Combining this calculation with Theorem \ref{thm:main}, we obtain the asymptotic support of Plancherel measure for $X_0$.
The discrete group part $\op{GL}(k,\mathbb{Z})$ does not affect the moment map image or
 the asymptotic support of Plancherel measure.

\begin{proposition} 
Let $X_0=\op{GL}(n,\mathbb{R})/(\op{GL}(m,\mathbb{R})\times \op{GL}(k,\mathbb{Z}))$.
\begin{enumerate}[{\rm (i)}]
\item  \label{GLcase1} 
If $2m\leq n$, then $\mu(\sqrt{-1}T^*X_0)$ contains a Zariski open dense subset of $\sqrt{-1}\mathfrak{gl}(n,\mathbb{R})^*$. If $\mathfrak{j}_{\mathbb{R}}\subset \mathfrak{gl}(n,\mathbb{R})$ is a Cartan subalgebra, then 
\[\op{AC}\bigl(\bigl\{\lambda \in \sqrt{-1}(\mathfrak{j}_{\mathbb{R}})_{\op{reg}}^*\mid
 \pi(\mathfrak{j}_{\mathbb{R}},\Gamma_{\lambda})\in \op{supp}L^2(X_0)\bigr\}\bigr)
=\sqrt{-1}\mathfrak{j}_{\mathbb{R}}^*.\]
In particular, $\operatorname{supp}L^2(X_0)$ ``asymptotically contains the entire tempered dual of $\op{GL}(n,\mathbb{R})$''.
\item \label{GLcase2} 
If $2m> n$, form the Levi subgroup 
\[L:= \op{GL}(1,\mathbb{C})^{\times (2n-2m)} \times \op{GL}(2m-n,\mathbb{C})\]
with Lie algebra $\mathfrak{l}$. Let $\mathfrak{l}_{\mathbb{R}}\subset \mathfrak{l}$ be a real form contained in $\mathfrak{gl}(n,\mathbb{R})$, and identify $\sqrt{-1}Z(\mathfrak{l}_{\mathbb{R}})^*\simeq Z(\mathfrak{l}_{\mathbb{R}})$ by dividing by $\sqrt{-1}$ and using the trace form. Let $Z(\mathfrak{l}_{\mathbb{R}})_0$ denote the set of matrices $X_0\in Z(\mathfrak{l}_{\mathbb{R}})$ with 
\[\operatorname{rank}X_0\leq 2n-2m,\] 
and let 
\[\sqrt{-1}Z(\mathfrak{l}_{\mathbb{R}})^*_{0,\op{reg}}\subset 
\sqrt{-1}Z(\mathfrak{l}_{\mathbb{R}})^*\] 
denote the set of regular elements in the corresponding subset of 
$\sqrt{-1}Z(\mathfrak{l}_{\mathbb{R}})^*$. 
Then 
\begin{multline*}
\op{AC}\bigl(\bigl\{ \lambda \in \sqrt{-1}Z(\mathfrak{l}_{\mathbb{R}})_{\op{gr}}^*\mid
 \pi(\mathfrak{l}_{\mathbb{R}},\Gamma_{\lambda})\in \op{supp}L^2(X_0)\bigr\}\bigr)
 \cap \sqrt{-1}Z(\mathfrak{l}_{\mathbb{R}})^*_{\op{reg}} \\
 =\sqrt{-1}Z(\mathfrak{l}_{\mathbb{R}})^*_{0,\op{reg}}.
\end{multline*}
\end{enumerate}
\end{proposition}

\begin{remark}
In the case \eqref{GLcase2}, real forms of $L$ are of the form 
\[L_{\mathbb{R}}^s= 
 \op{GL}(1,\mathbb{C})^{\times s}\times 
 \op{GL}(1,\mathbb{R})^{\times 2(n-m-s)}\times 
 \op{GL}(2m-n,\mathbb{R})\]
with $0\leq s\leq n-m$. 
For fixed $s$, we may form the larger real Levi subgroup
\[\widetilde{L_{\mathbb{R}}}^s=
 \op{GL}(2,\mathbb{R})^{\times s} \times 
  \op{GL}(1,\mathbb{R})^{\times 2(n-m-s)}\times \op{GL}(2m-n,\mathbb{R}).\]
Take a representation of the form 
\begin{equation}\label{eq:L'GL} 
\sigma_{1}\boxtimes \cdots \boxtimes \sigma_{s}\boxtimes
 \tau_{1}\boxtimes \cdots \boxtimes \tau_{2(n-m-s)} \boxtimes \tau_{\nu}
\end{equation}
where $\tau_{i}$ $(1\leq i\leq 2(n-m-s))$
 and $\tau_{\nu}$ are one-dimensional unitary representations
 and $\sigma_i$ are relative discrete series representations. 
If $P_{\mathbb{R}}^s$ is a real parabolic with Levi factor
 $\widetilde{L}_{\mathbb{R}}^s$,
 then the representations $\pi(\mathfrak{l}_{\mathbb{R}}^s,\Gamma_{\lambda})$
 with $\lambda\in Z(\mathfrak{l}_{\mathbb{R}}^{s})^{*}_{\op{gr}}$ are obtained
 by unitary parabolic induction 
 from $P_{\mathbb{R}}^s$-representations of the form \eqref{eq:L'GL}
 to $\op{GL}(n,\mathbb{R})$.

When $2m-n>1$, the condition
 $\lambda\in \sqrt{-1}Z(\mathfrak{l}_{\mathbb{R}}^s)^*_{0,\op{reg}}$
 implies that $\lambda$ vanishes on the last component
 $\mathfrak{gl}(2m-n,\mathbb{R})$ of $\mathfrak{l}_{\mathbb{R}}^s$
 and hence $\tau_{\nu}$ is trivial on the identity component of
 $\op{GL}(2m-n,\mathbb{R})$.

We remark that when $k=0$, according to a result of Benoist-Kobayashi~\cite{BK15},
 $L^2(X_0)$ is tempered if and only if $2m\leq n+1$.
\end{remark}

\begin{proof} First, we prove part \eqref{GLcase1}. Assume $n$ even, and put $p:=\frac{n}{2}$. Consider the set $\mathcal{F}_p$ consisting of all matrices of the following form:
\[ A=
\begin{pmatrix} a_1 & 0 & \cdots & 0 & b_1 & 0 & \cdots & 0\\
0 & a_2 & \cdots & 0 & 0 & b_2 & \cdots & 0\\ 
\vdots & \vdots & \ddots & \vdots & \vdots & \vdots & \ddots & \vdots\\
0 & 0 & \cdots & a_p & 0 & 0 & \cdots & b_p\\
1 & 0 & \cdots & 0 & 0 & 0 & \cdots & 0\\
0 & 1 & \cdots & 0 & 0 & 0 & \cdots & 0\\ 
\vdots & \vdots & \ddots & \vdots & \vdots & \vdots & \ddots & \vdots\\
0 & 0 & \cdots & 1 & 0 & 0 & \cdots & 0 \end{pmatrix}.\]
Each matrix $A$ has $s$ submatrices of the form
\[A_j=\begin{pmatrix} a_j & b_j\\ 1 & 0\end{pmatrix}.\]
We note that the $2n$ eigenvalues of a matrix $A\in \mathcal{F}_p$ is simply the union of the eigenvalues of these $n$ two by two matrices $A_j$. Now, for fixed eigenvalues $\lambda_1$, $\lambda_2$ with either (a) $\lambda_1$, $\lambda_2$ both real or (b) $\overline{\lambda_1}=\lambda_2$, we may choose $a_j$ and $b_j$ such that $A_j$ has the eigenvalues $\lambda_1$, $\lambda_2$ by setting $a_j:=\lambda_1+\lambda_2$ and $b_j:=-\lambda_1\lambda_2$. After identifying $\sqrt{-1}\mathfrak{gl}(n,\mathbb{R})^*\simeq \mathfrak{gl}(n,\mathbb{R})$, notice that all of the above matrices $A\in \mathcal{F}_p$ lie in $\mathfrak{gl}(m,\mathbb{R})^{\perp}\subset \mu(\sqrt{-1}L^2(X_0))$. Since two $\mathbb{C}$-diagonalizable matrices in $\mathfrak{gl}(n,\mathbb{R})$ are $\op{GL}(n,\mathbb{R})$-conjugate if, and only if they have the same eigenvalues, part \eqref{GLcase1} follows in the case $n$ even. When $n$ odd, add an extra row and column to every $A\in \mathcal{F}_p$ with all zeroes except for the desired $n$th (real) eigenvalue in the diagonal entry. 

For part \eqref{GLcase2}, put $p:=n-m$ and note $0\leq p< n/2$. 
Take complex numbers $\{\lambda_{11},\lambda_{12},\lambda_{21},\lambda_{22},\ldots,\lambda_{p1},\lambda_{p2}\}$  such that either (a) $\lambda_{k1}$, $\lambda_{k2}$ both real or (b) $\overline{\lambda}_{k1}=\lambda_{k2}$. As before, we can find a $2p$ by $2p$ matrix in $\mathcal{F}_p$ with these specified eigenvalues.
Adding extra rows and columns to get a $2n$ by $2n$ matrix $A$.
When the $2p$ eigenvalues are distinct,
 we see that the stabilizer of $A$ for 
 the adjoint action of $\operatorname{GL}(n,\mathbb{C})$ is isomorphic to
\[L:=\op{GL}(1,\mathbb{C})^{\times (2n-2m)}\times \op{GL}(2m-n,\mathbb{C}).\]
Further, we see that for every real form $\mathfrak{l}_{\mathbb{R}}$ of $\mathfrak{l}$ (as described in the remark above),
 every element of $Z(\mathfrak{l}_{\mathbb{R}})_0$
 is a conjugate of a matrix of the form $A$ as above.

Next, recall $\mu(\sqrt{-1}T^*X_0) = \op{Ad}^*(\op{GL}(n,\mathbb{R}))\cdot \mathfrak{gl}(m,\mathbb{R})^{\perp}$. We see that every $B\in \mathfrak{gl}(m,\mathbb{R})^{\perp}\subset \mathfrak{gl}(n,\mathbb{R})$ with zeroes in an $m\times m$ block in the bottom right is a sum of a matrix $B_1$ with $m$ zero columns and a matrix $B_2$ with $m$ zero rows. In particular, $\operatorname{rank} B\leq 2n-2m$. It follows that $\mathfrak{l}$ is the Levi subalgebra $\mathfrak{l}_X$ in Theorem~\ref{thm:moment_image} for $X=\op{GL}(2n)/\op{GL}(2m)$ and that the closure of the conjugates of the matrices of the form $B$ intersected with $Z(\mathfrak{l}_{\mathbb{R}})$ constitute $Z(\mathfrak{l}_{\mathbb{R}})_0$. Part \eqref{GLcase2} follows.
\end{proof}

\subsection{\texorpdfstring{$\op{Sp}(2n,\mathbb{R})/(\op{Sp}(2m,\mathbb{R})\times \op{Sp}(2k,\mathbb{Z}))$}{Sp(2n,R)/(Sp(2m,R)xSp(2k,R))}}\label{subsec:Sp}

\ 

Similarly to the previous subsection, 
 we calculate the moment map image 
 for $X_0=\op{Sp}(2n,\mathbb{R})/(\op{Sp}(2m,\mathbb{R})\times \op{Sp}(2k,\mathbb{Z}))$ with $m+k\leq n$, where $\op{Sp}(2m,\mathbb{R})\times \op{Sp}(2k,\mathbb{Z})$
 is embedded as a subgroup of $\op{Sp}(2n,\mathbb{R})$ in a standard way.

Let $G_\mathbb{R}=\op{Sp}(2n,\mathbb{R})$ and $H_\mathbb{R}=\op{Sp}(2m,\mathbb{R})$.
Let $V=\mathbb{R}^{2n}$ with a symplectic form $(\cdot,\cdot)$.
Then we identify $G_{\mathbb{R}}$ with the automorphism group of $(V, (\cdot,\cdot))$.
The Lie algebra $\mathfrak{g}_{\mathbb{R}}$ consists of $A\in \mathfrak{gl}(V)$ satisfying
\[
\langle A v_1, v_2\rangle + \langle v_1, A v_2\rangle = 0.
\]
For $A\in \mathfrak{g}_{\mathbb{R}}$, define a bilinear form $(\cdot, \cdot)_{A}$ on $V$ by 
\[(v_1,v_2)_{A}:=\langle A v_1, v_2 \rangle.\]
This form is symmetric and hence its signature $(p,q) = \op{sign}(\cdot,\cdot)_{A}$
 is defined.
Write $\op{sign}(A) := \op{sign}(\cdot,\cdot)_{A}$.

Let $V=W\oplus W'$ be an orthogonal decomposition into symplectic vector spaces
 with $\dim W=2m$.
Let 
\[\mathfrak{h}_{\mathbb{R}}:=\{A\in \mathfrak{g}_{\mathbb{R}}\mid A(W)\subset W,\ A(W')=0\}\simeq
 \mathfrak{sp}(2m,\mathbb{R}).\]
Then 
\[\mathfrak{h}_{\mathbb{R}}^{\perp}=\{A\in \mathfrak{g}_{\mathbb{R}}\mid \langle A(W), W\rangle=0\}.\]
Here and in what follows, we identify $\mathfrak{g}_{\mathbb{R}}$ with
 $\mathfrak{g}_{\mathbb{R}}^*$ by an invariant form.

\begin{lemma}\label{lem:Sp}
Let $A\in \mathfrak{g}_{\mathbb{R}}$. 
Then $A\in G_{\mathbb{R}}\cdot \mathfrak{h}^{\perp}_{\mathbb{R}}$ if and only if there exists
 a $2m$-dimensional subspace $W_1 \subset V$
 such that $\langle \cdot, \cdot \rangle|_{W_1}$ is nondegenerate
 and $(\cdot, \cdot)_A |_{W_1} = 0$.
\end{lemma}

\begin{proof}
If $A\in g\cdot \mathfrak{h}_{\R}^{\perp}$, then $W_1=g\cdot W$ satisfies the condition.

Conversely, suppose $W_1$ satisfies the condition in the lemma.
Then standard symplectic bases of $W_1$ and $W$ can
 be extended to a standard symplectic basis of $V$, respectively.
Hence we can find $g\in G_{\mathbb{R}}$ such that $g\cdot W = W_1$ and
 then we have $A\in g\cdot \mathfrak{h}_{\R}^{\perp}$. 
\end{proof}

For semisimple $A$, this condition is characterized by $\op{sign}(A)$.
\begin{lemma}\label{lem:Sp2}
Suppose that $A\in \mathfrak{g}_{\mathbb{R}}$ is semisimple
 and let $\op{sign}(A)=(p,q)$.
Then the following two conditions are equivalent.
\begin{enumerate}
\item \label{ExistSubsp}
There exists a $2m$-dimensional subspace $W_1 \subset V$
 such that $\langle \cdot, \cdot \rangle|_{W_1}$ is nondegenerate
 and $(\cdot, \cdot)_A |_{W_1} = 0$.
\item \label{SignCond}
 $\max\{p,q\}\leq 2n-2m$.
\end{enumerate}
\end{lemma}

\begin{proof}
It is easy to see that the maximal isotropic subspace
 of $V$ with respect to the symmetric form $(\cdot,\cdot)_{A}$, which has signature $(p,q)$,
 is $2n-\max\{p,q\}$.
Hence \eqref{ExistSubsp} implies \eqref{SignCond}.

We now prove the other implication.
Since $V = \op{Im} (A) \oplus \op{Ker} (A)$, 
 by considering $A|_{\op{Im} (A)}$, our claim is reduced to the case when $\op{Im}(A)=V$.
Thus we assume $\op{rank} A = 2n$.

Since $A$ is semisimple,
 we can find an orthogonal decomposition $V=\bigoplus_i V_i$ as a symplectic vector space
 such that $A(V_i)=V_i$ and $\dim V_i=2$ or $4$. 
This follows from the classification of Cartan subalgebras of
 $\mathfrak{sp}(2n,\mathbb{R})$.
See \cite[\S 3, Type (CI)]{Sug59} for such a classification result.
Let $A_i:=A|_{V_i}$ so that $A_i$ is regarded as an element in
 $\mathfrak{sp}(2,\mathbb{R})$ or $\mathfrak{sp}(4,\mathbb{R})$.

When $\dim V_i = 4$, we may assume that it cannot decompose into
 two $A_i$-stable $2$-dimensional symplectic vector spaces.
Then $\op{sign}(A_i)=(2,2)$.
In this case, there exists a $2$-dimensional subspace $W_i\subset V_i$
 such that $\langle\cdot,\cdot\rangle$ is nondegenerate and
 $(\cdot,\cdot)_A=0$ on $W_i$.

When $\dim V_i=\dim V_{i'}=2$ and
 $\op{sign}(A_i)=\op{sign}(A_{i'})=(1,1)$ with $i\neq i'$,
 there exists a $2$-dimensional subspace $W_i\subset V_i\oplus V_{i'}$
 such that $\langle\cdot,\cdot\rangle$ is nondegenerate and
 $(\cdot,\cdot)_A=0$ on $W_i$.

Similarly, 
 when $\dim V_i=\dim V_{i'}=2$, 
 $\op{sign}(A_{i})=(2,0)$ and
 $\op{sign}(A_{i'})=(0,2)$,
 there exists a $2$-dimensional subspace $W_i\subset V_i\oplus V_{i'}$
 satisfying the same conditions.

Making appropriate pairs among $V_i$ and 
 taking sum of above $W_i$, we obtain $W$ in \eqref{ExistSubsp}.
\end{proof}

For complex Lie algebras $\mathfrak{g}\supset \mathfrak{h}$
 analogues of  Lemmas \ref{lem:Sp} and \ref{lem:Sp2} are proved in a similar and easier way.
We have for a semisimple element $A\in \mathfrak{g}$ 
\begin{equation}\label{eq:SpCond}
A\in G\cdot\mathfrak{h}^{\perp} \Leftrightarrow \op{rank} A \leq 4n-4m,
\end{equation}
where $\op{rank} A$ is the rank of $A$ viewed as a $2n$ by $2n$ matrix with complex entries.

For $0\leq r\leq n$, let
\[L^r:=\op{GL}(1,\mathbb{C})^{\times r}\times \op{Sp}(2(n-r),\mathbb{C}),\]
 the Levi subgroup of $\op{Sp}(2n,\mathbb{C})$.
By \eqref{eq:SpCond},
 the Levi subalgebra $\mathfrak{l}_X$
 in Theorem~\ref{thm:moment_image} for $X=G/H$ is 
 a Cartan algebra if $2m\leq n$;
 and $\mathfrak{l}^{2(n-m)}$ if $2m>n$.
For $s,t,u\geq 0$ with $s+2t+u\leq n$, let 
\[L_{\mathbb{R}}^{s,t,u}
= \op{U}(1)^{\times s}\times \op{GL}(1,\mathbb{C})^{\times t}\times
 \op{GL}(1,\mathbb{R})^{\times u} \times \op{Sp}(2(n-s-2t-u),\mathbb{R}).\]
Then $L_{\mathbb{R}}^{s,t,u}$ with $s+2t+u=r$
 are all the real Levi subgroups of
 $\op{Sp}(2n,\mathbb{R})$ whose complexifications are conjugate to $L^r$.
In particular, $L_{\mathbb{R}}^{s,t,u}$ for $s+2t+u=n$ are
 all the Cartan subgroups of $\op{Sp}(2n,\mathbb{R})$
 up to conjugation.

For fixed $s,t,u$, we may form the larger real Levi subgroup
\[\widetilde{L}_{\mathbb{R}}^{s,t,u}
= \op{GL}(2,\mathbb{R})^{\times t} \times 
  \op{GL}(1,\mathbb{R})^{\times u}\times \op{Sp}(2(n-2t-u),\mathbb{R}).\]
Take a representation of the form 
\begin{equation}\label{eq:L'Sp} 
\sigma_{1}\boxtimes \cdots \boxtimes \sigma_{t}
 \boxtimes \tau_{1}\boxtimes \cdots \boxtimes \tau_{2(n-m-s)}
 \boxtimes \kappa
\end{equation}
where $\tau_{i}$ are one-dimensional unitary representations,
 $\sigma_{i}$ are relative discrete series representations,
 and $\kappa$ is (a Hilbert completion of) $A_{\mathfrak{q}}(\lambda)$
 such that the Levi factor of $\mathfrak{q}$ is the complexification
 of $\op{U}(1)^{\times s} \times \op{Sp}(2(n-s-2t-u),\mathbb{R})$. 
If $P_{\mathbb{R}}^{s,t,u}$ is a real parabolic with Levi factor
 $\widetilde{L}_{\mathbb{R}}^{s,t,u}$,
 then the representations $\pi(\mathfrak{l}_{\mathbb{R}}^{s,t,u},\Gamma_{\lambda})$
 with $\lambda\in \sqrt{-1}Z(\mathfrak{l}_{\mathbb{R}}^{s,t,u})^{*}_{\op{gr}}$ are obtained
 by unitary parabolic induction 
 from $P_{\mathbb{R}}^{s,t,u}$-representations of the form \eqref{eq:L'Sp}
 to $\op{Sp}(2n,\mathbb{R})$.

Let $\lambda\in \sqrt{-1}Z(\mathfrak{l}_{\mathbb{R}}^{s,t,u})^*$.
It has $s$ parameters corresponding
 to the first component $\op{U}(1)^{\times s}$, 
 which we denote by $(a_1,\dots,a_s)\in \mathbb{R}^s$.
If $\lambda\in \sqrt{-1}Z(\mathfrak{l}_{\mathbb{R}}^{s,t,u})^*_{\op{reg}}$,
 then $a_1,\dots,a_s$ are nonzero; and 
 if one has a representation
 $\pi(\mathfrak{l}_{\mathbb{R}}^{s,t,u}, \Gamma_{\lambda})$,
 then $a_1,\dots,a_s$ are integers.
For nonnegative integers $s_1,s_2$, write 
\begin{align*}
&\sqrt{-1}Z(\mathfrak{l}_{\mathbb{R}}^{(s_1,s_2),t,u})^* \\
&:=\bigl\{\lambda\in \sqrt{-1}Z(\mathfrak{l}_{\mathbb{R}}^{s,t,u})^*\mid \#\{i \mid a_i>0\}=s_1 \text{ and } \#\{i \mid a_i<0\}=s_2\bigr\}
\end{align*}

Suppose that among $s$ parameters $(a_1,\dots,a_s)$,
 $s_1$ of them are positive and $s_2$ of them are negative.
If we regard
 $\sqrt{-1}\lambda \in Z(\mathfrak{l}_{\mathbb{R}}^{(s_1,s_2),t,u})^*
 \subset\mathfrak{g}_{\mathbb{R}}^*$
 as an element in $\mathfrak{g}_{\mathbb{R}}$,
 the signature of $\sqrt{-1}\lambda$
 defined above is $(2s_1+2t+u,2s_2+2t+u)$
 when we suitably fix a parameterization of characters of $\op{U}(1)$.
We have a decomposition 
\[
\sqrt{-1}Z(\mathfrak{l}_{\mathbb{R}}^{s,t,u})^*_{\op{reg}}
=\bigcup_{s_1+s_2=s}
\sqrt{-1}Z(\mathfrak{l}_{\mathbb{R}}^{(s_1,s_2),t,u})^*
\cap \sqrt{-1}Z(\mathfrak{l}_{\mathbb{R}}^{s,t,u})^*_{\op{reg}}.
\]

Summing up above arguments and by Theorem~\ref{thm:main}, we have the following.

\begin{proposition} 
Let $X_0=\op{Sp}(2n,\mathbb{R})/(\op{Sp}(2m,\mathbb{R})\times \op{Sp}(2k,\mathbb{Z}))$.
\begin{enumerate}[{\rm (i)}]
\item If $2m\leq n$, then $\mu(\sqrt{-1}T^*X_0)$ intersects 
 the set of regular semisimple elements in
 $\sqrt{-1}\mathfrak{g}_{\mathbb{R}}^*$. 
Take a Cartan subalgebra
\[\mathfrak{j}_{\mathbb{R}}=\mathfrak{l}^{s,t,u}_{\mathbb{R}}
= \mathfrak{u}(1)^{\oplus s} \oplus
 \mathfrak{gl}(1,\mathbb{C})^{\oplus t}\oplus
 \mathfrak{gl}(1,\mathbb{R})^{\oplus u},\]
 where $s+2t+u=n$.
Then 
\begin{multline*}
\op{AC}\bigl( \bigl\{\lambda \in
  \sqrt{-1}(\mathfrak{j}_{\mathbb{R}})_{\op{reg}}^* \mid  
 \pi(\mathfrak{j}_{\mathbb{R}},\Gamma_{\lambda})\in 
 \op{supp}L^2(X_0) \bigr\} \bigr)
 \cap \sqrt{-1}(\mathfrak{j}_{\mathbb{R}})_{\op{reg}}^* \\
= \bigcup_{s_1}
 \sqrt{-1}(\mathfrak{l}^{(s_1,s-s_1),t,u}_{\mathbb{R}}
 )^*_{\op{reg}},
\end{multline*}
where $s_1$ runs over nonnegative integers satisfying
\[\frac{2m-n+s}{2}\leq s_1\leq \frac{n-2m+s}{2}.\]
\item If $2m> n$, then take a Levi subalgebra 
 \[\mathfrak{l}_{\mathbb{R}}^{s,t,u}
= \mathfrak{u}(1)^{\oplus s} \oplus
 \mathfrak{gl}(1,\mathbb{C})^{\oplus t}\oplus
 \mathfrak{gl}(1,\mathbb{R})^{\oplus u}\oplus 
 \mathfrak{sp}(2(n-s-2t-u),\mathbb{R})\]
for nonnegative integers $s,t,u$ such that $s+2t+u=2(n-m)$.
Then 
\[\op{AC}\bigl( \bigl\{ \lambda \in
  \sqrt{-1}(\mathfrak{l}_{\mathbb{R}}^{s,t,u})_{\op{reg}}^* \mid 
 \pi(\mathfrak{l}_{\mathbb{R}}^{s,t,u},\Gamma_{\lambda})\in 
 \op{supp}L^2(X_0) \bigr\} \bigr)
 \cap \sqrt{-1}(\mathfrak{l}_{\mathbb{R}}^{s,t,u})_{\op{reg}}^*\]
equals 
$\sqrt{-1}(\mathfrak{l}^{(\frac{s}{2},\frac{s}{2}),t,u}_{\mathbb{R}})^*_{\op{reg}}$
if $s$ is even; and empty if $s$ is odd.
\end{enumerate}
\end{proposition}

Note that when $k=0$,  $L^2(X_0)$ is tempered if and only if $2m\leq n$ by \cite{BK15}.

We now deduce which elliptic orbits appear in the image of moment map. 
Let $\mathfrak{t}$ be a Cartan subalgebra of $K(\simeq \op{GL}(n,\mathbb{C}))$
 and let $\epsilon_1,\dots,\epsilon_n$ be a standard basis of $\mathfrak{t}^*$.
The roots in $\mathfrak{k}$ and $\mathfrak{g}$ are as follows
\begin{align*}
&\Delta(\mathfrak{k},\mathfrak{t})=\{\epsilon_i-\epsilon_j:1\leq i,j\leq n,\ i\neq j\},\\
&\Delta(\mathfrak{g},\mathfrak{t})=\{\pm 2\epsilon_i:1\leq i\leq n\}\cup
\{\pm\epsilon_i\pm \epsilon_j:1\leq i,j\leq n,\ i\neq j\}.
\end{align*}

Suppose first that $n\geq 2m$.
This case was previously studied in \cite[Example 7.5]{HW17}.
Then the moment map image
 $\mu(T^*X_0)$ contains a regular semisimple orbit of $\mathfrak{g}_{\R}^*$.
Suppose $A$ is regular so that $\op{sign}(A)=(p,2n-p)$ for some $p$.
By Lemma~\ref{lem:Sp} and Lemma~\ref{lem:Sp2},
 $A\in \mu(T^*X_0)$ if and only if
 $2m\leq p\leq 2n-2m$.
If $n=2m$, then $\op{sign}(A)=(n,n)$ is the only possibility.
The Harish-Chandra parameters for
 discrete series of $\op{Sp}(2n,\mathbb{R})$ are given
 in terms of standard coordinates as follows:
\begin{equation} \label{eq:Sp.HCpar}
 \sum_{i=1}^{n} a_{i}\epsilon_{i}
 \text{ with } a_i\in \mathbb{Z}\text{ and }|a_1| > |a_2| > \cdots > |a_n|>0.
\end{equation}
If $(p,q)$ is the signature for the corresponding orbit,
 then $p$ is the number of positives in $\{a_1,\dots,a_n\}$
 and $q$ is the number of negatives.
As a consequence of Theorem \ref{thm:ds},
 for any given subset $S$ of $\{1,2,\dots,n\}$ with $2m\leq \# S\leq 2n-2m$,
 there exist infinitely many distinct discrete series representations of $\op{Sp}(2n,\mathbb{R})$
 which are isomorphic to subrepresentations of $L^2(X_0)$
 and has the Harish-Chandra parameters as \eqref{eq:Sp.HCpar} satisfying $\{i:a_i>0\}=S$.

Suppose next that $n < 2m$.
Then the maximal rank of $A$ is $4n-4m$.
If $\op{rank} A = 4n-4m$, then 
 $A\in \mu(T^*X_0)$ if and only if $(p,q)=(2n-2m,2n-2m)$.
Let $S$ be a subset of $\{1,\dots,2n-2m\}$ such that $\# S=n-m$.
Let $S':=\{1,\dots,2n-2m\}\setminus S$.
Let $\mathfrak{q}_S$ be a parabolic subalgebra of $\mathfrak{g}$ such that 
 the roots of its nilradical $\mathfrak{n}_S$ are 
\begin{align*}
\Delta(\mathfrak{n}_S,\mathfrak{t})
&=\{\epsilon_i\pm \epsilon_j: i\in S,\ i < j\}\cup\{2\epsilon_i:i\in S\} \\
&\cup 
\{-\epsilon_i\pm \epsilon_j: i\in S',\ i < j\}\cup\{-2\epsilon_i:i\in S'\}.
\end{align*}
The real Levi factor for $\mathfrak{q}$
 is isomorphic to $\mathfrak{u}(1)^{\oplus (2n-2m)}\oplus \mathfrak{sp}(4m-2n,\mathbb{R})$.
The elliptic coadjoint orbits with signature $(p,q)=(2n-2m,2n-2m)$
 correspond to $A_{\mathfrak{q}_S}(\lambda)$ for some $S$ as above.
Therefore, for any given $S\subset \{1,\dots,2n-2m\}$ with $\# S=n-m$, 
 there exist infinitely many parameters $\lambda$ in the good range such that 
 (Hilbert completions of) $A_{\mathfrak{q}_S}(\lambda)$ occurs
 as a discrete spectrum of $L^2(X_0)$.

In particular, we have
\begin{corollary}
$\op{Sp}(2n,\mathbb{R})/(\op{Sp}(2m,\mathbb{R})\times \op{Sp}(2k,\mathbb{Z}))$
 has discrete series for any $n,m,k$ with $m+k\leq n$.
\end{corollary}


\bibliographystyle{amsalpha}
\bibliography{UniversalBibtex}

\end{document}